\documentclass[12pt]{report}
\usepackage[leqno]{amsmath}
\usepackage{amsfonts}
\usepackage{amssymb}
\usepackage{amsthm}
\usepackage[all]{xy}
\usepackage{hyperref}
\usepackage{a4wide}
\usepackage{graphicx}

\begin{document}

\fontsize{12pt}{18pt}\selectfont
\newtheorem{thmint}{Theorem}
\renewcommand{\thethmint}{\Alph{thmint}}
\newtheorem{thmintmine}{Theorem}
\renewcommand{\thethmintmine}{\Roman{thmintmine}}

%----------------
\newcommand{\scsltitle}[2]{\par\noindent\textsc{#1}\ \textsl{#2}}
%enumerate style
\renewcommand{\theenumi}{\alph{enumi}}
\renewcommand{\labelenumi}{\textnormal{(\theenumi)}}

\newtheorem{lemma}{Lemma}[section]
\newtheorem{theorem}[lemma]{Theorem}
\newtheorem*{mainthm}{Main Theorem}
\newtheorem{corollary}[lemma]{Corollary}
\newtheorem{fact}[lemma]{Fact}
\newtheorem{proposition}[lemma]{Proposition}
\newtheorem{observation}[lemma]{Observation}
\theoremstyle{remark}
\newtheorem{claim}[lemma]{Claim}
\newtheorem{remark}[lemma]{Remark}
\newtheorem{remarks}[lemma]{Remarks}
\newtheorem*{rem}{Remark}
\newtheorem{example}[lemma]{Example}
\newtheorem{examples}[lemma]{Examples}
\newtheorem{idea}[lemma]{Idea}
\newtheorem{construction}[lemma]{Construction}
\newtheorem{exercise}[lemma]{Exercise}
\newtheorem{question}[lemma]{Question}
\newtheorem{problem}[lemma]{Problem}
\newtheorem{conjecture}[lemma]{Conjecture}
\newtheorem*{que}{Question}
\newtheorem*{openprob}{Open Problem}
\theoremstyle{definition}
\newtheorem{condition}[lemma]{Condition}
\newtheorem{definition}[lemma]{Definition}
\newenvironment{label1}{
\begin{enumerate}
\renewcommand{\theenumi}{\arabic{section}.\arabic{lemma}\alph{enumi}}
}{\end{enumerate}\renewcommand{\labelenumi}{(\alph{enumi})}}{}
%enumeration of equations by \equation\Lcounter 1a, 1b, 1c etc...
\newcounter{Lcounter}
\renewcommand{\theLcounter}{\theequation\alph{Lcounter}}

\newenvironment{labellist}{
    \refstepcounter{equation}
    \begin{list}{{\rm(\theLcounter)}}{\usecounter{Lcounter}\leftmargin=5pt}
    }{\end{list}
} %
\newcounter{condcounter}
\newenvironment{cond}[1]{
    \renewcommand{\thecondcounter}{#1}
    \begin{list}{#1}{\usecounter{condcounter}\leftmargin=\labelsep\setlength{\rightmargin}{\leftmargin}}
     \item
}
{
    \end{list}
}
%Eqnarray with enumeration 1a, 1b, 1c, etc...
\newcounter{NumberEquation}

\newenvironment{labeleqnarray}{%
%Initializing NumberEquation
\setcounter{NumberEquation}{\value{equation}}\addtocounter{NumberEquation}{1}
\setcounter{equation}{0}
\renewcommand{\theequation}{\theNumberEquation\alph{equation}}
\begin{eqnarray}%
}%
{\end{eqnarray}\setcounter{equation}{\value{NumberEquation}}
\renewcommand{\theequation}{\arabic{equation}}}

\newcommand{\I}{{\rm I}}
\let\origwr=\wre
\def\wre{\mathop{\rm wr}\nolimits}
\renewcommand{\wr}{\wre}
\renewcommand{\phi}{\varphi}
\newcommand{\lsdp}{\ltimes}
\newcommand{\isom}{\cong}
\newcommand{\invlim}{{\displaystyle \lim_{\longleftarrow}}}
\newcommand{\<}{\left<}
\renewcommand{\>}{\right>}
\newcommand{\Ind}{{\rm Ind}}
\newcommand{\To}{\longrightarrow}
\newcommand{\freeprod}{\mathop{\displaystyle\prod\hskip-10.5pt*}}

\def\st{\mathop{\mid\;}}
\newcommand{\Tr}{{\rm Tr}}
\newcommand{\Ram}{{\rm Ram}}
\newcommand{\kssigma}{K_s(\mbox{\boldmath $\sigma$})}
\newcommand{\ksrho}{K_s(\mbox{\boldmath $\rho$})}
\newcommand{\gal}{\textnormal{Gal}}
\newcommand{\Gal}{\textnormal{Gal}}
\newcommand{\bfsigma}{\boldsymbol{\sigma}}
\newcommand{\bfrho}{\mbox{\boldmath$\rho$}}
\newcommand{\const}{{\rm const}}
\newcommand{\Sym}{{\rm Sym}}
\newcommand{\Alt}{{\rm Alt}}
\newcommand{\AwrG}{A\textnormal{wr}G}
\newcommand{\wrG}[1]{#1\textnormal{wr}G}
\newcommand{\awrg}[2]{#1\textnormal{wr}#2}
\newcommand{\nek}{,\ldots,}
\newcommand{\semidirect}{\ltimes}
\newcommand{\eps}{\varepsilon}
\newcommand{\vphi}{\varphi}
\newcommand{\measure}[2]{\mu\left(\left\{#1\mid #2\right\}\right)}
\newcommand{\cont}{\subseteq}
\newcommand{\disc}{{\rm disc}}
\newcommand{\hefresh}{\smallsetminus}
%Braces-----------------------------------------------------------
\newcommand{\tribra}[1]{\left< #1\right>}
\newcommand{\norbra}[1]{\left( #1\right)}
\newcommand{\limitto}{\mathop{\to}}
\newcommand{\fEP}[5]{\xymatrix{& #1\ar[d]^{#2}\\ #3\ar[r]^{#4} & #5}}
% ----------------------------------------------------------------
\def\dotunion{
\def\dotunionD{\bigcup\kern-9pt\cdot\kern5pt}
\def\dotunionT{\bigcup\kern-7.5pt\cdot\kern3.5pt}
\mathop{\mathchoice{\dotunionD}{\dotunionT}{}{}}}

\def\iff{\leftrightarrow}
\def\Iff{\Longleftrightarrow}
\def\inter{\bigcap}
\def\nonempty{\ne\emptyset}
\def\normal{\triangleleft}
\def\notdivide{\nmid}
\def\tensor{\otimes}
\def\union{\bigcup}
\def\ab{{\rm ab}}
\def\sol{{\rm sol}}
\def\and{{\rm and}}
\def\cd{{\rm cd}}
\def\alg{{\rm alg}}
\def\Aut{{\rm Aut}}
\def\gcd{{\rm gcd}}
\def\GL{{\rm GL}}
\def\Hom{{\rm Hom}}
\def\Homeo{{\rm Homeo}}
\def\id{{\rm id}}
\def\Fr{{\rm Fr}}
\def\Im{{\rm Im}}
\def\Inn{{\rm Inn}}
\def\irr{{\rm irr}}
\def\Ker{{\rm Ker}}
\def\mod{\;\mathop{\rm mod}\hskip0.2em\relax}
\def\ord{{\rm ord}}
\def\Out{{\rm Out}}
\def\PGL{{\rm PGL}}
\def\PSL{{\rm PSL}}
\def\Br{{\rm Br}}
\def\pr{{\rm pr}}
\def\rank{\mathop{\rm rank}\nolimits}
\def\ring{{\rm ring}}
\def\res{{\rm res}}
\def\Res{{\rm Res}}
\def\SL{{\rm SL}}
\def\Spec{{\rm Spec}}
\def\Subg{{\rm Subg}}
\newcommand{\calA}{{\mathcal A}}
\newcommand{\calB}{{\mathcal B}}
\newcommand{\calC}{{\mathcal C}}
\newcommand{\calD}{{\mathcal D}}
\newcommand{\calE}{{\mathcal E}}
\newcommand{\calF}{{\mathcal F}}
\newcommand{\calG}{{\mathcal G}}
\newcommand{\calH}{{\mathcal H}}
\newcommand{\calI}{{\mathcal I}}
\newcommand{\calJ}{{\mathcal J}}
\newcommand{\calK}{{\mathcal K}}
\newcommand{\calL}{{\mathcal L}}
\newcommand{\calM}{{\mathcal M}}
\newcommand{\calN}{{\mathcal N}}
\newcommand{\calO}{{\mathcal O}}
\newcommand{\calP}{{\mathcal P}}
\newcommand{\calQ}{{\mathcal Q}}
\newcommand{\calR}{{\mathcal R}}
\newcommand{\calS}{{\mathcal S}}
\newcommand{\calT}{{\mathcal T}}
\newcommand{\calU}{{\mathcal U}}
\newcommand{\calV}{{\mathcal V}}
\newcommand{\calW}{{\mathcal W}}
\newcommand{\calX}{{\mathcal X}}
\newcommand{\calY}{{\mathcal Y}}
\newcommand{\calZ}{{\mathcal Z}}
\newcommand{\agal}{{\tilde a}}       \newcommand{\Agal}{{\tilde A}}
\newcommand{\bgal}{{\tilde b}}       \newcommand{\Bgal}{{\tilde B}}
\newcommand{\cgal}{{\tilde c}}       \newcommand{\Cgal}{{\tilde C}}
\newcommand{\dgal}{{\tilde d}}       \newcommand{\Dgal}{{\tilde D}}
\newcommand{\egal}{{\tilde e}}       \newcommand{\Egal}{{\tilde E}}
\newcommand{\fgal}{{\tilde f}}       \newcommand{\Fgal}{{\tilde F}}
\newcommand{\ggal}{{\tilde g}}       \newcommand{\Ggal}{{\tilde G}}
\newcommand{\hgal}{{\tilde h}}       \newcommand{\Hgal}{{\tilde H}}
\newcommand{\igal}{{\tilde i}}       \newcommand{\Igal}{{\tilde I}}
\newcommand{\jgal}{{\tilde j}}       \newcommand{\Jgal}{{\tilde J}}
\newcommand{\kgal}{{\tilde k}}       \newcommand{\Kgal}{{\tilde K}}
\newcommand{\lgal}{{\tilde l}}       \newcommand{\Lgal}{{\tilde L}}
\newcommand{\mgal}{{\tilde m}}       \newcommand{\Mgal}{{\tilde M}}
\newcommand{\ngal}{{\tilde n}}       \newcommand{\Ngal}{{\tilde N}}
\newcommand{\ogal}{{\tilde o}}       \newcommand{\Ogal}{{\tilde O}}
\newcommand{\pgal}{{\tilde p}}       \newcommand{\Pgal}{{\tilde P}}
\newcommand{\qgal}{{\tilde q}}       \newcommand{\Qgal}{{\tilde Q}}
                                     \newcommand{\bbQgal}{{\tilde \bbQ}}
\newcommand{\rgal}{{\tilde r}}       \newcommand{\Rgal}{{\tilde R}}
\newcommand{\sgal}{{\tilde s}}       \newcommand{\Sgal}{{\tilde S}}
\newcommand{\tgal}{{\tilde t}}       \newcommand{\Tgal}{{\tilde T}}
\newcommand{\ugal}{{\tilde u}}       \newcommand{\Ugal}{{\tilde U}}
\newcommand{\vgal}{{\tilde v}}       \newcommand{\Vgal}{{\tilde V}}
\newcommand{\wgal}{{\tilde w}}       \newcommand{\Wgal}{{\tilde W}}
\newcommand{\xgal}{{\tilde x}}       \newcommand{\Xgal}{{\tilde X}}
\newcommand{\ygal}{{\tilde y}}       \newcommand{\Ygal}{{\tilde Y}}
\newcommand{\zgal}{{\tilde z}}       \newcommand{\Zgal}{{\tilde Z}}
                                     \newcommand{\bbZgal}{{\tilde \bbZ}}

\newcommand{\alphagal}{\tilde\alpha}
\newcommand{\betagal}{\tilde\beta}
\newcommand{\Gammagal}{\tilde\Gamma}
\newcommand{\gag}[1]{\bar{#1}}
\def\agag{{\bar a}}     \def\Agag{{\gag{A}}}
\def\bgag{{\bar b}}     \def\Bgag{{\gag{B}}}
\def\cgag{{\bar c}}     \def\Cgag{{\gag{C}}}
\def\dgag{{\bar d}}     \def\Dgag{{\gag{D}}}
\def\egag{{\bar e}}     \def\Egag{{\gag{E}}}
\def\fgag{{\bar f}}     \def\Fgag{{\gag{F}}}
\def\ggag{{\bar g}}     \def\Ggag{{\gag{G}}}
\def\hgag{{\bar h}}     \def\Hgag{{\gag{H}}}
\def\igag{{\bar i}}     \def\Igag{{\gag{I}}}
\def\jgag{{\bar \jmath}}    \def\Jgag{{\bar{J}}}
\def\kgag{{\bar k}}     \def\Kgag{{\gag{K}}}
\def\lgag{{\bar l}}     \def\Lgag{{\gag{L}}}
\def\mgag{{\bar m}}     \def\Mgag{{\gag{M}}}
\def\ngag{{\bar n}}     \def\Ngag{{\gag{N}}}
\def\ogag{{\bar o}}     \def\Ogag{{\gag{O}}}
\def\pgag{{\bar p}}     \def\Pgag{{\gag{P}}}
\def\qgag{{\bar q}}     \def\Qgag{{\gag{Q}}}
\def\rgag{{\bar r}}     \def\Rgag{{\gag{R}}}
\def\sgag{{\bar s}}     \def\Sgag{{\gag{S}}}
\def\tgag{{\bar t}}     \def\Tgag{{\gag{T}}}
\def\ugag{{\bar u}}     \def\Ugag{{\gag{U}}}
\def\vgag{{\bar v}}     \def\Vgag{{\gag{V}}}
\def\wgag{{\bar w}}     \def\Wgag{{\gag{W}}}
\def\xgag{{\bar x}}     \def\Xgag{{\gag{X}}}
\def\ygag{{\bar y}}     \def\Ygag{{\gag{Y}}}
\def\zgag{{\bar z}}     \def\Zgag{{\gag{Z}}}

\def\alphagag{{\bar \alpha}}
\def\betagag{{\bar \beta}}
\def\gammagag{{\bar \gamma}}
\def\phigag{{\bar \varphi}}
\def\psigag{{\bar \psi}}
\def\rhogag{{\bar \rho}}
\def\pigag{{\bar \pi}}
\def\nugag{{\bar \nu}}
\def\thetagag{{\bar \theta}}
\def\etagag{{\bar \eta}}
\def\mugag{{\bar \mu}}
\def\sigmagag{{\bar \sigma}}
\def\calEgag{{\bar \calE}}
\def\calXgag{{\bar \calX}}

\def\ahat{{\hat a}}     \def\Ahat{{\hat A}}
\def\bhat{{\hat b}}     \def\Bhat{{\hat B}}
\def\chat{{\hat c}}     \def\Chat{{\hat C}}
\def\dhat{{\hat d}}     \def\Dhat{{\hat D}}
\def\ehat{{\hat e}}     \def\Ehat{{\hat E}}
\def\fhat{{\hat f}}     \def\Fhat{{\hat F}}
\def\ghat{{\hat g}}     \def\Ghat{{\hat G}}
\def\hhat{{\hat h}}     \def\Hhat{{\hat H}}
\def\ihat{{\hat i}}     \def\Ihat{{\hat I}}
\def\jhat{{\hat j}}     \def\Jhat{{\hat J}}
\def\khat{{\hat k}}     \def\Khat{{\hat K}}
\def\lhat{{\hat l}}     \def\Lhat{{\hat L}}
\def\mhat{{\hat m}}     \def\Mhat{{\hat M}}
\def\nhat{{\hat n}}     \def\Nhat{{\hat N}}
\def\ohat{{\hat o}}     \def\Ohat{{\hat O}}
\def\phat{{\hat p}}     \def\Phat{{\hat P}}
\def\qhat{{\hat q}}     \def\Qhat{{\hat Q}}
\def\rhat{{\hat r}}     \def\Rhat{{\hat R}}
\def\shat{{\hat s}}     \def\Shat{{\hat S}}
\def\that{{\hat t}}     \def\That{{\hat T}}
\def\uhat{{\hat u}}     \def\Uhat{{\hat U}}
\def\vhat{{\hat v}}     \def\Vhat{{\hat V}}
\def\what{{\hat w}}     \def\What{{\hat W}}
\def\xhat{{\hat x}}     \def\Xhat{{\hat X}}
\def\yhat{{\hat y}}     \def\Yhat{{\hat Y}}
\def\zhat{{\hat z}}     \def\Zhat{{\hat Z}} \def\bbZhat{{\hat \bbZ}}

\def\calEhat{{\hat \calE}}
\def\alphahat{{\hat \alpha}}
\def\betahat{{\hat \beta}}
\def\gammahat{{\hat \gamma}}
\def\phihat{{\hat \varphi}}
\def\Phihat{{\hat \Phi}}
\def\Psihat{{\hat \Psi}}
\def\psihat{{\hat \psi}}
\def\nuhat{{\hat \nu}}
\def\muhat{{\hat \mu}}
\def\thetahat{{\hat \theta}}
\def\etahat{{\hat \eta}}
\newcommand{\bbA}{\mathbb{A}}
\newcommand{\bbB}{\mathbb{B}}
\newcommand{\bbC}{\mathbb{C}}
\newcommand{\bbD}{\mathbb{D}}
\newcommand{\bbE}{\mathbb{E}}
\newcommand{\bbF}{\mathbb{F}}
\newcommand{\bbG}{\mathbb{G}}
\newcommand{\bbH}{\mathbb{H}}
\newcommand{\bbI}{\mathbb{I}}
\newcommand{\bbJ}{\mathbb{J}}
\newcommand{\bbK}{\mathbb{K}}
\newcommand{\bbL}{\mathbb{L}}
\newcommand{\bbM}{\mathbb{M}}
\newcommand{\bbN}{\mathbb{N}}
\newcommand{\bbO}{\mathbb{O}}
\newcommand{\bbP}{\mathbb{P}}
\newcommand{\bbQ}{\mathbb{Q}}
\newcommand{\bbR}{\mathbb{R}}
\newcommand{\bbS}{\mathbb{S}}
\newcommand{\bbT}{\mathbb{T}}
\newcommand{\bbU}{\mathbb{U}}
\newcommand{\bbV}{\mathbb{V}}
\newcommand{\bbW}{\mathbb{W}}
\newcommand{\bbX}{\mathbb{X}}
\newcommand{\bbY}{\mathbb{Y}}
\newcommand{\bbZ}{\mathbb{Z}}

\newcommand{\bfa}{{\mathbf a}}
\newcommand{\bfb}{\mathbf b}
\newcommand{\bfc}{\mathbf c}
\newcommand{\bfd}{\mathbf d}
\newcommand{\bfe}{\mathbf e}
\newcommand{\bff}{\mathbf f}
\newcommand{\bfg}{\mathbf g}
\newcommand{\bfh}{\mathbf h}
\newcommand{\bfi}{\mathbf i}
\newcommand{\bfj}{\mathbf j}
\newcommand{\bfk}{\mathbf k}
\newcommand{\bfl}{\mathbf l}
\newcommand{\bfm}{\mathbf m}
\newcommand{\bfn}{\mathbf n}
\newcommand{\bfo}{\mathbf o}
\newcommand{\bfp}{\mathbf p}
\newcommand{\bfq}{\mathbf q}
\newcommand{\bfr}{\mathbf r}
\newcommand{\bfs}{\mathbf s}
\newcommand{\bft}{\mathbf t}
\newcommand{\bfu}{\mathbf u}
\newcommand{\bfv}{\mathbf v}
\newcommand{\bfw}{\mathbf w}
\newcommand{\bfx}{\mathbf x}
\newcommand{\bfy}{\mathbf y}
\newcommand{\bfz}{\mathbf z}

\newcommand{\bfA}{\mathbf A}
\newcommand{\bfB}{\mathbf B}
\newcommand{\bfC}{\mathbf C}
\newcommand{\bfD}{\mathbf D}
\newcommand{\bfE}{\mathbf E}
\newcommand{\bfF}{\mathbf F}
\newcommand{\bfG}{\mathbf G}
\newcommand{\bfH}{\mathbf H}
\newcommand{\bfI}{\mathbf I}
\newcommand{\bfJ}{\mathbf J}
\newcommand{\bfK}{\mathbf K}
\newcommand{\bfL}{\mathbf L}
\newcommand{\bfM}{\mathbf M}
\newcommand{\bfN}{\mathbf N}
\newcommand{\bfO}{\mathbf O}
\newcommand{\bfP}{\mathbf P}
\newcommand{\bfQ}{\mathbf Q}
\newcommand{\bfR}{\mathbf R}
\newcommand{\bfS}{\mathbf S}
\newcommand{\bfT}{\mathbf T}
\newcommand{\bfU}{\mathbf U}
\newcommand{\bfV}{\mathbf V}
\newcommand{\bfW}{\mathbf W}
\newcommand{\bfX}{\mathbf X}
\newcommand{\bfY}{\mathbf Y}
\newcommand{\bfZ}{\mathbf Z}

\newcommand{\bftheta}{\mbox{\boldmath$\theta$}}

\newcommand{\mfrak}{{\mathfrak m}} 
\newcommand{\pfrak}{{\mathfrak p}}       \newcommand{\Pfrak}{{\mathfrak P}}
\newcommand{\frp}{{\mathfrak p}}         \newcommand{\frP}{{\mathfrak P}}
\newcommand{\qfrak}{{\mathfrak q}}       \newcommand{\Qfrak}{{\mathfrak Q}}

\newcommand{\Quot}{{\rm Quot}}
\newcommand{\C}{\calC}

\CompileMatrices

\begin{titlepage}
\center{\fontsize{10pt}{20pt} \selectfont
\includegraphics[width=40pt]{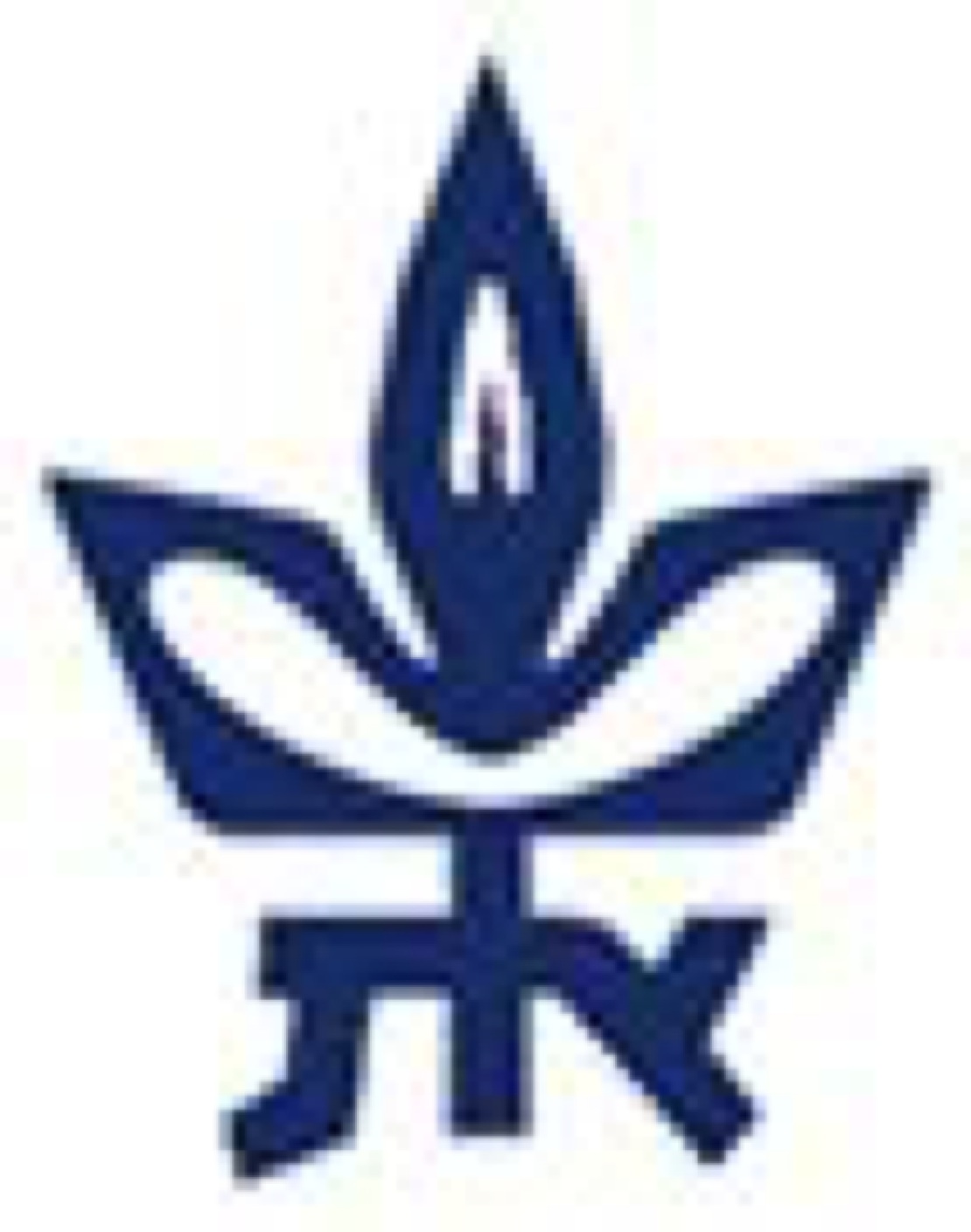}\\
Tel Aviv University\\
The Raymond and Beverly Sackler
Faculty of Exact Sciences\\
School of Mathematical Sciences } \vspace{\stretch{2}}
%\vskip80pt
\center{\fontsize{14pt}{24pt} \selectfont \textsc{Pseudo
Algebraically Closed Extensions}} \vspace{\stretch{1}}
%\vskip20pt
\center{A thesis submitted for the degree ``Doctor of Philosophy''}
%\vskip100pt
\vspace{\stretch{1}} \center{by\\
Lior Bary-Soroker} \vspace{\stretch{3}}
%\vskip20pt
\center{under the supervision of\\ Prof. Dan Haran}
\vspace{\stretch{1}}
%\vskip20pt
\center{Submitted to the Senate of Tel Aviv University\\
September, 2008}
\end{titlepage}

\tableofcontents

\chapter{Introduction}
\section{Background and Motivation}
\subsection{Hilbert's Tenth Problem}
Hilbert's tenth problem asks for a general algorithm to decide
whether a given polynomial Diophantine equation with integer
coefficients has a solution in integers. In 1972 Matijasevich gave a
negative solution to that problem relying on the works of Davis,
Putnam, and J. Robinson since the 1930's.

These developments led Robinson to ask Hilbert's tenth problem
w.r.t.\ the ring of all algebraic integers. In 1987 Rumely indeed
found an algorithm for solving Diophantine problems over the ring of
all algebraic integers using the following local-global principle.
Let $\bbQgal$ be the field of all algebraic numbers and $\bbZgal$
the ring of all algebraic integers. Then
a variety $V$ over $\bbQgal$ has an integral point, i.e.\
$V(\bbZgal)\neq \emptyset$, if and only if $V$ has an integral point
in each completion of $\bbQgal$.

Jarden and Razon obtained Rumely's local-global principle for an
abundance of algebraic integer rings
\cite{JardenRazon1998,JardenRazon_ms} and then derived an
affirmative solution of Hilbert's tenth problem for those rings.
The main innovation made by Jarden-Razon is the following key
definition \cite{JardenRazon1994}.

\begin{definition}\label{def:PAC}
Let $M$ be a field and $K$ a subset. Then $M/K$ is called
\textbf{pseudo algebraically closed (PAC)} (or, alternatively,
we say that $M$ is a PAC extension of $K$) if for every
absolutely irreducible polynomial $f(T,X)\in M[T,X]$ which is
separable in $X$ there are infinitely many $(a,b)\in K\times M$
for which $f(a,b) = 0$.
\end{definition}

Note that this definition generalizes the classical notion of
PAC fields (taking $K=M$), which is central in the subject Field
Arithmetics (some books on this subject are
\cite{FriedJarden2005,MalleMatzat1999,Volklein1996}).

Let us describe Jarden-Razon result on Rumely's local-global
principle more precisely. Let $K$ be a field.
Consider an $e$-tuple of Galois automorphisms
$\bfsigma=(\sigma_1,\ldots,\sigma_e)\in \gal(K)^e$. Then
$K_s(\bfsigma) = \{ x\in K_s \mid \sigma_i(x) = x,\ \forall\ i\}$
denotes the fixed field of $\bfsigma$ in the separable closure $K_s$
and $\left< \bfsigma\right>$ denote the closed subgroup generated by
$\sigma_1,\ldots,\sigma_e$.

\begin{thmint}\label{thm:intAlmostallPAC}
Let $R$ be a countable Hilbertian domain, $K$ its quotient ring, and
$e\geq 1$ an integer. Then $K_s(\bfsigma)/R$ is PAC and
$\left<\bfsigma\right> $ is free of rank $e$ for almost all
$\bfsigma\in \gal(K)^e$ (with respect to the Haar measure).
\end{thmint}

Recall that a Hilbertian domain is an integral domain $R$ with the
following feature. For every irreducible $f\in K[X,Y]$ which is
separable in $Y$ there exists $a\in R$ s.t.\ $f(a,Y)$ is irreducible
over $K$. For example, the ring of integers of an arbitrary global
field is a Hilbertian domain.

A classic result of Jarden \cite[Theorems 18.5.6 and
18.6.1]{FriedJarden2005} asserts that $K_s(\bfsigma)$ is PAC and
$\left<\bfsigma\right> $ is free of rank $e$. The rest of the
theorem, i.e.\ the PACness of $K_s(\bfsigma)/R$, is the main result
in \cite{JardenRazon1994}.

Now we are ready to present the extension of Rumely's local-global
principle.

\begin{thmint}\label{thm:intLGprinciple}
Let $K$ be a global field. For almost all $\bfsigma\in \gal(K)^e$
the following holds: Let $V$ be a variety defined over $M =
K_s(\bfsigma)$. Then $V$ has an integral point in $M$ if and only if
$V$ has an integral point in each completion (w.r.t.\ primes of the
ring of integers).
\end{thmint}

This theorem is a special case of \cite[Theorem
2.5]{JardenRazon_ms}.

\subsection{Connections to Other Areas}
Above we described the original motivation for the definition of PAC
extensions. Later on Jarden and Razon sharpened
Theorem~\ref{thm:intLGprinciple} in a series of papers (see
\cite{JardenRazon_ms} for the latest version). However some new
applications to PAC extensions having no relation to Hilbert's tenth
problem started to come out.

\subsubsection{Abundance of Hilbertian Fields}
A Galois extension $N/K$ is said to be unbounded if the set
$\{\ord(\sigma) \mid \sigma\in \gal(N/K)\}$ is unbounded. In
\cite{Razon1997} Razon applies Theorem~\ref{thm:intAlmostallPAC} to
find an abundance of Hilbertian fields:

\begin{thmint}
Let $K$ be a countable Hilbertian field and $N/K$ an unbounded
abelian extension. Then for almost all $\bfsigma\in \gal(K)^e$ every
subextension $M$ of $K_s(\bfsigma)N/N$ is Hilbertian.
\end{thmint}

The following two results show some nice properties of a field $K$
having a ``non-degenerate'' PAC extension. The main idea lying
behind these results is that such fields satisfy some weak Hilbert's
irreducibility theorem, as described further below in the
introduction. Let us briefly describe the results.

\subsubsection{Dirichlet's Theorem for Polynomial Rings}
The Artin-Kornblum analog of Dirichlet's theorem on primes in
arithmetical progressions asserts the following for a finite field
$\bbF$.

\begin{thmint}
For any coprime $a,b\in \bbF[X]$ and sufficiently large positive
integer $n$ there exists $c\in \bbF[X]$ for which $a+bc$ is
irreducible of degree $n$.
\end{thmint}

A natural continuation is to generalize Dirichlet's theorem to
infinite fields. Note that, obviously, Dirichlet's theorem does not
hold for $\bbC$ or $\bbR$. A more subtle but still easy observation
which the author learnt from Jarden is that over a Hilbertian field
$K$ Dirichlet's theorem does hold (see
Chapter~\ref{chapter:Dirichlet} for details).
In this work we prove that Dirichlet's theorem holds for a field $K$
provided it has a PAC extension.

\begin{thmintmine}\label{thm:intDirichlet}
Let $K$ be a field, $a,b\in K[X]$ coprime, and $n$ sufficiently
large. Assume that there exists a PAC extension $M/K$ and a
separable extension $N/M$ of degree $n$. Then there exist infinitely
many $c\in K[X]$ such that $a+bc$ is irreducible of degree $n$.
\end{thmintmine}

This result appears in \cite{Bary-SorokerDirichlet}.

\subsubsection{Scaled Trace Forms}
Let $L/K$ be a separable extension of degree $n$. Then $L$ is
equipped with the non-degenerate quadratic trace form. Namely,
$x\mapsto \Tr_{L/K}(x^2)$, for $x\in L$. A generalization of the
trace form is the following scaled trace form. Fix $\lambda \in
L^\times$. Then the scaled (by $\lambda$) trace from is the
non-degenerate quadratic form on $L$ defined by
\[
x\mapsto \Tr_{L/K}(\lambda x^2),
\]
for all $x\in L$.

Scharlau \cite{Scharlau1987} and Waterhouse
\cite{Waterhouse1986} (independently) prove that over a
Hilbertian field $K$ every non-degenerated quadratic form is
isomorphic to a scaled trace form. In a joint work with Kelmer
\cite{Bary-SorokerKelmer}, we extend Scharlau-Waterhouse proof
to fields having a PAC extension.

\begin{thmint}\label{thm:intTrace}
Let $K$ be a field. Assume that there exists a PAC extension $M/K$
and a separable extension $N/M$ of degree $n$. Then every
non-degenerate quadratic form of degree $n$ over $K$ is isomorphic
to a scaled trace form.
\end{thmint}

\begin{remark}
The latter two theorems make the property that a field $K$ has a PAC
extension $M/K$ and a separable extension $N/M$ of degree $n$
interesting. Below we find some families of fields having this
property for many $n$'s, e.g., pro-solvable extensions of $\bbQ$.
\end{remark}

\section{The Galois Structure of PAC Extensions}
\label{sec:itroGalois} We develop an elementary machinery to study
field extensions. The essential ingredient is a generalization of
embedding problems to field extensions. Then we apply it to PAC
extensions. In particular we get a key property -- the \emph{lifting
property} (Proposition~\ref{prop:ExtensionSoltoGeoSol_DEP}). Let us
describe some of the consequences of this developments.

\subsection{Restrictions on PAC Extensions}
In \cite{JardenRazon1994} (where PAC extensions firstly appear)
Jarden and Razon find some Galois extensions $M$ of $\bbQ$ such
that $M$ is PAC as a field but $M$ is a PAC extension of no
number field. (For this a heavy tool is used, namely Faltings'
theorem.) This led them to ask whether this is a coincidence or
a general phenomenon.

In \cite{Jarden2006} Jarden answers this question by showing that
the only Galois PAC extension of an arbitrary number field is its
algebraic closure. Jarden does not use Faltings' theorem, but
different kind of results. Namely, Frobenius' density theorem,
Neukirch's characterization of $p$-adically closed fields among all
algebraic extensions of $\bbQ$, and the special property of $\bbQ$
that it has no proper subfields (!). For that reason Jarden's method
is restricted to number fields.

The next step is to generalize Jarden's results to a finitely
generated infinite field $K$. Elaborating Jarden-Razon original
method, i.e.\ using Faltings' theorem (and Grauert-Manin theorem in
positive characteristic), Jarden and the author
\cite{Bary-SorokerJarden} generalize Jarden's result to $K$.

\begin{thmint}\label{thm:intGaloisPACextFG}
Let $K$ be a finitely generated infinite field. The only Galois PAC
extension of $K$ is its separable closure.
\end{thmint}

Using the lifting property mentioned above and wreath products, we
can elementarily and easily reprove
Theorem~\ref{thm:intGaloisPACextFG}. Moreover, our proof uses no
special features of finitely generated fields, so it generalizes
Theorem~\ref{thm:intGaloisPACextFG} to arbitrary fields.
\begin{thmintmine}\label{thm:intmineGalois}
Let $M/K$ be a proper separable algebraic PAC extension. Then the
Galois closure of $M/K$ is the separable closure of $K$. \\
In particular, if $M/K$ is a Galois PAC extension, the either $M=K$
or $M=K_s$.
\end{thmintmine}

The Galois structure of a PAC extension $M/K$ has more restrictions.
For example in Corollary~\ref{cor: finite PAC extension} we classify
all finite PAC extensions.
\begin{thmintmine}\label{thm:intminefinite}
Let $M/K$ be a finite extension. Then either $M/K$ is PAC if and
only if one of the following holds.
\begin{enumerate}
\item $K_0$ is real closed and $K$ is the algebraic closure of
$K_0$.
\item $K_0$ is PAC and $K/K_0$ is purely inseparable.
\end{enumerate}
\end{thmintmine}

Interestingly, this theorem implies a positive answer to a
seemingly unrelated problem, namely Problem~18.7.8 of
\cite{FriedJarden2005} in case of finitely generated fields: For
a given Hilbertian field $K$, this problem asks whether the
following `bottom theorem' holds. For almost all $\bfsigma\in
\gal(K)^e$ the field $M=K_s(\bfsigma)$ has no subfield $L$ such
that $1<[M:L]<\infty$.

In \cite{Haran1985} Haran proves an earlier version of the
bottom theorem, namely with the additional condition that
$K\subseteq L$. Note that if, e.g., $K = \bbF_p(x)$, then
$[M:M^p]=p$ for any algebraic separable extension $M/K$ (see
discussion before Theorem~\ref{thm:bottom} for more details).
Therefore the bottom theorem fails for this $K$ for trivial
reasons. Hence the condition `$M/L$ is separable' should be
added to the formulation of the bottom theorem.

Now the bottom theorem for $K$ a finitely generated infinite field
is valid. The main ingredient in the proof is, as mentioned above,
Theorem~\ref{thm:intminefinite}.
\begin{thmintmine}
The bottom theorem holds for all finitely generated infinite fields.
\end{thmintmine}

\subsection{Descent of Galois Groups}
In \cite{Razon2000} Razon proves for a PAC extension $M/K$ that
every separable extension $N/M$ descends to a separable extension
$L/K$.

\begin{thmint}\label{thm:intSplitting}
Let $M/K$ be a PAC extension and $N/M$ a separable algebraic
extension. Then there exists a separable algebraic extension $L/K$
that is linearly disjoint from $M$ over $K$ such that $N=ML$.
\end{thmint}

We generalize Razon's result and get a stronger descent result:
\begin{thmintmine}\label{thm:intmineDescent}
Let $M/K$ be a PAC extension and $N/M$ a finite Galois
extension. Assume $\gal(N/M) \leq H$, where $H$ is regularly
realizable over $K$. Then there exists a Galois extension $L/K$
such that $\gal(L/K)\leq H$ and $N=LM$.
\end{thmintmine}

Indeed, applying Theorem~\ref{thm:intmineDescent} to the symmetric
group $H=S_n$ yields a proof of Theorem~\ref{thm:intSplitting} (see
proof of Corollary~\ref{cor:Razon}). It is interesting to note that
the original proof of Theorem~\ref{thm:intSplitting} is similar but
very specific: One only considers the regular realization of $H=S_n$
generated by the generic polynomial $f(T_1,\ldots,T_n,X) = X^n + T_1
X^{n-1} + \cdots + T_n$.

Let me explain the name `descent' attached to
Theorem~\ref{thm:intmineDescent}. If a finite Galois group
$G=\gal(N/M)$ over $M$ regularly occurs over $K$, then, by taking
$H=G$ in Theorem~\ref{thm:intmineDescent} we get that $G$ occurs
over $K$ (since $G=\gal(L/K)$ in that case). Thus $G$
\emph{descends} to a Galois group over $K$.

As a consequence of this and of the fact that abelian groups are
regularly realizable over any field, we get, for example, that
\[M_{\ab} = MK_{\ab}.\]

\section{Projective Pairs}
There is a strong connection between PAC fields and projective
profinite groups.

\begin{thmint}
The absolute Galois group of an arbitrary PAC field is projective,
and vice-versa, every projective profinite group can be realized as
the absolute Galois group of a PAC field.
\end{thmint}

Ax proved the first assertion \cite[Theorem~11.6.2]{FriedJarden2005}
and Lubotzky and v.d.~Dries proved the second one
\cite[Corollary~23.1.2]{FriedJarden2005}.

We introduce a similar characterization for PAC extensions of
PAC fields and \emph{projective pairs}. A projective pair
$(\Gamma,\Lambda)$ is composed of a profinite group $\Lambda$
and a closed subgroup $\Gamma$ for which any \emph{double finite
embedding problem} is weakly solvable (see
Definition~\ref{def:pp}).

\begin{thmintmine}
\begin{enumerate}
\item
Let $M$ be a PAC extension of a PAC field $K$. Then $(\gal(M),
\gal(K))$ is a projective pair.
\item
Let $M$ be an algebraic extension of a PAC field $K$. Then $M/K$
is PAC if and only if $(\gal(M), \gal(K))$ is projective.
\item
Let $(\Gamma,\Lambda)$ be a projective pair. Then there exists a
PAC extension $M$ of a PAC field $K$ such that $\Gamma\cong
\gal(M)$, $\Lambda\cong \gal(K)$.
\end{enumerate}
\end{thmintmine}

Thus any result about projective pairs immediately yields an analog
for PAC extensions of PAC fields. This connection motivates the
study of projective pairs.

Interestingly, although we cannot directly apply the results about
projective pairs to analogous results on PAC extensions (of non PAC
field) we can apply the ideas and technics. And indeed the
elementary machinery mentioned in Section~\ref{sec:itroGalois} comes
from this analogy.

The study of projective pairs in this work go along two
perspectives. First we prove the analogs of results about PAC
extensions. For example, the analog of
Theorem~\ref{thm:intmineGalois} asserts that the only normal
projective pairs (i.e.\ $\Gamma\normal \Lambda)$ are when $\Gamma=1$
or $\Gamma=\Lambda$. Sometimes in the group theoretic setting the
results become stronger, e.g., the analog of
Theorem~\ref{thm:intmineDescent} implies the following
\begin{thmintmine}
Let $(\Gamma,\Lambda)$ be a projective pair. Then $\Lambda =
N\rtimes \Gamma$, for some $N\normal \Lambda$.
\end{thmintmine}

In the same time we find families of examples of projective
pairs. By the transitivity of PAC extensions
(Proposition~\ref{prop:transitivityofPACextension}) we get new
families of PAC extensions. For example, the following is a
special case of Corollary~\ref{cor:ppAbelianExtension}.

\begin{thmintmine}
Every projective group $P$ can be realized as the absolute Galois
group of a PAC extension of some Hilbertian field $K$. Moreover if
the rank of $P$ is at most countable, then we can take
$K=\bbQ_{\ab}$.
\end{thmintmine}

Unfortunately, this does not lead us to discover new PAC extensions
of $\bbQ$.

\section{Weak Hilbert's Irreducibility Theorem}
There is a deep connection between embedding problems, irreducible
specializations of polynomials, and rational points on varieties.
The following theorem is a good evidence for this connection.

\begin{thmint}
Let $K$ be a PAC field. Then $K$ is Hilbertian if and only if $K$ is
$\omega$-free.
\end{thmint}
(Recall that the PACness means that all varieties have rational
points, Hilbertianity implies that all polynomials have the
irreducible specialization property, and $\omega$-freeness
asserts that all finite embedding problems are solvable.) In
their work on Frobenius fields \cite{FriedHaranJarden1984},
Haran, Fried, and Jarden refined the above connection.

We extract from the previous result the exact statement that
connects embedding problems, irreducible specializations of
polynomials, and rational points on varieties (see
Lemma~\ref{lem:characterizationofirrspecialization}). Next we apply
that general criterion to a field $K$ which has PAC extensions. This
gives a weak Hilbert's irreducibility theorem for $K$
(Proposition~\ref{prop:PACextirr}). One nice corollary is in the
case where the polynomial is the ``most irreducible."

\begin{thmintmine}
Let $K$ be a field and $f(T,X)\in K(T)[X]$ an irreducible polynomial
of degree $n$ in $X$ whose Galois group over $\Kgal(T)$ is $S_n$.
Further assume that there exists a PAC extension $M/K$ and a
separable extension $N/M$ of degree $n$. Then there exist infinitely
many $a\in K$ for which $f(a,X)$ is irreducible over $M$, and hence
over $K$.
\end{thmintmine}

This result is essential in the proofs of
Theorems~\ref{thm:intDirichlet} and \ref{thm:intTrace}.

\section{Fields Having PAC Extensions} In light of the above
results, it is a significant feature for a field $K$ to have
non-trivial PAC extensions.

We focus on two directions. First we generalize
Theorem~\ref{thm:intAlmostallPAC} in case $K$ is a finitely
generated field. Recall that a finitely generated infinite field is
always Hilbertian.

\begin{thmintmine}
Let $K$ be a finitely generated infinite field and $e\geq 1$ a
positive integer. Then for almost all $\bfsigma\in \gal(K)^e$
and for every field $F\subseteq K_s(\bfsigma)$ which is not
algebraic over a finite field the extension $K_s(\bfsigma)/F$ is
PAC.
\end{thmintmine}

We do not know whether this result is true in the more general case
where $K$ is an arbitrary countable Hilbertian field (e.g.\
$K=\bbQ_{\ab}$). A positive answer would imply the bottom theorem
for countable Hilbertian fields.

Next we go in the opposite direction, that is, we consider a fixed
field $F$ and try to find PAC extensions of it. More precisely we
ask the following
\begin{question}
For what positive integers $n$ there exist a PAC extension $M/F$ and
a separable extension $N/M$ of degree $n$?
\end{question}

In several cases we can give a satisfactory answer as presented in
the following result.

\begin{thmintmine}
Let $K$ be a countable Hilbertian field and $F/K$ a separable
algebraic extension. Then there exist a PAC extension $M/F$ and a
separable extension $N/M$ of degree $n$ in the following cases.
\begin{enumerate}
\item
$F/K$ is pro-solvable and $n\geq 5$.
\item
There exists a prime $p$ such that $p\nmid [\Fhat:K]$ (as
supernatural numbers), where $\Fhat$ is the Galois closure of $F/K$.
\end{enumerate}
\end{thmintmine}

This result has several applications. For example one can extend
Theorem~\ref{thm:intTrace} to the following

\begin{thmint}
Let $K$ be a Hilbertian field (of any cardinality) and let $F/K$ be
a pro-solvable extension (resp.\ an extension whose Galois closure
is prime to $p$). Then for every $n\geq 5$ (resp.\ $n\geq 1$) any
non-degenerate quadratic form of dimension $n$ is isomorphic to a
scaled trace form over $F$.
\end{thmint}

This result appear in \cite{Bary-SorokerKelmer}.
%-------------------------------------------------------------------------------
\chapter{Preliminaries in Galois Theory}
This chapter sets up the necessary background in Galois theory and
finite group theory needed for this work. It mainly fixes notation
and gives some of the basic properties which will be used later.
Also some of the technical details of the upcoming chapters are
hidden in this chapter. The expert reader may wish to skip this
expository part, and return to it when needed.

\section{Permutational Galois Groups}
Let $F/K$ be a Galois extension. The Galois group $G = \gal(F/K)$ is
equipped with several natural actions coming from polynomials and
subgroups. We describe two of these actions below.

Firstly let $f(X)\in K[X]$ be a separable monic polynomial which
splits over $F$. Then $f(X) = \prod_{i=1}^{n}(X-\alpha_i)$, where
all $\alpha_i\in F$ and are distinct. Since $G$ fixes the
coefficients of $f(X)$, it permutes the roots of $f$. Properties of
$f$ are encoded in this action (this idea goes back to Galois
himself), e.g., the following

\begin{lemma}\label{lem:polynomial_action}
Let $R$ be the set of roots of $f$.
\begin{enumerate}
\item
$R$ generates $F$ over $K$ if and only if the action of $G$ on
$R$ is faithful.
\item
There is a correspondence between the factorization
$f(X)=\prod_{i=1}^m f_i(X)$ of $f(X)$ into irreducible polynomials
over $K$ and the decomposition $R = \dotunion_{i=1}^m R_i$ of $R$
into $G$-orbits $R_i$, given by $f_i(X) = \prod_{\alpha\in
R_i}(X-\alpha)$ (under a suitable rearrangement of the factors
$f_i$).
\item
$f(X)$ is irreducible over $K$ if and only if $(G,R)$ is transitive.
\end{enumerate}
\end{lemma}

\begin{proof}
Let $E\subseteq F$ be the field generated by $R$. By the Galois
correspondence $E=F$ if and only if no $\sigma\in G$ fixes $E$. This
implies (a).

For (b) notice that the coefficients of each $f_i(X) =
\prod_{\alpha\in R_i}(X-\alpha)$ are $G$-invariant, hence
$f_i\in K[X]$. Furthermore, if $f_i$ were reducible, by Galois
correspondence, there would been a proper nonempty subset of
$R_i$ which is $G$-invariant, a contradiction. (c) is a special
case of (b).
\end{proof}

Let us describe the second type of actions. Let $H\leq G$ be a
subgroup. Then $G$ acts naturally on the cosets $G/H$ of $H$ in $G$
by left multiplication. This action is (non-canonically)
isomorphic\footnote{If $G$ is an infinite group, then we assume that
it is profinite and that $H$ is open in $G$.} to the previous
action. Namely, let $E = F^H$ be the fixed field of $H$ in $F$, let
$\alpha_1,\ldots,\alpha_n\in E$ be a generating set of $E/K$, i.e.\
$E=K(\alpha_1,\ldots,\alpha_n)$, and let $f(X)\in K[X]$ be the
minimal monic polynomial with $\alpha_1,\ldots,\alpha_n$ as roots.
Then $G$ permutes the roots of $f(X)$, and, by the Galois
correspondence, $\gal(F/E)=H$. This implies that the stabilizer of
$\alpha_1,\ldots,\alpha_n$ is $H$.

\section{Embedding Problems}\label{sec:EP}
A profinite group is defined as an inverse limit of finite groups,
or equivalently as a totally disconnected compact Hausdorff group
\cite[Lemma 1.1.7]{FriedJarden2005}. The former description makes it
plausible to characterize a profinite group $\Gamma$ via finite
objects. If $\Gamma$ is finitely generated\footnote{in the
topological sense.}, then it is determined (up-to-isomorphism) by
the set
\[
\{\Gamma/N \mid N \mbox{ is open and normal in }
\Gamma\}/\textnormal{isom.}
\]
of all finite quotients \cite[Proposition 16.10.7]{FriedJarden2005}.

However in general this set of finite quotient does not
determine $\Gamma$. For example, the direct product $\Gamma_1 =
\prod_{G} G$ of all finite groups $G$ and the free profinite
group of countable rank $\Gamma_2 = \Fhat_\omega$, both have all
finite groups as quotients, but $\Gamma_1\not\cong \Gamma_2$
(since, e.g., $\Gamma_2$ has no torsion). Therefore a more
refined finite object -- \emph{finite embedding problems} -- is
needed.

Let us present two classical results before going into details.
\begin{itemize}
\item
A profinite group $\Gamma$ is projective if and only if every finite
embedding problem is weakly solvable
\cite[Lemma~7.6.1]{RibesZalesskii}.
\item
Let $\Gamma$ be a profinite group of infinite rank $\kappa$. Then
$\Gamma$ is free if and only if every non-trivial finite embedding
problem has $\kappa$ solutions
\cite[Theorem~25.1.7]{FriedJarden2005}.
\end{itemize}

\subsection{Embedding Problems for Profinite Groups}
Let $\Gamma$ be a profinite group. An \textbf{embedding problem} for
$\Gamma$ is a diagram
\[
\xymatrix{%
    &\Gamma\ar[d]^{\mu}\ar@{.>}[dl]_{\exists\theta?}\\
G\ar[r]^\alpha
    &A,
}%
\]
where $G$ and $A$ are profinite groups and $\mu$ and $\alpha$ are
(continuous) epimorphisms. In short we write $(\mu,\alpha)$.

A \textbf{solution} of $(\mu,\alpha)$ is an epimorphism
$\theta\colon \Gamma\to G$ such that $\alpha\theta = \mu$. If
$\theta$ satisfies $\alpha\theta = \mu$ but is not necessarily
surjective, we say that $\theta$ is \textbf{weak solution}. In
particular, a profinite group $G$ is a quotient of $\Gamma$ if and
only if the embedding problem $(\Gamma\to 1, G\to 1)$ is solvable.

If $G$ is finite (resp.\ $\alpha$ splits), we say that the embedding
problem is \textbf{finite} (resp.\ \textbf{split}).

Two embedding problems $(\mu\colon \Gamma\to A, \alpha\colon G\to
A)$ and $(\nu \colon \Gamma \to B, \beta \colon H\to B)$ are said to
be \textbf{equivalent} if there exist isomorphisms $i\colon G\to H$
and $j\colon A\to B$ for which the following diagram commutes.
\[
\xymatrix{%
G\ar[r]^{\alpha} \ar[d]^i
    &A\ar[d]^j
        &\Gamma\ar@{=}[d]\ar[l]_{\mu}\\
H\ar[r]^{\beta}
    &B
        &\Gamma\ar[l]_{\nu}
}%
\]
It is evident that any (weak) solution of $(\mu,\alpha)$ corresponds
to a (weak) solution of $(\nu,\beta)$ and vice versa.

The following lemma gives an obvious, but useful, criterion for a
weak solution to be a solution (i.e.\ surjective).

\begin{lemma}\label{lem:criterionforproperness}
A weak solution $\theta\colon \Gamma \to G$ of an embedding problem
$(\mu\colon \Gamma\to A, \alpha \colon G\to A)$ is surjective if and
only if $\ker(\alpha) \leq \theta(\Gamma)$.
\end{lemma}

\begin{proof} Suppose $\ker(\alpha)\leq \theta(\Gamma)$.
Let $g\in G$, put $a=\alpha(g)$, and let $f\in \mu^{-1} (a)$. Then
$\theta(f)^{-1} g \in \ker(\alpha)\leq \theta(\Gamma)$, and hence
$g\in \theta(\Gamma)$. The converse is obvious.
\end{proof}

\subsection{Embedding Problems for Fields}
Let $K$ be a field with a separable closure $K_s$. The absolute
Galois group $\gal(K) = \gal(K_s/K)$ is profinite. This defines the
notion of an embedding problem for $K$.

More precisely, an embedding problem for a field $K$ is an embedding
problem \[(\mu\colon \gal(K) \to A, \alpha\colon G\to A)\] for
$\gal(K)$. If we denote by $L$ the fixed field of $\ker(\mu)$, then
$\mu$ factors as $\mu = \mugag\mu_0$, where $\mu_0\colon \gal(K)\to
\gal(L/K)$ is the restriction map and $\mugag\colon \gal(L/K)\to A$
is an isomorphism. Then the embedding problems $(\mu,\alpha)$ and
$(\mu_0,\mugag^{-1}\alpha)$ are equivalent.
So, from now on, we shall assume that $A = \gal(L/K)$ and $\mu$ is
the restriction map (unless we explicitly specify differently).
\begin{equation}\label{EP:basic}
\xymatrix{%
    &\gal(K)\ar[d]^{\mu}\ar@{.>}[dl]_{\exists\theta?}\\
G\ar[r]^(0.35){\alpha}
    &\gal(L/K)
}%
\end{equation}
Another piece of notation is that of solution fields: Let
$\theta\colon \gal(K) \to G$ be a weak solution of $(\mu,\alpha)$.
The fixed field of $\ker(\theta)$ is called the \textbf{solution
field}.

In terms of fields, a weak solution of an embedding problem
corresponds to a $K$-embedding of the field $L$ inside the
solution field $F$ in such a way that the restriction map
$\gal(F/K)\to \gal(L/K)$ coincides with $\alpha$. In particular,
the embedding problems $(\mu,\alpha)$ and the embedding problem
defined by the restriction map, i.e.\ $(\mu, \res\colon
\gal(F/K) \to \gal(L/K))$, are equivalent.

%-----------------------------------------------------------------------
\subsection{Rational and Geometric Embedding Problems}
We define embedding problems coming from geometric objects.

\begin{definition}
Let $E$ be a finitely generated regular extension of $K$, let $F/E$
be a Galois extension, and let $L = F\cap K_s$, where $K_s$ is a
separable closure of $K$). Then the restriction map $\alpha \colon
\gal(F/E)\to \gal(L/K)$ is surjective, since $E\cap K_s = K$.
Therefore
\begin{equation}\label{EP:fields}
\xymatrix{%
    &\gal(K)\ar[d]^{\mu}\\
\gal(F/E)\ar[r]^{\alpha}
    &\gal(L/K)
}%
\end{equation}
is an embedding problem for $K$. We call such an embedding problem
\textbf{geometric embedding problem}.

If $E=K(t_1,\ldots,t_e)$ is a field of rational functions over $K$
then we call \eqref{EP:fields} which gets the form
\begin{equation}\label{EP:regularlysolvable}
\xymatrix{%
    &\gal(K)\ar[d]^{\mu}\\
\gal(F/K(\bft))\ar[r]^{\alpha}
    &\gal(L/K)
}%
\end{equation}
\textbf{rational embedding problem}.
\end{definition}

It seems that the rationality of an embedding problem depends on
$e\geq 1$. However Lemma~\ref{lem:regularsolvablewithonet} shows
that if the condition is satisfied for some $e\geq 1$, then it also
holds for $e=1$.

\begin{lemma}
Every finite embedding problem is equivalent to a geometric
embedding problem.
\end{lemma}

See \cite[Lemma~11.6.1]{FriedJarden2005} for the proof of this
result. Now we can restrict and discuss only geometric embedding
problems.

We say that a finite group is \textbf{regularly realizable} if there
exists a Galois extension $F/K(\bft)$ such that $F$ is regular over
$K$ and $G\cong \gal(F/K(\bft))$. In other words, in
\eqref{EP:regularlysolvable} we have $L=K$. The following classical
result asserts that certain families of finite groups are regularly
realizable over any field:

\begin{theorem}\label{thm:regulargroups}
All abelian groups, all symmetric groups, and all alternating groups
are regularly realizable over any field.
\end{theorem}

A proof can be found e.g.\ in \cite{FriedJarden2005}. (This result
is by no mean the state of art.)

\section{Fiber Products}
Let $\alpha_1\colon G_1\to A$ and $\alpha_2 \colon G_2\to A$ be
epimorphisms of (profinite) groups. Then the fiber product is
defined by
\[
G_1\times_A G_2 = \{ (g_1,g_2)\in G_1\times G_2 \mid \alpha_1(g_1) =
\alpha_2(g_2)\}\leq G_1\times G_2.
\]
It is equipped with two natural projection maps $\pi_i\colon
G_1\times_A G_2 \to G_i$ defined by $\pi_i(g_1,g_2) = g_i$, $i=1,2$.

Fiber products rise up `naturally' in many situations. We shall use
them mainly to find a convenient embedding problem that dominates a
given embedding problem.
\begin{definition}
An embedding problem $(\mu'\colon \Gamma\to A',\alpha'\colon G'\to
A')$ for a profinite group $\Gamma$ \textbf{dominates} an embedding
problem $(\mu\colon \Gamma\to A, \alpha\colon G\to A)$ if every
(weak) solution $\theta'$ of $(\mu',\alpha')$ induces a (weak)
solution $\theta$ of $(\mu,\alpha)$ via a commutative diagram
\begin{equation}\label{eq:dom}
\xymatrix{%
    &\Gamma\ar[d]_{\mu'}\ar@/^/[dd]^{\mu}\ar@{.>}[dl]_{\theta'}\\
G'\ar[d]^\pi\ar[r]^{\alpha'}
    &A'\ar[d]_{\nu}\\
G\ar[r]^\alpha
    &A.
}%
\end{equation}
Here $\pi\colon G'\to G$ is surjective.
\end{definition}

The following lemma gives a dominating embedding problem via fiber
products.

\begin{lemma}\label{lem:dominatin_fiber_product}
Let $(\mu \colon \Gamma\to A, \alpha\colon G\to A)$ be an embedding
problem for a profinite group $\Gamma$. Let $\mu'\colon \Gamma\to
A'$ be an epimorphism such that $\ker(\mu')\leq \ker(\mu)$. Let $G'
= G\times_A A'$ and let $\alpha'\colon G' \to A'$ be the
corresponding projection map. Then the embedding problem $(\mu',
\alpha')$ dominates $(\mu,\alpha)$. Moreover
\begin{enumerate}
\item
Assume that $\theta\colon \Gamma\to G$ is a weak solution and
$\ker(\mu')\leq \ker(\theta)$. Then $(\mu',\alpha')$ splits. (In
particular, if $(\mu,\alpha)$ splits, then so does
$(\mu',\alpha')$.)
\item
A weak solution $\theta$ of $(\mu,\alpha)$ induces a weak solution
$\theta'$ of $(\mu',\alpha')$ defined by $\theta'(x) = (\theta(x),
\mu'(x))$.
\end{enumerate}
\end{lemma}

\begin{proof}
The commutative diagram \eqref{eq:dom}, where $\pi$ is the
projection map and $\nu$ is induced by $\mu'$ and $\mu$ implies that
$(\mu',\alpha')$ dominates $(\mu,\alpha)$, as needed.

(b) is trivial.

Finally we prove (a). The map $\theta\times \mu' \colon \Gamma\to
G'$ sends $x\in \Gamma$ to the pair $(\theta(x),\mu'(x))$. Let
$K=\ker(\theta\times \mu')$. Clearly $K = \ker(\theta)\cap
\ker(\mu')$ and hence by assumption $K=\ker(\mu')$. Therefore
$\theta\times \mu'$ induces a map $\beta'\colon A'\to G'$. Then
$\alpha'\beta'\mu' = \alpha' (\theta\times \mu') = \mu'$; hence
$\alpha'\beta'=\id$.
\end{proof}

A fiber product also describes the Galois group of the compositum of
two Galois extensions. The restriction maps are then realized as the
projection maps:

\begin{lemma}\label{lem:GaloisandFiber}
Let $M_1$ and $M_2$ be Galois extensions of a field $K$ and let
$L=M_1\cap M_2$. Then $\gal(M_1M_2/K)$ is canonically isomorphic to
$\gal(M_1/K) \times_{\gal(L/K)} \gal(M_2/K)$. Under this isomorphism
the restriction maps coincide with the projection maps.
\end{lemma}

\begin{proof}
The canonical isomorphism
\[
\gal(M_1M_2/K) \to \gal(M_1/K) \times_{\gal(L/K)} \gal(M_2/K)
\]
is given by $\sigma \mapsto (\sigma|_{M_1}, \sigma|_{M_2})$. (See
\cite[VI \S1 Theorem~1.14]{Lang2002}.)
\end{proof}

There is a connection between fiber products and independent
solutions of embedding problems. Let $\Psi = \{\theta_i\mid i\in I
\}$ be a family of solutions of a finite embedding problem
$(\mu\colon \Gamma \to A, \alpha\colon G\to A)$ for a profinite
group $\Gamma$. We call $\Psi$ \textbf{independent} if $\{\ker
(\theta_i)\mid i\in I\}$ is an independent family w.r.t.\ the Haar
measure of $\ker (\mu)$. That is to say, $\Psi$ is independent if
and only if
\[
(\ker(\mu):\bigcap_{i\in I_0}\ker(\theta_i) ) = \prod_{i\in I_0}
(\ker(\mu): \ker(\theta_{i}))
\]
for every finite subset $I_0\subseteq I$.

We denote
\[
G^I_A = \{(g_i)\in G^I \mid \alpha(g_i) = \alpha(g_j),\ \forall
i,j\in I\}
\]
for the fiber product of $I$ copies of $G$. For each $i\in I$ let
$\pi_i\colon G^I_A \to G$ be the $i$th projection map. We have an
epimorphism $\alpha_I \colon G^I_A \to A$ defined by $\alpha_I =
\alpha \pi_i$, for some $i\in I$. Clearly this definition is
independent of the choice of $i\in I$.

\begin{lemma}
\label{lem:independent_fiber} If $\theta \colon \Gamma \to G^I_A$ is
a solution of $(\mu , \alpha_I)$, then $\{\theta_i=\pi_i\theta\mid
i\in I\}$ is an independent family of solutions of $(\mu,\alpha)$.
\end{lemma}

\begin{proof}
For finite $I_0\subseteq I$ set $\pi_{I_0}\colon G^{I}_A \to
G^{I_0}_A$ and denote $\theta_{I_0} = \pi_{I_0}\theta$.
\[
\xymatrix{
    &\Gamma\ar[d]^{\mu}\ar[dl]_{\theta}\ar@/^7pt/[ddl]^{\theta_{I_0}}\\
G^I_A\ar[r]^{\alpha_I}\ar[d]_{\pi_{I_0}}
    &A\ar@{=}[d]\\
G^{I_0}_A\ar[r]^{\alpha_{I_0}}
    &A
}
\]
The above diagram with $I_0 =\{i\}$ implies that $\theta_i$ is a
solution, for any $i\in I$.

Since
\[
1 = \theta_{I_0}(x) = \pi_{I_0} (\theta(x)) \Longleftrightarrow
\forall i\in I_0,\ 1 = \pi_{i} (\theta(x)) =\theta_{i}(x),
\]
we get that $\bigcap_{i\in I_0}\ker(\theta_i) = \ker(\theta_{I_0})$.
Thus
\begin{eqnarray*}
(\ker(\mu):\bigcap_{i\in I_0}\ker(\theta_i) ) &=& (\ker(\mu):
\ker(\theta_{I_0})) = \frac{|G^{I_0}_A|}{|A|}  =
\left(\frac{|G|}{|A|}\right)^{|I_0|}\\ &=& \prod_{i\in I_0}
(\ker(\mu): \ker(\theta_{i})).
\end{eqnarray*}
This shows that $\{\theta_i\mid i\in I\}$ is independent family of
solutions, as needed.
\end{proof}

\section{Places}
In this section we recall some of the definitions and basic
properties of places, for reference see \cite{FriedJarden2005}.

A \textbf{place} $\phi$ of a field $F$ is a map of $\phi\colon F\to
M\cup \{\infty\}$ such that $M$ is a field and the following
properties hold.
\begin{enumerate}
\item $x+\infty = \infty + x = \infty$, $\forall x\in M$.
\item $x\cdot\infty = \infty \cdot x = \infty$, $\forall x\in M^\times$ ($:= M\smallsetminus \{0\}$).
\item $\phi(x+y) = \phi(x) + \phi(y)$ (whenever the right hand side is defined).
\item $\phi(xy) = \phi(x) \phi(y)$ (whenever the right hand side is defined).
\item $\exists x,y\in F$ such that $\phi(x) = \infty$ and
$\phi(y)\neq 0,\infty$.
\end{enumerate}
If $F$ and $M$ are extensions of a common field $K$, and $\phi(x) =
x$ for every $x\in K$, we say that $\phi$ is a \textbf{$K$-place}.

Every place $\phi$ of $F$ comes with a local ring $\calO_\phi =
\{x\in F \mid \phi(x)\neq \infty\}$ whose maximal ideal is
$\mfrak_\phi = \{x\in \calO_\phi \mid \phi(x) = 0\}$. The
quotient field $\calO_\phi/\mfrak_\phi$ is canonically
isomorphic to the \textbf{residue field} $\Fgag = \{\phi(x)\mid
x\in \calO_\phi\}$. Places are said to be \textbf{equivalent} if
they have the same local ring. In particular, equivalent places
have isomorphic residue fields.

Let $E/K$ be a regular extension and let $F/E$ be a finite Galois
extension. Assume that $\phi$ is a $K$-place of $E$. Then, by
Chevalley's lemma, the place $\phi$ extends to a place $\Phi\colon
F\to M\cup\{\infty\}$, where $M$ is the algebraic closure of $\Egag$
\cite[Proposition 2.3.1]{FriedJarden2005}. Assume from now on that
$\Fgag/\Egag$ is separable. Then $\Fgag/\Egag$ is, in fact, Galois
and the Galois group $\gal(\Fgag/\Egag)$ is then called the
\textbf{residue group}.

Let $L = F\cap K_s$, then $L\subseteq \Fgag$. (Indeed, since
$K\subseteq \calO_\Phi$, $L$ is integral over $K$, and $\calO_\Phi$
is integrally closed, it follows that $L\subseteq \calO_\Phi$.) By
composing $\Phi$ with an appropriate Galois automorphism of
$\Fgag/K$, we may assume that $\Phi$ is an $L$-place.

We call the subgroup
\[
D_{\Phi/\phi} = \{ \sigma\in \gal(F/E) \mid \sigma \calO_\Phi =
\calO_\Phi\} = \{ \sigma\in \gal(F/E) \mid \sigma \mfrak_\Phi =
\mfrak_\Phi\}
\]
of $\gal(F/E)$ the \textbf{decomposition group} of $\Phi/\phi$. The
fixed field of the decomposition group in $F$ is (as expected) the
\textbf{decomposition field} of $\Phi/\phi$. The decomposition field
is the maximal subextension of $F/E$ having the same residue field
as of $E$.

There is a natural epimorphism from the decomposition group to the
residue group, namely $\sigma \in D_{\Phi/\phi}\mapsto \sigmagag \in
\gal(\Fgag/\Egag)$, where $\sigmagag$ is defined by $\sigmagag(x +
\mfrak_\Phi) = \sigma x + \mfrak_\Phi$. The kernel of this
epimorphism is called the \textbf{inertia group}, and is denoted by
$I_{\Phi/\phi}$, viz.
\[
I_{\Phi/\phi} = \{ \sigma \in \gal(F/E) \mid x - \sigma (x) \in
\mfrak_\Phi, \forall x\in \calO_\Phi\}.
\]
If the inertia group is trivial, the place $\Phi$ is said to be
\textbf{unramified} over $\phi$ (or $\phi$ is unramified in $F$). In
other words, $\Phi$ is unramified over $\phi$ if and only if the
decomposition group is isomorphic to the residue group.

From now on we assume the $F,E$ are function fields. That is, we
assume that $F$ and $E$ are finite separable extensions of
$K(\bft)$, where $\bft=(t_1,\ldots,t_e)$ is a finite $e$-tuple of
variables for some $e\geq 1$.

\begin{lemma} \label{lem:RES}
Let $x\in F$ be such that $F = E(x)$. Then there exists a nonzero
$g(\bft)\in K[\bft]$ such that for every $K$-place $\phi$ of $F$ if
$\bfa = \phi(\bft)$ is finite and $g(\bfa)\neq 0$, then $\xgag =
\phi(x)$ is finite, $\phi$ is unramified over $E$, and $\Fgag =
\Egag(\xgag)$.
\end{lemma}

\begin{proof}
Let $f(\bft,X) \in K[\bft,X]$ be the irreducible polynomial of $x$
over $K(\bft)$. Let $g(\bft)\in K[\bft]$ be the product of the
leading coefficient of $f$ (as a polynomial in $X$) and its
discriminant. Since $f$ is separable, $g(\bft)\neq 0$.

Let $R = K[\bft, g(\bft)^{-1}]$ and let $S$, $U$ be the integral
closures of $R$ in $F$, $E$, respectively. Then $S = R[x]$
(\cite[Definition 6.1.3]{FriedJarden2005}), in particular $S =
U[x]$. Conclude that $S/U$ is a ring cover. Now \cite[Lemma
6.1.4]{FriedJarden2005} finishes the proof.
\end{proof}

Let $x\in F$ be a primitive element of $F/E$, i.e.\ $F=E(x)$ and let
$g(\bft)\in K[\bft]$ from the above lemma. Assume that $\bfa =
\phi(\bft)$ is finite and $g(\bfa)\neq 0$. Let $f$ be the
irreducible polynomial of $x$ over $E$.

Then $\fgag=\phi(f)$ is a separable polynomial (since the $\Phi$ is
unramified over $E$). Therefore $\Phi$ induces a bijective map
between the set $\calX$ of all roots of $f$ and the set $\calXgag$
of all roots of $\fgag$. We thus have an isomorphism $\rho\colon
S_\calX\to S_\calXgag$. Recall that $D_{\Phi/\phi}$ embeds into
$S_\calX$ (since $\gal(F/E)$ does) and that $\gal(\Fgag/\Egag)$
embeds into $S_\calXgag$. Then the following result holds (cf.\
\cite[Lemma 6.1.4]{FriedJarden2005}).

\begin{lemma}\label{lem:act_geo_sol}
The restriction of $\rho\colon S_\calX\to S_\calXgag$ to the
decomposition group $D_{\Phi/\phi}$ coincides with the natural map
$D_{\Phi/\phi}\to \gal(\Fgag/\Egag)$.
\end{lemma}

The above natural map relates to embedding problems in the following
way. An unramified place $\Phi/\phi$ for which $\Egag = K$ induces a
map $\Phi^* \colon \gal(K) \to \gal(F/E)$ which factors through the
residue group and whose image $\Phi^* (\gal(K)) = D_{\Phi/\phi}$ is
the decomposition group. We have
\[
\Phi(\Phi^*(\sigma) x) = \sigma \Phi(x),
\]
for all $\sigma\in \gal(K)$ and $x\in \calO_\Phi$. Thus if $\Phi$ is
an $L$-place, then $\res_{F,L}\circ \Phi^* = \res_{K_s,L}$. In other
words, $\Phi^*$ is a weak solution of the embedding problem
$$(\res\colon \gal(K)\to \gal(L/K), \res\colon \gal(F/E)\to
\gal(L/K)).$$ Such weak solutions are called \textbf{geometric} and
will be discussed below.

Let $\Psi$ be another place of $F$ lying over $\phi$, then $\Phi =
\Psi\sigma$ for some Galois automorphism $\sigma\in \gal(F/E)$, the
decomposition groups are conjugate (by the same $\sigma$) and
$\Psi^*$ and $\Phi^*$ differ by this conjugation. Therefore, to make
notation simpler, and when there is no risk of confusing, we omit
$\Phi$ from the notation. For example, $D_\phi$ stands for
$D_{\Phi/\phi}$ for some extension $\Phi$ of $\phi$ and $\phi^*$ is
actually $\Phi^*$ for some $\Phi$ lying above $\phi$; we shall say
that $\phi$ is unramified at $E$ if $\Phi/\phi$ is unramified, etc.

Let $\bft = (t_1\nek t_e)$ be an $e$-tuple of variables over $K$.
Any $a_1\nek a_e\in K_s$ defines a specialization $K[\bft] \to
K[\bfa]$. This specialization can be extended to a $K$-place $\phi$
of $K(\bft)$ into $K_s\cup\{\infty\}$. If $e=1$, the place $\phi$ is
uniquely determined by $a=a_1\in K_s$, and, in particular, the
residue field is $K(a)$.

However, if $e>1$, then $\phi$ is not uniquely determined by $\bfa$
and even the residue field need not be $K(\bfa)$. (The reason for
this is that an element such as $\phi(\frac{t_i-a_i}{t_j-a_j})$,
$i\neq j$, can be defined to be $\infty$ or any other element of
$K_s$.) Nevertheless, we can extend the specialization $\bft \to
\bfa$ to a $K$-place $\phi$ of $K(\bft)$ such that its residue field
is $K(\bfa)$, and moreover, for every finite extension $L/K$ the
residue field of $L(\bft)$ is $L(\bfa)$ (under any extension of
$\phi$ to an $L$-place of $L(\bft)$), see \cite[Lemma
2.2.7]{FriedJarden2005}.

\subsection{Points on Varieties and Places} \label{pntsplace}
Let $V$ be an algebraic variety defined over $K$. In this work
varieties always stay irreducible over the algebraic closure (that
is $V\tensor_K \Kgal$ is irreducible). Let $E$ denote the function
field of $V$. Let $\bfa\in V(K)$ and let $U=\Spec R$ be an affine
neighborhood of $\bfa$. Then $\bfa$ defines a $K$-homomorphism
$s_{\bfa}\colon R\to K$.

The homomorphism $s_{\bfa}$ can be extended to a $K$-place of $E$.
Similarly to the case of extending specializations to places, there
are several ways to extends $s_{\bfa}$ to $E$ when $V$ is not a
curve and it can be extended to a $K$-place of $E$ with residue
field $\Egag = K(\bfa)$.

Let $\nu\colon V \to \bbA^{\dim(V)}$ be a finite separable morphism
over $K$. Then the corresponding field extension $E/K(\bft)$ is
finite and separable. Let $\bfa\in K_s^{\dim(V)}
(=\bbA^{\dim(V)}(K_s))$ and take $\bfb \in \nu^{-1}(\bfa)\subseteq
V(K_s)$. By Lemma~\ref{lem:RES}, there exists $0\neq g(\bft)\in
K[\bft]$ such that if $g(\bfa)\neq 0$, then the specialization
$\bft\mapsto\bfa$ can be extended to a $K$-place of $E$ such that
$\overline{K(\bft)} = K(\bfa)$ and $\Egag = K(\bfb)$.

\subsection{Geometric Solutions of Embedding Problems}
\label{sec:geometricsolution}%
Consider a geometric embedding problem
$$
(\mu\colon \gal(K)\to \gal(L/K),\alpha\colon \gal(F/E) \to
\gal(L/K))
$$
for a field $K$. As we saw before, a place $\phi$ of $E$ that is
unramified in $F$ and with residue field $\Egag=K$ induces a weak
solution $\phi^*$ whose image is the decomposition group.
\begin{definition}
A weak solution $\theta\colon \gal(K) \to \gal(F/E)$ of a geometric
embedding problem $(\mu,\alpha)$ is called \textbf{geometric} if
there exists a $K$-place $\phi$ of $E$ unramified in $F$ such that
$\theta = \phi^*$.
\end{definition}

Geometric solutions are compatible with scalar extensions.

\begin{lemma}\label{lem:geometricdomination}
Let $(\mu,\alpha)$ be a geometric embedding problem. Let $M/K$ be a
Galois extension with $L\subseteq M$. Then the geometric embedding
problem
\[
(\mu'\colon \gal(K) \to \gal(M/K), \alpha'\colon \gal(FM/E)\to
\gal(M/K)),
\]
where $\alpha'$ and $\mu'$ are the corresponding restriction maps
dominates the embedding problem $(\mu,\alpha)$ with respect to the
restriction maps. Furthermore, if $\psi^*$ is a geometric (weak)
solution of $(\mu',\alpha')$, then $(\psi|_{F})^*$ is a geometric
(weak) solution of $(\mu,\alpha)$.
\end{lemma}

\begin{proof}
As $E/K$ is regular and by Lemma~\ref{lem:GaloisandFiber}, we have
\[\gal(FN/E) =
\gal(F/E)\times_{\gal(LE/E)}\gal(ME/E) \cong
\gal(F/E)\times_{\gal(L/K)} \gal(M/K)
\]
and the projection maps coincide with the restriction maps. Thus
$(\mu',\alpha')$ dominates $(\mu,\alpha)$
(Lemma~\ref{lem:dominatin_fiber_product}). Now let $\psi^*$ be a
geometric (weak) solution of $(\mu',\alpha')$. For $\phi =
\psi|_{F}$, we have that $\phi$ is unramified in $F$ and
$\res_{FM,F}\circ \psi^* = \phi^*$, as needed.
\end{proof}

Combination of Matsusaka-Zariski Theorem and Bertini-Noether Lemma
reduces the transcendence degree in \eqref{EP:regularlysolvable} to
$1$.

\begin{lemma}\label{lem:regularsolvablewithonet}
Let $K$ be an infinite field, $L/K$ a finite Galois extension, and
$(u,\bft)$ an $e+1$-tuple of algebraically independent elements over
$K$. Consider a rational embedding problem
\eqref{EP:regularlysolvable}. Then there exists a solution of
\[
(\mu_u\colon \gal(K(u)) \to \gal(L/K), \alpha\colon
\gal(F/K(\bft))\to \gal(L/K))
\]
whose solution field is regular over
$L$.
\end{lemma}

\begin{proof}
Let $x\in F$ be integral over $L[\bft]$ such that $F=K(\bft,x)$. Let
$f(\bfT,X)\in K[\bfT,X]$ be the absolutely irreducible polynomial
that is separable and monic in $X$ for which $f(\bft,x)=0$.

Take two variables $U,V$. Matsusaka-Zariski Theorem
\cite[Proposition 10.5.4]{FriedJarden2005} implies that
$f(T_1,\ldots, T_{e-1}, U + V T_1,X)$ is irreducible in the ring
$\Lgal [T_1,\ldots,T_{e-1} , X]$, where $\Lgal$ is an algebraic
closure of $L(U,V)$. Let $h(T_1,\ldots,T_r)\in K[\bfT]$ be the
non-zero polynomial given in Lemma~\ref{lem:RES} with respect to
the extension $F/K(\bft)$.

By Bertini-Noether Lemma \cite[Proposition 10.4.2]{FriedJarden2005}
there exists a nonzero $c(U,V)\in L[U,V]$ such that for any
$\alpha_e,\beta_e\in K$ satisfying $c(\alpha_e,\beta_e)\neq 0$ the
monic polynomial $f(T_1,\ldots,T_{e-1}, \alpha_e + \beta_e T_1 , X)$
remains absolutely irreducible over $K$. Since $K$ is infinite,
there exist $\alpha_e,\beta_e\in K$ with $\beta_e\neq 0$ such that
$c(\alpha_e,\beta_e)\neq 0$ and $h(T_1,\ldots,\alpha_e+\beta_e
T_1)\neq 0$. Induction on $e$ yields $\alpha_i,\beta_i\in K$,
$\beta_i\neq 0$, $i=2,\ldots, e$ such that $g(T,X) = f(T, \alpha_2 +
\beta_2 T,\ldots, \alpha_e + \beta_e T,X)$ is an absolutely
irreducible polynomial and $h(T, \alpha_2 + \beta_2 T,\ldots,
\alpha_e + \beta_e T,X)\neq 0$.

Extend the specialization $\bft\mapsto(u,\alpha_2 + \beta_2u
,\ldots, \alpha_e + \beta_e u)$ to an $L$-place $\phi$ of
$F/K(\bft)$ and denote by $F_0/K(u)$ the residue field extension
\cite[Lemma 2.2.7]{FriedJarden2005}. By Lemma~\ref{lem:RES}, $F_0 =
L(u,x_0)$, where $x_0=\phi(x)$ is a root of $g(u,X)$. Hence $F_0$ is
regular over $L$. The place $\phi$ also induces a geometric weak
solution $\phi^*\colon \gal(K(u)) \to \gal(F/K(\bft))$ of
$(\mu_u,\alpha)$ whose residue field is $F_0$. But
\begin{eqnarray*}
[F_0:K(t)] &=& [F_0:L(t)][L:K] = \deg_X g(t,X) [L:K]=\deg_X
f(\bft,X)[L:K]\\
&=&[F:L(\bft)][L:K] = [F:K(\bft)],
\end{eqnarray*}
so $\phi^*$ is surjective, and thus the assertion follows.
\end{proof}

\section{Wreath Products}\label{sec:wreathproduct}
In this section we introduce wreath products, and the more general
notion of twisted wreath products.

\subsection{Definition}
Let $A$, $G_0 \leq G$ be finite groups. Assume that $G_0$ acts on
$A$ (from the right). Let
\[
\Ind_{G_0}^G(A) = \{ f\colon G\to A\mid f(\sigma\rho) =
f(\sigma)^\rho,\ \forall \sigma\in G, \rho\in G_0\} \cong
A^{(G:G_0)}.
\]
Here multiplication is component-wise, i.e.\ $(fg)(\sigma) =
f(\sigma) g(\sigma)$. Then $G$ acts on $\Ind_{G_0}^G(A)$ by
$f^\sigma(\tau) = f(\sigma \tau)$. We define the \textbf{twisted
wreath product} to be the semidirect product
\[
A\wr_{G_0} G = \Ind_{G_0}^G(A) \rtimes G,
\]
i.e., an element in $A\wr_{G_0} G$ can be written uniquely as
$f\sigma$, where $f\in \Ind_{G_0}^G(A)$ and $\sigma\in G$. The
multiplication is then given by $(f\sigma) (g\tau) = f
g^{\sigma^{-1}} \sigma\tau$. The twisted wreath product is equipped
with the quotient map $\alpha\colon A\wr_{G_0} G\to G$ defined by
$\alpha(f\sigma) = \sigma$.

If $G_0=1$, the twisted wreath product is simply called
\textbf{wreath product}. We abbreviate $A\wr_{1} G$ and write $A\wr
G$ for the wreath product.

The next result states that $A\rtimes G_0$ embeds in $A\wr_{G_0} G$.
\begin{lemma}\label{lem:embeddinginwreath}
For each $a\in A$ let $f_a \in \Ind_{G_0}^G(A)$ be defined by
\[
f_a(\sigma)=\begin{cases}a^\sigma & \sigma\in G_0\\1 &
\mbox{otherwise.}\end{cases}
\]
Then the map $\rho\colon A\rtimes G_0 \to A\wr_{G_0} G$ defined by
$\rho(a\sigma) = f_a \sigma$ is a monomorphism.
\end{lemma}

\begin{proof}
Clearly $\rho$ is injective. Since $(f_a)^\tau = f_{a^\tau}$ for any
$a\in A$ and $\tau\in G_0$, the map $\rho$ is a homomorphism.
\end{proof}

\subsection{Twisted Wreath Products and Embedding Problems}
The following result strengthens
Lemma~\ref{lem:criterionforproperness} in the case of embedding
problems with twisted wreath products.

\begin{lemma}\label{lem:criterionforpropernesswreath}
Let $(\phi\colon \Gamma\to G, \alpha\colon A\wr_{G_0} G \to G)$ be a
finite embedding problem for a profinite group $\Gamma$ and let
$\theta\colon \Gamma\to A\wr_{G_0} G$ be a weak solution. Assume
that $A = \{ f_a\mid a\in A \}\leq \theta(\Gamma)$. Then $\theta$ is
surjective.
\end{lemma}

\begin{proof}
The image of $\theta$ is $G$-invariant, hence $A^\sigma \leq
\theta(\Gamma)$ for every $\sigma \in G$. Therefore
\[\Ind_{G_0}^G(A) = \prod_{\sigma\in G} A^\sigma \leq \theta(\Gamma). \]
But $\Ind_{G_0}^G(A)=\ker(\alpha)$; hence
Lemma~\ref{lem:criterionforproperness} implies that $\theta$ is
surjective.
\end{proof}

There is a close relation between solutions of embedding problem
and twisted wreath products. An application of this connection
is Haran's diamond theorem \cite{Haran1999JGroupTheory}. We
bring below a simple result in that spirit, but first we need a
simple preparation.

We claim that the map $\pi\colon \Ind_{G_0}^G(A)\rtimes G_0\to
A\rtimes G_0$ defined by $\pi(f\sigma) = f(1)\sigma$ is an
epimorphism: It is obvious that $\pi$ is surjective. For any $f\in
\Ind_{G_0}^G(A)$ and $\sigma\in G_0$ we have
\[
\pi(f^\sigma)=f(\sigma)=f(1)^\sigma = \pi(f)^\sigma,
\]
and hence $\pi$ is also a homomorphism. We call $\pi$ the
\textbf{Shapiro map}. Note that the map $\rho$ defined in
Lemma~\ref{lem:embeddinginwreath} is a section of $\pi$, i.e.\
$\pi \rho$ is the identity map on $A\rtimes G_0$.

\begin{lemma}\label{lem:wreath_solution}
Let $\Gamma\leq \Lambda$ be profinite groups. Let $\nu\colon \Lambda
\to G$ be an epimorphism onto a finite group, let $G_0=\nu(\Gamma)$,
and set $\mu=\nu|_{\Gamma}$. Assume $G_0$ acts on a finite group
$A$, and let $\theta$ be a weak solution of $(\nu\colon \Lambda \to
G, \alpha\colon A\wr_{G_0} G\to G)$. Then $\eta = \pi
\theta|_{\Gamma}$ is defined and is a weak solution of $(\mu\colon
\Gamma\to G_0, \beta\colon A\rtimes G_0\to G_0)$.
\end{lemma}

\begin{proof}
Note that since $\alpha\theta(\Gamma)=\nu(\Gamma)=G_0$, we have
$\theta(\Gamma)\leq \Ind_{G_0}^G(A)\rtimes G_0$. So $\eta = \pi
\theta|_{\Gamma}$ is well defined. It is evident that $\eta$ is a
weak solution of $(\mu,\beta)$.
\end{proof}

We shall need the following technical result.

\begin{lemma}\label{lem:embedding in wreath}
Let $A\wr_{G_0} G$ be a twisted wreath product of finite groups and
denote by $\pi\colon \Ind_{G_0}^G(A) \rtimes G_0 \to A\rtimes G_0$
the corresponding Shapiro map. Let $i \colon G_0 \to A\rtimes G_0$
be some splitting of the quotient map $\alpha_0\colon A\rtimes
G_0\to G_0$. Then there exists a splitting $j\colon G\to A\wr_{G_0}
G$ of the quotient map $\alpha\colon A\wr_{G_0} G\to G$ such that
$j(\sigma) \in \Ind_{G_0}^G(A)\rtimes G_0$ and  $\pi (j(\sigma)) =
i(\sigma)$ for every $\sigma\in G_0$.
\end{lemma}

\begin{proof}
For each $\sigma\in G_0$ let $a_\sigma = i(\sigma) \sigma^{-1}$,
i.e.\ $i(\sigma) = a_\sigma \sigma$. Since $i$ is a homomorphism,
for every $\sigma,\tau\in G_0$ we have
\begin{eqnarray}
\label{eq:wre2}&&a_{\sigma\tau} = a_\sigma a_\tau^{\sigma^{-1}}.
\end{eqnarray}
(Note that substituting $\sigma=\tau=1$ we get that $a_1 = 1$.)
Choose a set of representatives $R$ of left cosets of $G_0$ in $G$,
i.e.\ any $\sigma\in G$ uniquely factorizes as $\sigma = \sigma'
\rho$, $\sigma'\in G_0$ and $\rho\in R$. Assume that $1\in R$. Then
define $a_\sigma = a_{\sigma'}$.

We note that \eqref{eq:wre2} holds, under the extended definition,
for every $\sigma\in G_0$ and $\tau\in G$. Indeed, assume $\tau =
\tau' \rho$, $\tau'\in G_0$ and $\rho\in R$. Then $a_{\tau} =
a_{\tau'}$ and $a_{\sigma\tau} = a_{\sigma \tau'}$. Now using
\eqref{eq:wre2} with $\sigma,\tau'\in G_0$ we get
\[
a_{\sigma\tau} = a_{\sigma\tau'} = a_\sigma a_{\tau'}^{\sigma^{-1}}
= a_\sigma a_{\tau}^{\sigma^{-1}}.
\]

It suffices to construct $f_\sigma\in \Ind_{G_0}^G(A)$ for each
$\sigma \in G$ such that
\begin{eqnarray}
&&f_\sigma(1) = a_\sigma,\label{eq:wre3}\\
&&f_{\sigma\tau} (\rho)= f_\sigma (\rho)
f_\tau(\sigma^{-1}\rho)\label{eq:wre4}
\end{eqnarray}
for all $\sigma,\tau\in G$. Indeed, assume we done that. Then we set
$j(\sigma) = f_\sigma \sigma$ and we have $j(\sigma) \in
\Ind_{G_0}^G(A)\rtimes G_0$ for $\sigma\in G_0$. Now \eqref{eq:wre4}
implies that $j$ is a homomorphism and \eqref{eq:wre3} implies that
$\pi(j(\sigma)) = i(\sigma)$ for all $\sigma\in G_0$. (Note that
substituting $\tau = \rho = 1$ in \eqref{eq:wre4} one gets
$f_1(\sigma^{-1}) = f_{\sigma}(1)^{-1} f_{\sigma}(1) = 1$, i.e.\
$f_1 = 1$.)

Let $f_\tau(\sigma) = a_{\sigma^{-1}}^{-1} a_{\sigma^{-1} \tau}$
(this definition comes from \eqref{eq:wre4} with $\rho = 1$ and
\eqref{eq:wre3}). Then, clearly, \eqref{eq:wre3} holds. For
$\sigma,\tau,\rho\in G$ we have
\[
f_\sigma (\rho) f_\tau(\sigma^{-1}\rho)=a_{\rho^{-1}}^{-1}
a_{\rho^{-1}\sigma} a_{\rho^{-1}\sigma}^{-1} a_{\rho^{-1}\sigma\tau}
= a_{\rho^{-1}}^{-1} a_{\rho^{-1}\sigma\tau} = f_{\sigma\tau}(\rho),
\]
hence \eqref{eq:wre4} holds. Next let $\sigma,\tau \in G$ and
$\rho\in G_0$. By \eqref{eq:wre2}, we get that
\[
f_{\sigma}(\tau \rho) = a_{\rho^{-1} \tau^{-1}}^{-1} a_{\rho^{-1}
\tau^{-1} \sigma} = (a_{\rho^{-1}} a_{\tau^{-1}}^{\rho})^{-1}
a_{\rho^{-1}} a_{\tau^{-1} \sigma}^{\rho} = (a_{\tau^{-1}}^{-1}
a_{\tau^{-1} \sigma})^{\rho} = (f_{\sigma}{\tau})^\rho,
\]
that is to say, $f_\sigma \in \Ind_{G_0}^G(A)$, as needed.
\end{proof}

\begin{proposition} \label{prop:embeddingtwotowreath}
Let $A$ and $H$ be finite groups, let $H_0,G$ be subgroups of $H$,
and let $G_0 = G\cap H_0$. Assume that $H_0$, and hence also $G_0$,
acts on $A$ and that there exists a splitting $i\colon G_0 \to
A\rtimes G_0\leq A\rtimes H_0$ of the projection map $A\rtimes
G_0\to G_0$.
Then there exists an embedding $j\colon G \to A\wr_{H_0} H$ such
that the diagram
\[
\xymatrix{%
G_0 \ar[d]_{j|_{G_0}} \ar[dr]^{i}\\
\Ind_{H_0}^H(A)\rtimes H_0\ar[r]^(0.6){\pi}
    &A\rtimes H_0
}%
\]
commutes. (Here $\pi$ is the Shapiro map.)
\end{proposition}

\begin{proof}
Let $f\in \Ind_{G_0}^G(A)$. Note that $G/G_0$ naturally embeds
into $H/H_0$ by mapping $\sigma G_0$ to $\sigma H_0$ and its
image is the set $(G\cdot H_0)/H_0$. Extend $f$ to $\fgal\colon
H \to A$ by setting $\fgal(\sigma h) = f(\sigma)^h$ if
$\sigma\in G$, $h \in H_0$ and $f(\sigma)=1$ if $\sigma\in H
\smallsetminus (G\cdot H_0)$. Then $\fgal \in \Ind_{H_0}^H(A)$.
Moreover the map $\phi\colon A\wr_{G_0} G \to A\wr_{H_0} H$
defined by $\phi(f \sigma) = \fgal \sigma$, $f\in
\Ind_{G_0}^G(A)$, $\sigma\in G$ is an embedding.

\[
\xymatrix{ G\ar[r]_{j'}\ar@/^10pt/[rr]^j   & A\wr_{G_0} G
\ar[r]_\phi
      &A\wr_{H_0} H \\
G_0\ar[u]\ar[r]_(0.3){j'}   & {\Ind_{G_0}^G (A)} \rtimes G_0 \ar[r]
\ar[u] \ar[d]^{\pi_0}
      & {\Ind_{H_0}^H(A)} \rtimes H_0 \ar[u] \ar[d]^\pi\\
G_0 \ar[r]^i \ar@{=}[u]
   & A\rtimes G_0\ar[r]
      &A\rtimes H_0
}
\]

Now, by the previous lemma, for $i\colon A\rtimes G_0$ there exists
$j'\colon G\to A\wr_{G_0} G$ such that $\pi_0(j'(\sigma))=i(\sigma)$
for all $\sigma\in G_0$. (Here $\pi_0$ is the Shapiro map w.r.t.\
$A\wr_{G_0} G$.) Let $j = \phi j'$ we show below that
$\pi(j(\sigma)) = i(\sigma)$ for all $\sigma \in G_0$.

Indeed, let $\sigma\in G_0$. Denote $j'(\sigma) = f\sigma$, $f\in
\Ind_{G_0}^G(A)$; then $j(\sigma) = \phi(f\sigma) = \fgal \sigma$.
Thus
\[
\pi(j(\sigma)) = \fgal(1)\sigma = f(1)\sigma = \pi_0(f \sigma) =
\pi_0(j'(\sigma))=i(\sigma),
\]
as needed.
\end{proof}

\subsection{Twisted Wreath Product in Fields}

\begin{definition}
Let $\Fhat/ K$ be a Galois extension whose Galois group is
$A\wr_{G_0} G$. Set $I = \Ind_{G_0}^G(A)$. To the chain of subgroups
\[
1\leq \{ f\in I \mid f(1)=1\} \leq I \leq I\rtimes G_0 \leq
A\wr_{G_0} G
\]
there corresponds a tower of fields (in the inverse order)
\begin{equation}\label{eq:towerwreath}
K\subseteq L_0 \subseteq L \subseteq F \subseteq \Fhat.
\end{equation}
In particular, $G_0 = \gal(L/L_0)$ and $G = \gal(L/K)$. Then we say
that \eqref{eq:towerwreath} \textbf{realizes} $A\wr_{G_0} G$.
\end{definition}

\begin{remark}
Consider an embedding problem $(\mu \colon \gal(K) \to G,
\alpha\colon A\wr_{G_0} G \to G)$, where $G=\gal(L/K)$ and $\mu$ is
the restriction map. Then \eqref{eq:towerwreath} realizes
$A\wr_{G_0} G$ implies that $\theta \colon \gal(K)\to \gal(\Fhat/K)$
is a solution.
\end{remark}

The following lemma, due to Haran \cite{Haran1999InventMath},
enables us to descend split embedding problems in terms of twisted
wreath products.

But first let us recall several facts. Let $M$ be a field, $N/M$ a
Galois extension, and $F/M(t)$ a Galois extension such that,
$N\subseteq F$, $F/N$ is regular, and $t$ is a transcendental
element over $M$. Now if the restriction map $\beta\colon
\gal(F/M(t)) \to \gal(N/M)$ splits, then $F = EN$, where $E$ is the
fixed field in $F$ of the image of $\gal(N/M)$ under some splitting
of $\beta$. Then $E/M$ is regular and $M(t)\subseteq E$.

Let $x\in E$ be an element for which $E = M(t,x)$ and let $f(t,X)\in
M[t,X]$ be its irreducible polynomial over $M(t)$. Then $f$ is
absolutely irreducible (since $E/M$ is regular) and $f$ is Galois
over $N(t)$ (since $F=EN=N(t,x)$).

\begin{lemma} \label{lem:wreath_regularlysolvable}
Let $M/K$ be a separable algebraic extension, $t$ a transcendental
element over $M$, and consider a rational finite split embedding
problem
\[
(\mu\colon \gal(M) \to \gal(N_1/M), \beta\colon \gal(F/M(u)) \to
\gal(N_1/M))
\]
for $M$. In particular, $F/N_1$ is regular. Let $f(u,X)\in M[u,X]$
be as above, i.e., $f$ is absolutely irreducible, Galois over
$N_1(u)$, and a root of which generates $F/N_1(u)$. Assume that
there exists a finite Galois extension $L/K$ satisfying
\begin{enumerate}
\item
$f(u,X)\in L_0[u,X]$, where $L_0=L\cap M$,
\item
$f(u,X)$ is Galois over $L(u)$, and
\item
$N_1 \subseteq N$, where $N = ML$.
\end{enumerate}
Set $G=\gal(L/K)$, $G_0 = \gal(L/L_0) \cong \gal(N/M)$, let
$\phi\colon \gal(K)\to G$ be the restriction map, and let
$A=\ker(\beta)=\gal(F/N_1(u))\cong \gal(FL/N(u))$. Then
\begin{equation}\label{eq:epinside a lemma 2} (\phi\colon \gal(K)
\to G, \alpha\colon A\wr_{G_0} G \to G)
\end{equation}
is rational. (Here $\gal(N_1/M)$ acts on $A$ via a splitting of
$\beta$ and $G_0$ acts on $A$ via the restriction map $G_0\to
\gal(N_1/M)$.)
\end{lemma}

\[
\xymatrix{
           &&F\\
    &M (u)\ar@{-}[r]^{G_0}
           &N(u)\ar@{-}[u]_A\\
    &M\ar@{-}[u]\ar@{-}[u]\ar@{-}[r]|{N_1}
           &N\ar@{-}[u]\\
K \ar@{-}[r]\ar@{.}@/_10pt/[rr]_{G}
    &L_0\ar@{-}[u]\ar@{-}[r]^{G_0}
           &L\ar@{-}[u]
}
\]

\begin{proof}
Let $c_1,\ldots, c_n$ be a basis of $L_0/K$ and let
$\bft=(t_1,\ldots,t_n)$ be an $n$-tuple of algebraically independent
elements over $L_0$. Then by \cite[Lemma~3.1]{Haran1999InventMath},
there exist fields $F_0,\Fhat_0$ such that
\begin{enumerate}
\item
$K(\bft)\subseteq L_0(\bft)\subseteq L(\bft)\subseteq F_0\subseteq
\Fhat_0$ realizes $A\wr_{G_0} G$,
\item
$\Fhat_0/L$ is regular, and
\item
$F_0=L(\bft)(z)$, where $\irr(z,L(\bft))= f(\sum_1^n c_it_i,Z)\in
L_0[\bft,X]$.
\end{enumerate}
In particular, $\theta\colon \gal(K(\bft))\to \gal(\Fhat_0/K(\bft))$
is a solution of \eqref{eq:epinside a lemma 2} whose solution field
is regular over $L$. This means that the embedding problem is
rational.
\end{proof}

\begin{remark}
The connection between solutions of the two embedding problems in
the above lemma is much deeper. In \cite{Haran1999InventMath}, Haran
establishes this connection and applies it to prove his diamond
theorem. This theorem states a general sufficient condition for a
separable extension of a Hilbertian field to be Hilbertian.
\end{remark}

\begin{corollary} \label{cor:wreath_regularlysolvable}
Let $K\subseteq L \subseteq M$ be a tower of separable algebraic
extensions such that $L/K$ is a finite Galois extension with a
Galois group $G = \gal(L/K)$. Let $A$ be a finite group which is
regularly realizable over $K$. Then the embedding problem
\[
(\res\colon \gal(K) \to G, \alpha \colon A\wr G \to G)
\]
is rational.
\end{corollary}

\begin{proof}
By assumption there exists a regular extension $F_0$ of $K$ and
$t\in F_0$ such that  $F_0/K(t)$ is a Galois extension with a Galois
group $A = \gal(F_0/K(t))$. Let $x\in F_0$ be a primitive element
and $f(t,X) = \irr (x,K(t))$. We use
Lemma~\ref{lem:wreath_regularlysolvable} with $F_1 = F_0 M$, $N_1 =
M$, and $L$ to get the assertion.
\end{proof}

\subsection{Permutational Wreath Product}\label{sec:permutationalwreath}
Often wreath products occur
in nature as permutation groups. We do not present here the most
general setting, for a more general definition see e.g.\
\cite{Meldrum1995}.

Let $A,G$ be finite groups. Assume that $A$ acts on some set $X$.
Then $A\wr G$ acts on $X\times G$ by the following rule.
\[
(f\sigma) (x,\tau) = (f(\sigma\tau) x , \sigma \tau), \quad \forall
f\sigma\in A\wr G,\ x\in X,\ \and \ \tau\in G.
\]
This action is well defined since
\begin{eqnarray*}
f\sigma f'\sigma' (x,\tau) &=& f\sigma (f'(\sigma'\tau)x,
\sigma'\tau) =
(f(\sigma\sigma'\tau)
f'(\sigma'\tau)x,\sigma\sigma'\tau)\\
&=& (ff'^{\sigma^{-1}}(\sigma\sigma'\tau) x, \sigma\sigma'\tau) =
ff'^{\sigma^{-1}}\sigma\sigma' (x,\tau).
\end{eqnarray*}

\begin{proposition}
If the action of $A$ on $X$ is faithful (resp.\ transitive), then so
is the action of $A\wr G$ on $X\times G$.
\end{proposition}

\begin{proof}
Clear.
\end{proof}

A group $A$ is of \textbf{degree} $n$ if $A$ acts faithfully on a
set $X$ with $n$ elements. The symmetry group $S_n$ is maximal with
respect to the property that any group of degree $n$ can be embedded
inside $S_n$ as a permutation group. In what follows we show that
$S_n\wr G$ is the analogous `maximal' permutation group if we
consider permutation groups with an epimorphism onto $G$ whose
kernel has degree $n$.

Let us be more precise. Let $\alpha\colon H\to G$ be an epimorphism
of finite groups. Let $X$ be a set of cardinality $n$. An action of
$H$ on $X\times G$ is said to be \textbf{fine} if it is transitive,
faithful, and any $h\in H$ maps (bijectively) $X\times\{\tau\}$ onto
$X\times \{\alpha(h) \tau\}$ for all $\tau\in G$. Therefore, any
$h\in H$ and $\tau\in G$ define a permutation $h_\tau\in S_X$ by the
formula
\[
h(x,\tau) = (h_\tau(x), \alpha(h) \tau).
\]
For $h_1,h_2\in H$, $g_1=\alpha(h_1),g_2=\alpha(h_2)\in G$, and
$(x,\tau)\in X\times G$ we have
\begin{eqnarray*}
((h_1h_2)_{\tau}(x),g_1g_2\tau)
    &=
        & (h_1h_2) (x,\tau) = h_1 ((h_2)_\tau(x),g_2\tau)\\
    &=
        & ((h_1)_{h_2\tau}(h_2)_{\tau}(x),g_1g_2\tau),
\end{eqnarray*}
and thus
\begin{equation}\label{eq:fineaction}
(h_1h_2)_{\tau} = (h_1)_{g_2\tau} (h_2)_\tau.
\end{equation}

Let $A$ be a transitive group of degree $n$ acting on $X$. Then the
permutational wreath product $A\wr G$ acts finely on $X\times G$.
The following lemma asserts that the wreath product is maximal
w.r.t.\ groups that act finely.

\begin{lemma}
Let $X=\{1,\ldots,n\}$, let $\alpha\colon S_n\wr G\to G$, and let
$\beta\colon H\to G$ be an epimorphism of finite groups. Assume that
$H$ acts finely on $X\times G$. Then there exists an embedding
$\nu\colon H\to S_n\wr G$ that respects the actions on $X\times G$
such that $\alpha \nu = \beta$.
\end{lemma}

\begin{proof}
Define $\nu\colon H \to S_n\wr G$ by setting $\nu(h) = f_h g$, where
$g=\beta(h)$ and $f_h(\sigma)(x) = h_{g^{-1}\sigma}(x)$, for $x\in
X$. It is obvious that $\alpha \nu =\beta$. Also $\nu$ respects the
action on $X\times G$. Indeed,
\[
\nu(h) (x,\tau) = (f_h(\beta(h)\tau) x, \beta(h)\tau) = (h_{\tau}
(x), \beta(h)\tau) = h(x,\tau).
\]
As the actions of $H$ and $S_n\wr G$ on $X\times G$ are faithful, we
get that $\nu$ is injective. It remains to show that $\nu$ is a
homomorphism.

Let $h_1,h_2\in H$ and $g_1= \alpha(h_1)$, $g_2 = \alpha(h_2)$.
Since $\nu(h_1h_2) = f_{h_1h_2}g_1g_2$ and $\nu(h_1)\nu(h_2) =
f_{h_1} f_{h_2}^{g_1^{-1}} g_1 g_2$, it suffices to verify that
$f_{h_1h_2} = f_{h_1} f_{h_2}^{g_1^{-1}}$. By \eqref{eq:fineaction}
we have
\begin{eqnarray*}
f_{h_1h_2}(\sigma)(x)
    &=
        & (h_1h_2)_{g_2^{-1}g_1^{-1} \sigma}(x) =
        (h_1)_{g_1^{-1} \sigma} (h_2)_{g_2^{-1}g_1^{-1} \sigma}(x)=\\
    &=
        &
        f_{h_1}(\sigma)f_{h_2}(g_1^{-1}\sigma)(x)=(f_{h_1}f_{h_2}^{g^{-1}})(\sigma)(x),
\end{eqnarray*}
as needed.
\end{proof}

\subsection{The Embedding Theorem}
The wreath product has the following interesting property that any
extension of $A$ and $G$ can be embedded in $A\wr G$. We will not
use this result here.

\begin{theorem}\label{thm:Embedding Theorem}
Let $\xymatrix@1{1\ar[r] & A\ar[r]&H\ar[r]^{\pi}& G \ar[r]&1}$ be an
exact sequence of groups. Then there exists an embedding $i\colon
H\to A\wr G$ such that $\alpha i = \pi$, where $\alpha\colon A\wr G
\to G$ is the projection map.
\end{theorem}

For a proof see \cite[Corollary 2.10]{Meldrum1995}.

%--------------------------------------------------------------------------------------------------------------------------------------

\chapter{Double Embedding Problems and PAC Extensions}
This chapter constitutes the technical backbone of the thesis. The
study of the Galois structure of a field $K$ can be carried out by
finite embedding problems. We introduce the notion of \emph{double
embedding problems} for a field extension $K/K_0$ which consists on
two compatible embedding problems -- one for $K$ and one for $K_0$.

It is known (although in different terminology) that over a PAC
field every solution of a geometric finite embedding problem is
geometric (see Subsection~\ref{sec:geometricsolution} for
definitions). We go in a parallel way, and characterize the PACness
of an extension in terms of geometric solutions of certain double
embedding problems.

Surprisingly, a stronger property is valid -- every solution of an
embedding problem for $K$ (under some regularity condition) can be
lifted to a \emph{geometric} solution of the whole double embedding
problem. This key property is called the \textbf{lifting property}
and it is the main result of the chapter. We also give a stronger,
but a bit more technical, lifting property for PAC extensions of
finitely generated fields.

This group theoretic approach proves to be extremely efficient. In
the following chapters we apply it to the study of the Galois
structure PAC extensions and other applications.

\section{Basic Properties of PAC Extensions}
Recall the definition of a PAC field.

\begin{definition}
A field $K$ is called \textbf{PAC} if every nonempty absolutely
irreducible variety defined over $K$ has a $K$-rational point.
\end{definition}

In \cite{JardenRazon1994}, Jarden and Razon introduce the more
general notion of PAC extensions:

\begin{definition}
A field extension $K/K_0$ is said to be \textbf{PAC} if for every
absolutely irreducible variety $V$ of dimension $e\geq 1$ defined
over $K$ and for every separable dominating rational map $\nu \colon
V \to \bbA^e$ there exists $\bfb\in V(K)$ such that $\bfa=\nu(\bfb)
\in K_0^e$.
\end{definition}

\begin{remark}
In fact, \cite{JardenRazon1994} considers a more general
setting, that is to say, it allows $K_0$ to be a subring or even
a subset of $K$, see Definition~\ref{def:PAC}.
\end{remark}

\begin{remark}
Note that $K_0$ must be infinite. Indeed, if $K_0$ were finite, then
$\nu^{-1}(K_0^e)$ would be also finite, and thus $\tilde V =
V\smallsetminus \nu^{-1}(K_0^e)$ would have no point $\bfb\in \tilde
V(K)$ satisfying $\nu(\bfa)\in K_0^e$.
\end{remark}

\begin{remark}
If an extension $K/K_0$ is PAC, then the field $K$ is obviously
PAC. In particular, $K$ is PAC if and only if $K/K$ is.

If $K$ is a $\bbZ_l$ extension of a finite prime field $\bbF_p$,
then $K$ is PAC \cite{FriedJarden2005}. However, any proper subfield
of $K$ is finite. Hence $K$ is not a PAC extension of any proper
subfield.

In zero characteristic there are also examples of PAC fields which
are PAC extensions of no proper subfields
(Corollary~\ref{cor:PACofnoSubField}).
\end{remark}

The following proposition gives several equivalent definitions of
PAC extensions in terms of polynomials and places, including a
reduction to plane curves. A proof of that proposition essentially
appears in \cite{JardenRazon1994}. Nevertheless, for the sake of
completeness, we give here a formal proof.

\begin{proposition}\label{prop:DefinitionPACextension}
The following conditions are equivalent for a field extension
$K/K_0$.
\begin{label1}
\item \label{con pac:DefinitionPACextension}%
$K/K_0$ is PAC.
\item \label{con poly r:DefinitionPACextension}%
For every absolutely irreducible polynomial $f(\bfT,X)\in
K[T_1,\ldots, T_e,X]$ that is separable in $X$, and nonzero
$r(\bfT)\in K[\bfT]$ there exists $(\bfa,b)\in K_0^e\times K$ for
which $r(\bfa)\neq 0$ and $f(\bfa, b) = 0$.
\item \label{con poly:DefinitionPACextension}%
For every absolutely irreducible polynomial $f(T,X)\in K[T,X]$ that
is separable in $X$, and nonzero $r(T)\in K(T)$ there exists
$(a,b)\in K_0\times K$ for which $r(a)\neq 0$ and $f(a,b) = 0$.
\item\label{con a:DefinitionPACextension}
For every finitely generated regular extension $E/K$ with a
separating transcendence basis $\bft = (t_1, \ldots , t_e)$ and
every nonzero $r(\bft)\in K(\bft)$, there exists a $K$-place $\phi$
of $E$ unramified over $K(\bft)$ such that $\Egag = K$,
$\overline{K_0(\bft)} = K_0$, $\bfa = \phi(\bft)$ is finite, and
$r(\bfa)\neq 0,\infty$.
\item\label{con b:DefinitionPACextension}%
For every finitely generated regular extension $E/K$ of
transcendence degree $1$ with a separating transcendence basis $t$
and every nonzero $r(t)\in K[t]$ there exists a $K$-place $\phi$ of
$E$ unramified over $K(t)$ such that $\Egag = K$, $\overline{K_0(t)}
= K_0$, $a = \phi(t)\neq \infty$, and $r(a)\neq0,\infty$.
\end{label1}
\end{proposition}

\begin{proof}
The proof of \cite[Lemma 1.3]{JardenRazon1994} gives the equivalence
between \eqref{con pac:DefinitionPACextension}, \eqref{con poly
r:DefinitionPACextension}, and \eqref{con
poly:DefinitionPACextension}. Obviously \eqref{con
a:DefinitionPACextension} implies \eqref{con
b:DefinitionPACextension}, so it suffices to prove that \eqref{con
poly r:DefinitionPACextension} implies \eqref{con
a:DefinitionPACextension} and that \eqref{con
b:DefinitionPACextension} implies \eqref{con
poly:DefinitionPACextension}.

\eqref{con poly r:DefinitionPACextension} $\Rightarrow$
\eqref{con a:DefinitionPACextension}: Let $x\in E/K(\bft)$ be
integral over $K[\bft]$ such that $E = K(\bft,x)$. Let
$f(\bfT,X)\in K[\bfT,X]$ be the absolutely irreducible
polynomial which is monic and separable in $X$ and for which
$f(\bft,x)=0$. Let $g(\bft)\in K[\bft]$ be the polynomial given
in Lemma~\ref{lem:RES} for the extension $E/K(\bft)$. We have
$(\bfa,b) \in K_0^e\times K$ such that $f(\bfa,b)=0$ and
$g(\bfa)r(\bfa)\neq 0,\infty$. Extend the specialization
$\bft\mapsto\bfa$ to a $K$-place $\phi$ of $E$ with the
following properties to conclude the implication: (1) $\phi(x) =
b\neq \infty$ (this is possible since $x$ is integral over
$K[\bft]$); (2) $\overline{K_0(\bft)} = K_0$ and $\Egag =
K(b)=K$ (Lemma~\ref{lem:RES}); (3) $\phi$ is unramified over
$K(\bft)$ (Lemma~\ref{lem:RES}).

\eqref{con b:DefinitionPACextension} $\Rightarrow$ \eqref{con
poly:DefinitionPACextension}: Let $f(T,X) = \sum_{k=0}^n a_k(T) X^k$
and $r(T)$ be as in \eqref{con poly:DefinitionPACextension}. Set
$r'(T) = r(T) a_n(T)$. Let $t$ be a transcendental element and let
$x\in \widetilde{K(t)}$ be such that $f(t,x) = 0$. Let $E = K(t,x)$.
Then $E$ is regular over $K$ and separable over $K(t)$. Applying
\eqref{con b:DefinitionPACextension} to $E$ and $r'(t)$ we get a
$K$-place $\phi$ of $E$ satisfying the following properties. (1)
$a=\phi(t)\in K_0$ which implies that $b=\phi(x)$ is finite, since
$\phi(f(t,x))=0$ and $f(a,X)$ has a nonzero leading coefficient; (2)
$\Egag=K$, which concludes the proof since $b\in \Egag=K$.
\end{proof}

\section{Geometric Solutions and PAC Fields}
The following result characterizes when a solution is geometric in
terms of a rational place of some regular extension. 

\begin{proposition}\label{prop: characterization of solutions}
Let $K$ be a field and consider a geometric embedding problem
\begin{equation}
\label{eq:ins-prop-char-of-sol}
(\mu\colon \gal(K)\to\gal(L/K),\alpha\colon \gal(F/E)\to
\gal(L/K))
\end{equation}
for $K$. Let $\theta\colon \gal(K)\to \gal(F/E)$ be a
weak solution. Then there exists a finite separable extension
$\Ehat/E$ such that $\Ehat/K$ is regular and for every place $\phi$
of $E/K$ that is unramified in $F$ the following two conditions are
equivalent.
\begin{enumerate}
\item \label{enu:characterization of solutions a}
$\phi$ extends to a place $\Phi$ of $F$ such that $\Egag= K$ and $\Phi^* = \theta$.
\item \label{enu:characterization of solutions b}
$\phi$ extends to a $K$-rational place of $\Ehat$.
\end{enumerate}
\end{proposition}

\begin{proof}
First we consider the special case when $\gal(F/E)\cong \gal(L/K)$.
Then there is a unique solution of $(\mu,\alpha)$, namely $\theta =
\alpha^{-1}\mu$. Let $\Phi$ be an extension of $\phi$ to $F$ and
take $\Ehat = E$. It is trivial that \eqref{enu:characterization of
solutions a} implies \eqref{enu:characterization of solutions b}.
Assume \eqref{enu:characterization of solutions b}. Then $\Phi^*$ is
defined, and from the uniqueness, $\Phi^* = \theta$.

Next we prove the general case. Let $M/K$ be a Galois extension such
that $\gal(M)= \ker(\theta)$ (in particular, $L\subseteq M$) and let
$\Fhat = FM$.
\[
\xymatrix{%
    & F\ar@{-}[r]
    & \Fhat\\
E\ar@{-}[r]\ar@{.}@/^/[urr]|(.25){\mbox{$\Ehat$}}
    & EL\ar@{-}[r]\ar@{-}[u]
        & EM\ar@{-}[u]\\
K\ar@{-}[r]\ar@{-}[u]
    & L\ar@{-}[r]\ar@{-}[u]
        & M\ar@{-}[u]\\
}%
\]
As $F$ and $M$ are linearly disjoint over $L$, the fields $F$ and
$EM$ are linearly disjoint over $EL$. We have
\[
\gal(\Fhat/E) = \gal(F/E) \times_{\gal(L/K)} \gal(M/K).
\]
Define $\hat\theta \colon \gal(K) \to \gal(\Fhat/E)$ by $\hat\theta
(\sigma) = (\theta(\sigma), \sigma|_{M})$. Let $\Ehat$ denote the
fixed field of $\hat\theta(\gal(K))$ in $\Fhat$. Then $\hat\theta$
is a solution in
\[
\xymatrix{%
    &\gal(K)\ar[d]^{\mu}\ar[dl]_{\hat\theta}\\
\gal(\Fhat/\Ehat)\ar[r]^{\hat\alpha}
    &\gal(L/K).
}%
\]
Here $\alphahat$ is the restriction map. In particular, $\Ehat/K$ is
regular. Also $\ker(\hat\theta) = \ker(\theta)\cap \gal(M) =
\gal(M)$, so $\gal(\Fhat/\Ehat)\cong \gal(M/K)$.

Assume $\phi$ extends to a $K$-rational place $\phihat$ of $\Ehat$.
Extend $\phihat$ $L$-linearly to a place $\Phihat$ of $\Fhat$. Then
$\Phihat/\phihat$ is unramified. Let $\Phi = \Phihat|_{F}$. Then by
the first part $\Phihat^* = \thetahat$.
Lemma~\ref{lem:geometricdomination} then asserts that $\Phi^* =
\theta$.

On the other hand, assume that $\phi$ extends to a place $\Phi$ of
$F$ such that $\Egag = K$ and $\Phi^*=\theta$. Extend $\Phi$
$M$-linearly to a place $\Phihat$ of $\Fhat$. Then, since
$\res_{\Fhat,F}(\Phihat^*)=\Phi^*$ and
$\res_{\Fhat,M}(\Phihat^*)=\res_{K_s,M}$, we have
\[
\Phihat^*(\sigma) = (\Phi^*(\sigma),\sigma|_M) = \thetahat.
\]
Therefore the residue field of $\Ehat$ is also $K$.
\end{proof}

\begin{remark}
In the proof it was shown that $\Ehat\subseteq FM$, where $M$ is
the solution field of $\theta$.
\end{remark}

\begin{remark}
In the proof we showed that $\ker \thetahat = \gal(M)$, that is
$\alphahat$ is an isomorphism. Thus $\Ehat M = \Fhat$. This implies
that our proof is a group theoretic formulation of the field
crossing argument.
\end{remark}

\begin{remark}
Proposition~\ref{prop: characterization of solutions} sharpens 
earlier works of Roquette on PAC Hilbertian fields \cite[Corollary
27.3.3]{FriedJarden2005}, Fried-Haran-Jarden on Frobenius
fields \cite[Proposition 24.1.4]{FriedJarden2005}, and of D\`ebes on Beckmann-Black problem \cite{Debes}. 

Although in each of the earlier works the authors prove a slightly weaker version of the proposition, e.g., Fried-Haran-Jarden consider PAC fields and D\`ebes considers the case where $L=K$, the proofs are essentially  the same as the proof given here. In other words one can easily extend the proof of \cite[Lemma 24.1.1]{FriedJarden2005} or alternatively the proof of \cite[Proposition~2.2]{Debes} to get a proof of this proposition. In fact, the author of this work chose the former. 

The innovation of our result is the focus on the solution $\theta$, rather on its image which is the decomposition group. This observation is the basis of the result on this dissretation.
\end{remark}

Proposition~\ref{prop: characterization of solutions} is extremely
useful. For example, the following result which is in fact a
reformulation of \cite[Lemma~24.1.1]{FriedJarden2005} (excluding the
part on $p$-independent elements) is a straightforward conclusion of
the proposition.

\begin{corollary}\label{cor:PACfield_solutionisgeometric}
Every weak solution of a finite embedding problem \eqref{EP:fields}
for a PAC field $K$ is geometric.
\end{corollary}

\begin{proof}
By Proposition~\ref{prop: characterization of solutions} for each
weak solution there exists a finitely generated regular extension
$\Ehat$ of $K$ such that the solution is geometric if and only if
there exists a $K$-rational place of $\Ehat$. Thus, as $K$ is PAC,
there is always such a $K$-rational place.
\end{proof}

\begin{proof}[Proof of Lemma~24.1.1 of Fried-Jarden]
Let $K$ be a PAC field, $S/R$ a regular finitely generated Galois
ring cover over $K$, and $F/E$ the corresponding fraction fields
extension. Let $L$ denote the algebraic closure of $K$ in $F$. Let
$E'$ be a subextension of $F/E$ such that the restriction map
$\alpha' \colon \gal(F/E') \to \gal(L/K)$ is surjective. Assume the
existence of an epimorphism $\gamma \colon \gal(K) \to \gal(F/E')$
such that $\alpha\gamma = \nu$, where $\nu\colon \gal(K) \to
\gal(L/K)$ is the restriction map. Let $M$ be the fixed field of
$\ker \gamma$ in $K_s$.

We have to prove the existence of an epimorphism $\phi\colon S\to M$
such that $\phi(R) = K$ and $D_\phi = \gal(F/E')$. In the notation
of this work, $\gamma$ is a solution of $(\nu,\alpha')$. The
previous corollary implies that $\gamma = \phi^*$, for a place
$\phi$ of $E$ unramified in $F$ and such $\Egag = K$. Thus
\[
D_\phi = \phi^*(\gal(K))= \gamma(\gal(K)).
\]
In fact we can choose $\phi$ such that $S\subseteq  \calO_\phi$,
since the only requirement is that the residue field of $\Ehat$ is
$K$. Then $\phi(S) = \Fgag = M$ and $\phi(R) = \Egag = K$.
\end{proof}

\section{Double Embedding Problems} \label{sec:DoubleEmbeddingProblem}
In this section we generalize the notion of embedding problems to
field extensions. The reader is advised to recall the definitions
and notation concerning embedding problems, if he needs to.

\subsection{The Definition of Double Embedding Problems}
Let $K/K_0$ be a field extension. A \textbf{double embedding problem
(DEP)} for $K/K_0$ consists on two embedding problems, one
$(\mu\colon \gal(K) \to G, \alpha\colon H\to G)$ for $K$ and one
$(\mu_0\colon \gal(K_0) \to G_0, \alpha\colon H_0\to G_0)$ for $K_0$
which are compatible in the following sense. $H\leq H_0$, $G\leq
G_0$, and if we write $i\colon H\to H_0$ and $j\colon G\to G_0$ for
the inclusion maps, then the following diagram commutes.
\begin{equation}
\label{eq:DoubleEP}%
\xymatrix{%
    &\gal(K_0) \ar[dr]^{\mu_0}\ar@{.>}[dl]_{\exists \theta_0?}\\
H_0 \ar[rr]^(0.4){\alpha_0}
         &&G_0 \\
    &\gal(K) \ar'[u]^(0.7){r}[uu]\ar[dr]^{\mu}\ar@{.>}[dl]_{\exists \theta?}\\
H \ar@{^(->}[uu]^{i} \ar[rr]^(0.4){\alpha}
         &&G \ar@{^(->}[uu]^{j}
}
\end{equation}

Given a DEP for $K/K_0$, we refer to the corresponding embedding
problem for $K$ (resp.\ $K_0$) as the \textbf{lower} (resp.\ the
\textbf{higher}) embedding problem. We call a DEP \textbf{finite} if
the higher (and hence also the lower) embedding problem is finite.

A \textbf{weak solution} of a DEP \eqref{eq:DoubleEP} is a weak
solution $\theta_0$ of the higher embedding problem which restricts
to a solution $\theta$ of the lower embedding problem via the
restriction map $\gal(K)\to \gal(K_0)$. In case $K/K_0$ is a
separable algebraic extension, the restriction map is the inclusion
map, and hence the condition on $\theta_0$ reduces to
$\theta_0(\gal(K))\leq H$. To emphasize the existence of $\theta$,
we usually regard a weak solution of a DEP as a pair
$(\theta,\theta_0)$ (where $\theta$ is the restriction of $\theta_0$
to $\gal(K)$).

\subsection{Regularly Solvable Double Embedding Problems}
Consider a double embedding problem \eqref{eq:DoubleEP} and let
$L_0$ and $L$ be the fixed fields of the kernels of $\mu_0$ and
$\mu$, respectively. Then we have isomorphisms
$\mugag_0\colon\gal(L_0/K_0) \to G_0$ and $\mugag\colon\gal(L/K) \to
G$. Hence (as in the case of embedding problems) replacing $G_0$ and
$G$ with $\gal(L_0/K_0)$ and $\gal(L/K)$ (and replacing
correspondingly all the maps) gives us an equivalent DEP. The
compatibility condition is realized as $L = L_0 K$.

\begin{lemma}
Every DEP for which the upper embedding problem is rational is
equivalent to the following DEP:
\begin{equation}
\label{eq:RegularDoubleEP}%
\xymatrix@R=15pt@C=15pt{%
    &\gal(K_0) \ar[dr]^{\mu_0}\\
\gal(F_0/K(\bft)) \ar[rr]^(0.4){\alpha_0}
         &&\gal(L_0/K_0) \\
    &\gal(K) \ar'[u]^(0.7){r}[uu]\ar[dr]^{\mu}\\
\gal(F/E) \ar@{^(->}[uu]^{i} \ar[rr]^(0.4){\alpha}
         &&\gal(L/K). \ar@{^(->}[uu]^{j}
}
\end{equation}
Moreover, the results even holds if we take $t = \bft$ a single
variable.
\end{lemma}

\[
\xymatrix@R=12pt@C=12pt{
         &&K(\bft)\ar@{-}[rr]
                  &&L(\bft)\ar@{-}[r]
                       &F\\
    &K\ar@{-}[rr]\ar@{-}[ur]
            &&L\ar@{-}[ur]\\
        &&(L_0\cap K)(\bft)\ar@{-}'[r][rr]\ar@{-}'[u][uu]
                 &&L_0(\bft)\ar@{-}[r]\ar@{-}[uu]
                     &F_0\ar@{-}[uu]\\
K_0\ar@{-}[r]
    &L_0\cap K\ar@{-}[rr]\ar@{-}[uu]\ar@{-}[ur]
            &&L_0\ar@{-}[ur]\ar@{-}[uu]
}
\]

\begin{proof}
Let \eqref{eq:DoubleEP} be a rational double embedding problem. By
definition, it means that the higher embedding problem is rational,
so we may replace the higher embedding problem of
\eqref{eq:DoubleEP} with an equivalent embedding problem as in
\eqref{eq:RegularDoubleEP}. By
Lemma~\ref{lem:regularsolvablewithonet} we may assume that $\bft=t$.

Let $F=F_0K$ and $L = L_0K$. The compatibility condition implies
that $H$ embeds into $\gal(F_0/K_0(\bft))$ (via $i$) as a subgroup
of $\gal(F_0/ (L_0 \cap K) (\bft)) \cong \gal(F/K(\bft))$. Let
$E\subseteq F$ be the fixed field of $H$, i.e., $H = \gal(F/E)$.
Under this embedding, $\alpha\colon \gal(F/K(\bft))\to \gal(L/K)$ is
the restriction map. Therefore $\alpha(H) = \gal(L/K)$ implies that
$E\cap L = K$, and hence $E$ is regular over $K$.
\end{proof}

\begin{definition}
A double embedding problem \eqref{eq:RegularDoubleEP} as in the
above lemma is called \textbf{rational} double embedding problem.
\end{definition}

\begin{remark}
The converse of the above lemma is also valid, that is to say,
assume we have a rational double embedding problem as in
\eqref{eq:RegularDoubleEP}, i.e.\ a finitely generated regular
extension $E/K$, a separating transcendence basis $\bft$ for $E/K$,
and a finite Galois extension $F_0/K_0(\bft)$ such that $E\subseteq
F$, where $F=F_0K$. Then all the restriction maps in
\eqref{eq:RegularDoubleEP} are surjective, and hence
\eqref{eq:RegularDoubleEP} defines a finite double embedding
problem.
\end{remark}

\subsection{Geometric Solutions of Double Embedding Problems}
First recall that a weak solution $\theta$ of a geometric embedding
problem
\[
(\mu\colon \gal(K)\to \gal(L/K), \alpha\colon \gal(F/E)\to
\gal(L/K))
\]
is geometric if $\theta=\phi^*$, for some $K$-place $\phi$ of
$E$ that is unramified in $F$ and under which the residue field
of $E$ is $K$. Then we call a weak solution $(\theta,\theta_0)$
of \eqref{eq:RegularDoubleEP} \textbf{geometric solution} if
$(\theta,\theta_0)=(\phi^*,\phi_0^*)$, where $\phi^*$ is a
geometric solution of the lower embedding problem and $\phi_0 =
\phi|_{K_0}$.

Note that since $\phi_0^*$ is a solution of the higher embedding
problem, the residue field of $K_0(\bft)$ is $K_0$. In particular,
if $\phi(\bft)$ is finite, then $\phi(\bft) \in K_0^e$. Also note
that for a $K$-place of $E$ that is unramified at $F$ and such that
$\Egag= K$ and $\overline{K_0(\bft)}=K_0$, the pair
$(\phi^*,\phi_0^*)$ is indeed a weak solution of
\eqref{eq:RegularDoubleEP}, since $\phi_0^* =
\res_{K_s,K_{0s}}\phi^*$.

\section{The Lifting Property}
In this section we formulate and prove the lifting property.
First we reduce the discussion to separable algebraic extensions
by showing that if $K/K_0$ is PAC, then $K\cap K_{0s}/K_0$ is
PAC and $\gal(K)\cong \gal(K\cap K_{0s})$ via the restriction
map. Then we characterize separable algebraic PAC extensions in
terms of geometric solutions of double embedding problems. From
this characterization we establish the lifting property. Finally
we prove a strong (but complicated) version of the lifting
problem to PAC extensions of finitely generated fields.

\subsection{Reduction to Separable Algebraic Extensions}
In {\cite[Corollary 1.5]{JardenRazon1994}} Jarden and Razon show

\begin{lemma}[Jarden-Razon] If $K/K_0$ is PAC, then so is $K\cap
K_{0s}/K_0$.
\end{lemma}

Moreover, we have:

\begin{theorem}\label{thm:nonalgebraicPAC}
Let $K/K_0$ be a PAC extension. Then $K\cap K_{0s}/K_0$ is PAC and the
restriction map $\gal(K)\to \gal(K\cap K_{0s})$ is an isomorphism.
\end{theorem}

\begin{proof}
It suffices to show that $K_s =  K_{0s} K$. Let $L/K$ be a finite
Galois extension with Galois group $G$ of order $n$. Embed $G$ into
the symmetric group $S_n$. Let $F_0/K_0(\bft)$ be a regular
realization of $S_n$ with $F_0$ algebraically independent from $K$
over $K_0$ (Theorem~\ref{thm:regulargroups}). Then $F = F_0 K$ is
regular over $K$ and $\gal(F/K(\bft)) \cong S_n$. Furthermore
\begin{eqnarray*}
\gal(FL/K(\bft)) &=& \gal(FL/L(\bft))\times \gal(FL/F)\\
    &\cong& \gal(F/K(\bft))\times \gal(L/K) \cong S_n \times G.
\end{eqnarray*}
Let $E$ be the fixed field of the subgroup $\Delta=\{(g,g)\mid g\in
G\}$ in $FL$, i.e., $\gal(FL/E) \cong \Delta$. By Galois
correspondence,  $S_n \Delta = S_n\times G$ implies that $E\cap L =
K$ and $1 = \Delta\cap S_n = G\cap \Delta$ implies that $FL = EL =
FE$. In particular, $E/K$ is regular.
\[
\xymatrix{%
F\ar@{-}[r]
    &FL\\
K(\bft)\ar@{-}[u]\ar@{-}[r]\ar@{-}[ur]|{E}
    &L(\bft)\ar@{-}[u]\\
K\ar@{-}[u]\ar@{-}[r]
    &L\ar@{-}[u].
}%
\]

Let $x_F, x_E, x_{FL} $ be primitive elements of $F/K(t)$,
$E/K(\bft)$, and $FL/K(\bft)$, respectively. Let $h(\bft)\in
K[\bft]$ be the products of the corresponding polynomials given
in Lemma~\ref{lem:RES}.

As $K/K_0$ is PAC, there is a $K$-place $\phi$ of $E$ such that
$\Egag = K$, $\overline{K_0(\bft)} = K_0$, $\bfa=\phi(\bft)$ is
finite, and $h(\bfa)\neq 0$. Extend $\phi$ to an $L$-place
$\Phi$ of $FL$.

Then, $FL = EL$ implies that $\overline{FL} =
\overline{EL}=L(\Phi(x_E)) = L$. However, as $F = F_0 K$, it
follows that $\Fgag = \Fgag_0 K$, hence, $L = \overline{FL} =
\overline{FE}= K(\Phi(x_E),\Phi(x_F))= \Fgag = \Fgag_0
K\subseteq K_{0s}K$, as needed.
\end{proof}

\begin{remark}
The above theorem also follows from the main result of
\cite{Razon2000}. However we shall prove that result using the
theorem.
\end{remark}

PAC fields have a nice elementary theory. Since $K$ and $K\cap
K_{0s}$ are PAC fields and since they have isomorphic absolute
Galois groups, they are elementary equivalent under some necessary
condition:

\begin{corollary}
Let $K/K_0$ be a separable PAC extension. Assume that $K$ and $K\cap
K_{0s}$ have the same imperfect degree. Then $K$ is an elementary
extension of $K\cap K_{0s}$.
\end{corollary}

\begin{proof}
The assertion follows from \cite[Corollary 20.3.4]{FriedJarden2005}
and Theorem~\ref{thm:nonalgebraicPAC}.
\end{proof}

\subsection{Characterization of Separable Algebraic PAC Extensions}

\begin{proposition}
\label{prop:ExistGeoSol}%
Let $K/K_0$ be a separable algebraic field extension. The following
conditions are equivalent:
\begin{label1}
\item
\label{con_a:ExistGeoSol} $K/K_0$ is PAC.
\item
\label{con_b_mid:ExistGeoSol}%
For every finite rational double embedding problem
\eqref{eq:RegularDoubleEP} for $K/K_0$ and every nonzero rational
function $r(\bft)\in K(\bft)$, there exists a  geometric weak
solution $(\phi^*,\phi_0^*)$ such that $\bfa=\phi(\bft)$ is finite
and $r(\bfa)\neq 0,\infty$.
\item
\label{con_b:ExistGeoSol}%
For every finite rational double embedding problem
\eqref{eq:RegularDoubleEP} for $K/K_0$ with $\bft=t$ a
transcendental element there exist infinitely many geometric weak
solution $(\phi^*,\phi_0^*)$.
\end{label1}
\end{proposition}

\begin{proof}
The implication
\eqref{con_a:ExistGeoSol}$\Rightarrow$\eqref{con_b_mid:ExistGeoSol}
follows from Proposition~\ref{prop:DefinitionPACextension} (part
\ref{con a:DefinitionPACextension}) and the definition of geometric
weak solutions.

The implication
\eqref{con_b_mid:ExistGeoSol}$\Rightarrow$\eqref{con_b:ExistGeoSol}
is immediate.

\eqref{con_b:ExistGeoSol}$\Rightarrow$\eqref{con_a:ExistGeoSol}:
We apply Proposition~\ref{prop:DefinitionPACextension} and show
that \eqref{con b:DefinitionPACextension} holds. Let $E/K$ be a
regular extension with a separating transcendence basis $t$ and
let $r(t)\in K[t]$ be nonzero. Choose $F_0$ to be a finite
Galois extension of $K_0(t)$ such that $E\subseteq F_0 K$ (such
$F_0$ exists since $K/K_0$ is separable and algebraic) and let
$F = F_0K$, $L = F\cap K_s$, and $L_0 = F_0\cap K_{0s}$. By
assumption there are infinitely many geometric weak solution
$(\phi^*,\phi_0^*)$ of the DEP
\[
\xymatrix{
    &\gal(K_0)\ar[dr]\ar@{.>}[dl]_{\phi_0^*}\\
\gal(F_0/K_0(t))\ar[rr]
        &&\gal(L_0/K_0)\\
    &\gal(K)\ar'[u][uu]\ar[dr]\ar@{.>}[dl]_{\phi^*}\\
\gal(F/E)\ar[uu]\ar[rr]
        &&\gal(L/K).\ar[uu]
}
\]
Since for only finitely many solutions $\phi(t)$ is infinite or
$r(\phi(t))=0,\infty$ we can find a solution such that
$\phi(t)\neq \infty$ and $r(\phi(t))\neq 0,\infty$. In
particular, $\Egag = K$ and $\overline{K_0(t)}=K_0$, as required
in \eqref{con b:DefinitionPACextension}.
\end{proof}

Let $K/K_0$ be a PAC extension and consider a rational DEP for
$K/K_0$. The following key property -- the lifting property --
asserts that any weak solution of the lower embedding problem can be
lifted to a geometric weak solution of the DEP.

\begin{proposition}[The lifting property]
\label{prop:ExtensionSoltoGeoSol_DEP}%
Let $K/K_0$ be a PAC extension, let \eqref{eq:RegularDoubleEP} be a
rational DEP for $K/K_0$, and let $\theta\colon \gal(K)\to
\gal(F/E)$ be a weak solution of the lower embedding problem in
\eqref{eq:RegularDoubleEP}. Then there exists a geometric weak
solution $(\phi^*,\phi_0^*)$ of \eqref{eq:RegularDoubleEP}
such that $\theta = \phi^*$. \\
Moreover, if $r(\bft)\in K(\bft)$ is nonzero, we can choose $\phi$
such that $\bfa = \phi(\bft)\in K_0^e$ and $r(\bfa)\neq 0,\infty$.
\end{proposition}

\begin{proof}
By Proposition~\ref{prop: characterization of solutions} there
exists a finite separable extension $\Ehat/E$ that is regular over
$K$ with the following property. If $\Phi$ is a $K$-rational place
of $\Ehat$ and if $\phi = \Phi|_{E}$ is unramified in $F$, then
$\phi^* = \theta$.

By the PACness of $K/K_0$ there exists a $K$-rational place $\Phi$
of $\Ehat$ such that $\phi = \Phi|_{E}$ is unramified in $F$, the
residue field of $K_0(\bft)$ is $K_0$, $\bfa = \Phi(\bft)$ is
finite, and $r(\bfa)\neq 0,\infty$. If we let
$\phi_0=\phi|_{K_0(\bft)}$, then we get that $(\phi^*,\phi_0^*)$ is
a geometric weak solution and that $\phi^*=\theta$.
\end{proof}

A first and an easy consequence is the transitivity of PAC
extensions. To the best of our knowledge there is no other way to
prove this property in the literature.

\begin{lemma}\label{lem:transitive}
Let $K_0 \subseteq K_1 \subseteq K_2$ be a tower of separable
algebraic extensions. Assume that $K_2/K_1$ and $K_1/K_0$ are PAC
extensions. Then $K_2/K_0$ is PAC.
\end{lemma}

\begin{proof}
Let
\begin{eqnarray*}
&&((\mu_0\colon \gal(K_0)\to \gal(L_0/K_0),\alpha_0\colon \gal(F_0/K_0(t))\to \gal(L_0/K_0),\\
&&(\mu_2\colon \gal(K_2)\to \gal(L_2/K_2),\alpha_2\colon \gal(F_2/E)
\to \gal(L_2/K_2))
\end{eqnarray*}
be a rational finite DEP for $K_2/K_0$. By
Lemma~\ref{prop:ExistGeoSol} it suffices to find a geometric weak
solution to $((\mu_0,\alpha_0),(\mu_2,\alpha_2))$. Set $F_1 = F_0
K_1$, $L_1 = L_0 K_1$. Then, since $K_2/K_1$ is PAC there exists a
weak solution $(\phi_2^*,\phi_1^*)$ of the double embedding problem
defined by the lower part of the following commutative diagram.
\begin{equation}
\xymatrix{
    &\gal(K_0)\ar[dr]^{\mu_0}\ar@{.>}[dl]_{\phi_0^*}\\
\gal(F_0/K_0(t))\ar[rr]^(.4){\alpha_0}
        &&\gal(L_0/K_0)\\
    &\gal(K_1)\ar'[u][uu]\ar[dr]^{\mu_1}\ar@{.>}[dl]_{\phi_1^*}\\
\gal(F_1/K_1(t))\ar[uu]\ar[rr]^(.4){\alpha_1}
        &&\gal(L_1/K_1)\ar[uu]\\
    &\gal(K_2)\ar'[u][uu]\ar[dr]^{\mu_2}\ar@{.>}[dl]_{\phi_2^*}\\
\gal(F_2/E)\ar[uu]\ar[rr]^(.4){\alpha_2}
        &&\gal(L_2/K_2).\ar[uu]
}
\end{equation}
Now we lift $\phi_1^*$ to a geometric weak solution
$(\phi_1^*,\phi_0^*)$ of the DEP for $K_1/K_0$ defined by the higher
part of the diagram. This is possible by the lifting property
applied to the PAC extension $K_1/K_0$.

Since $\phi_0^*|_{\gal(K_2)} = \phi^*_1|_{\gal(K_2)}=\phi_2^*$ we
get that $(\phi_2^*,\phi_0^*)$ is a geometric weak solution of the
DEP we started from.
\end{proof}

\begin{proposition}\label{prop:transitivityofPACextension}
Let $\kappa$ be an ordinal number and let
\[
K_0 \subseteq K_1 \subseteq K_2 \subseteq \cdots \subseteq K_\kappa
\]
be a tower of separable algebraic extensions. Assume that
$K_{\alpha+1}/K_\alpha$ is PAC for every $\alpha<\kappa$ and that
$K_\alpha = \bigcup_{\beta<\alpha} K_\beta$ for every limit
$\alpha\leq \kappa$. Then $K_\kappa/K_0$ is PAC.
\end{proposition}

\begin{proof}
We apply transfinite induction. Let $\alpha\leq \kappa$. If $\alpha$
is a successor ordinal, then the assertion follows from the previous
lemma.

Let $\alpha$ be a limit ordinal. Let
\begin{equation}\label{eq:transfinitetransitive}
\xymatrix{%
    &\gal(K_0)\ar[dr]\\
\gal(F_0/K_0(\bft))\ar[rr]
    &&\gal(L_0/K_0)\\
    &\gal(K_\alpha)\ar[dr]\ar'[u][uu]\\
\gal(F_\alpha/E_\alpha)\ar[rr]\ar[uu]
    &&\gal(L_\alpha/K_\alpha)\ar[uu]
}%
\end{equation}
be a rational DEP for $K_\alpha/K_0$. Here $F_\alpha = F_0 K_\alpha$
and $L_\alpha = L_0K_\alpha$. Now since all the extensions are
finite, there exists $\beta<\alpha$ such that
$\gal(F_\alpha/K_\alpha) \cong \gal(F_\beta/K_\beta)$ and
$\gal(L_\alpha/K_\alpha) \cong \gal(L_\beta/K_\beta)$ (via the
corresponding restriction maps), where $F_\beta = F_0 K_\beta$ and
$L_\beta=L_0K_\beta$.

Induction gives a weak solution of the double embedding problem
\eqref{eq:transfinitetransitive} with $\beta$ replacing $\alpha$.
This weak solution induces a weak solution of
\eqref{eq:transfinitetransitive} via the above isomorphisms.
\end{proof}

\section{Strong Lifting Property for PAC Extensions of Finitely
Generated Fields} Let $K_0$ be a finitely generated field (over its
prime field). In this section we prove a strong lifting property for
PAC extensions $K/K_0$. The additional ingredient is the Mordell
conjecture for finitely generated fields (now a theorem due to
Faltings in characteristic $0$ \cite{Faltings83} and to
Grauert-Manin in positive characteristic \cite[page 107]{Samuel66}).

The following lemma is based on the Mordell conjecture.

\begin{lemma}[{\cite[Proposition 5.4]{JardenRazon1994}}]\label{lem:Mordell}
Let $K_0$ be an finitely generated infinite field, $f\in
K_0[T,X]$ an absolutely irreducible polynomial which is
separable in $X$, $g\in K_0[T,X]$ an irreducible polynomial
which is separable in $X$, and $0\ne r\in K_0[T]$. Then there
exist a finite purely inseparable extension $K_0'$ of $K_0$, a
nonconstant rational function $q\in K_0'(T)$, and a finite
subset $B$ of $K_0'$ such that $f(q(T),X)$ is absolutely
irreducible, $g(q(a),X)$ is irreducible in $K_0'[X]$, and
$r(q(a))\ne 0$ for any $a\in K_0'\smallsetminus B$.
\end{lemma}

Let $K/K_0$ be an extension. Consider a rational double embedding
problem \eqref{eq:RegularDoubleEP} for $K/K_0$ (with $\bft=t$). For
any subextension $K_1$ of $K/K_0$ we have a corresponding rational
double embedding problem. Namely
\begin{equation}
\label{eq:CorresponingRegularDoubleEP}%
\xymatrix{%
    &\gal(K_1)\ar[dr]^{\mu_1}\\
\gal(F_1/K_1(t))\ar[rr]^(.4){\alpha_1}
    &&\gal(L_1/K_1)\\
    &\gal(K)\ar[dr]^{\mu}\ar'[u][uu]\\
\gal(F/E)\ar[rr]^{\alpha}\ar[uu]
    &&\gal(L/K),\ar[uu]
}%
\end{equation}
where $F_1 = F_0 K_1$, $L_1=L_0K_1$, and $\mu_1$ and $\alpha_1$ are
the restriction maps.

Assume that  $K_1'/K_1$ is a purely inseparable extension.
Then the double embedding problem
\eqref{eq:CorresponingRegularDoubleEP} remains the same if we
replace all fields by their compositum with $K_1'$.

\begin{proposition}[Strong Lifting Property]
Let $K$ be a PAC extension of a finitely generated field $K_0$. Let
\[
\calE(K)  = (\mu\colon \gal(K) \to \gal(L/K), \alpha \colon
\gal(F/E)\to \gal(L/K))
\]
be an embedding problem as in \eqref{EP:fields}\footnote{That is to
say, $E$ is regular (of transcendence degree $1$) over $K$, $F/E$ is
finite and Galois, $L=F\cap K_s$, and all the maps are the
restriction maps.}, and let $\theta\colon \gal(K) \to \gal(F/E)$ be
a weak solution of $\calE(K)$. Then there exist a finite
subextension $K_1/K_0$ and a finite purely inseparable extension
$K_1'/K_1$ satisfying the following properties.
\begin{enumerate}
\item For any rational double embedding problem
\eqref{eq:RegularDoubleEP} (whose lower embedding problem is
$\calE(K)$), we can lift $\theta$ to a weak solution
$(\theta,\theta_1)$ of the double embedding problem
\eqref{eq:CorresponingRegularDoubleEP} in such a way that $\theta_1$
is surjective.
\item
The solution $(\theta,\theta_1)$ is a geometric solution of the
double embedding problem that we get from
\eqref{eq:CorresponingRegularDoubleEP} by replacing all fields with
their compositum with $K_1'$.
\end{enumerate}
\end{proposition}

\begin{proof}
By Proposition~\ref{prop: characterization of solutions} there
exists a finite separable $\Ehat/E$ that is regular over $K$ such
that a $K$-place $\phi$ of $E$ that is unramified in $F$ can be extended to a place $\Phi$ of $F$ such that
$\Phi^*=\theta$ if and only if $\phi$ extends to a $K$-rational
place of $\Ehat$. Let $f(t,X)\in K[t,X]$ be an absolutely
irreducible polynomial whose root $x$ generates $\Ehat/K(t)$, i.e.\
$\Ehat = K(t,x)$. Let $M$ be the fixed field of $\ker(\theta)$ in
$K_s$. Then $M/K$ is a finite Galois extension. Let $h(X)\in K[X]$
be a Galois irreducible polynomial whose root generates $M/K$.

Let $K_1$ be a finite subextension of $K/K_0$ that contains the
coefficients of $f$ and $h$ and such that $h$ is Galois over it. Let
$M_1$ be the splitting field of $h$ over $K_1$ and let $L_1$, $F_1$
be as in the corresponding rational double embedding problem
\eqref{eq:CorresponingRegularDoubleEP}. Then $\gal(M/K)\cong
\gal(M_1/K_1)$, and thus also $\gal(L/K)\cong \gal(L_1/K_1)$.

\[
\xymatrix@R=10pt{
            &&&F\\
    &E \ar@{-}[r]
        &EL\ar@{-}[rr]\ar@{-}[ru]
                &&EM\\
    &K(t) \ar@{-}[r]\ar@{-}[u]
        &L(t)\ar@{-}[rr]\ar@{-}[u]
                &&M(t)\ar@{-}[u]\\
            &&&F_1\ar@{-}'[u]'[uu][uuu]\\
K_0(t)\ar@{-}[r]
    &K_1(t) \ar@{-}[r]\ar@{-}[uu]
        &L_1(t)\ar@{-}[rr]\ar@{-}[uu]\ar@{-}[ur]
                &&M_1(t)\ar@{-}[uu]
    }
\]

Let $g(t,X)\in K_1[T,X]$ be an irreducible polynomial whose root
generates $F_1/K_1(t)$. Choose $r(t)\in K_1(t)$ such that $r(a)\neq
0$ implies that the prime $(t-a)$ is unramified in $F_1$ and that
the leading coefficients of $f(t,X)$ and $g(t,X)$ do not vanish at
$a$. Let $K_1'/K_1$ be the purely inseparable extension, $B\subseteq
K_1'$ the finite subset, and $q\in K_1'(T)$ the nonconstant rational
function that Lemma~\ref{lem:Mordell} gives for $K_1$, $g$, $f$, and
$r$. Let $K' = KK_1'$.

Since $K'/K_1'$ is PAC (\cite[Corollary 2.5]{JardenRazon1994}) there
exist $a\in K_1'\smallsetminus B$ and $b\in K'$ for which
$f(q(a),b)=0$ (Proposition~\ref{prop:DefinitionPACextension}).
Extend $t\mapsto q(a)$ to a $K'$-rational place $\Phi$ of $\Ehat
K_1'$. Then $\phi=\Phi|_{E K_1'}$ is unramified in $FK_1'$ (since
$r(q(a))\neq 0$). It follows that $(\phi^*,\phi_1^*)$ is a geometric
weak solution of the DEP $((\mu,\alpha),(\mu_1,\alpha_1))$ that we
get from \eqref{eq:CorresponingRegularDoubleEP} by replacing all
fields with their compositum with $K_1'$.

Moreover, since $F_1K_1'/K_1'(t)$ is generated by $g(t,X)$ and
$g(q(a),X)$ is irreducible, we get that $\phi_1^*$ is surjective.
This proves (b). Now assertion (a) follows since $(\phi^*,\phi_1^*)$
is a (not necessarily geometric) solution of
$((\mu,\alpha),(\mu_1,\alpha_1))$.
\end{proof}

%-----------------------------------------------------------------------------------------------------------

\chapter{On the Galois Structure of PAC Extensions}
In \cite{JardenRazon1994} Jarden and Razon use a deep theorem of
Faltings to prove that some specific Galois extensions of $\bbQ$
are not PAC extensions of any number field. Then they ask
whether this is a general phenomenon or a coincidence.

In \cite{Jarden2006} Jarden settles the above question by
showing that the only Galois PAC extension of a number field is
its algebraic closure. This result is based on Frobenius'
density theorem, on Neukirch's characterization of $p$-adically
closed fields among all algebraic extensions of $\bbQ$, and on
the special property of $\bbQ$ that it has no proper subfields
(!).

In \cite{Bary-SorokerJarden} Jarden and the author elaborate the
Jarden-Razon method and extend Jarden's answer to the larger
family of finitely generated fields $K$ using the Mordell
conjecture (proved by Faltings in characteristic $0$ and
Grauert-Manin in positive characteristic).

Using the lifting property, we can elementarily reprove all the
above results. This new proof does not use any of the properties of
finitely generated fields, and hence is valid for an arbitrary field
$K$ (Theorem~\ref{thm:GaloisClosurePAC}). This result and other
results appearing in this chapter (e.g.\ descent results) manifest
that the right approach to PAC extension is via double embedding
problems.

\section{Descent Features}
The following result allows us to descend Galois groups in a PAC
extension.

\begin{theorem}
\label{thm realizing group}%
Let $K/K_0$ be a PAC extension, let $\phi\colon \gal(K)\to
\gal(K_0)$ be the restriction map, and let $\theta\colon\gal(K)
\to G$ be an epimorphism onto a finite group $G$. Furthermore,
let $i\colon G\hookrightarrow G_0$ be an embedding of $G$ into a
finite group $G_0$ which is regularly realizable over $K_0$.
Then there exists a homomorphism $\theta_0\colon \gal(K_0)\to
G_0$ such that $\theta_0 \phi = i\theta$:
\[
\xymatrix{%
{\gal}(K)\ar[r]^{\phi}\ar[d]_{\theta}
    & {\gal}(K_0) \ar[d]^{\theta_0}\\
G \ar[r]_{i}
    & G_0.
}%
\]
\end{theorem}

\begin{proof}
The map $\theta\colon \gal(K)\to G$ is a solution of the lower
embedding problem of the double embedding problem
\begin{equation}
\label{eq thm realizing groups}%
\xymatrix{%
    &\gal(K_0)\ar[dr]\ar@{.>}[dl]_{\theta_0}\\
G_0\ar[rr]
        &&1    \\
&\gal(K)\ar[dr]\ar[dl]_{\theta}\ar'[u][uu]\\
G\ar[rr]\ar[uu]
        &&1.\ar[uu]
}%
\end{equation}
Since $G_0$ is regularly realizable, the DEP is in fact rational.
Thus $\theta$ extends to a geometric weak solution
$(\theta,\theta_0)$ of \eqref{eq thm realizing groups} by the
lifting property (Proposition~\ref{prop:ExtensionSoltoGeoSol_DEP}).
\end{proof}

In case $G = G_0$, Theorem~\ref{thm realizing group} immediately
yields:

\begin{corollary}\label{cor:decentofregularlyrealizable}
Let $K/K_0$ be a PAC extension. Let $G$ be a finite Galois group
over $K$ that is regularly realizable over $K_0$. Then $G$ also
occurs as a Galois group over $K_0$.
\end{corollary}

Since every abelian group is regularly realizable over any field
(Theorem~\ref{thm:regulargroups}), we get the following
\begin{corollary}
Let $K/K_0$ be a PAC extension. Then $K^{\rm ab} = K_0^{\rm ab} K$.
\end{corollary}

%----------------------------------------------------------------

We reprove \cite[Theorem 5]{Razon2000} using Theorem~\ref{thm
realizing group} and the fact that the symmetric group is regularly
realizable over any field. This proof provides an insight into
Razon's original technical proof.

\begin{corollary}[Razon]\label{cor:Razon}%
Let $K/K_0$ be a PAC extension. Let $L/K$ be a separable algebraic
extension. Then there exists a separable algebraic extension
$L_0/K_0$ that is linearly disjoint from $K$ over $K_0$ for which
$L=L_0K$.
\end{corollary}

\begin{proof}
First assume that $[L:K]$ is finite. Let $M$ be the Galois closure
of $L/K$, $G=\gal(M/K)$, $G' = \gal(M/L)$, and $\theta\colon \gal(K)
\to G$ the quotient map. The action of $G$ on the cosets
$\Sigma=G/G'$ admits an embedding $i\colon G\to S_\Sigma$.

As $S_\Sigma$ is regularly realizable
(Theorem~\ref{thm:regulargroups}), Theorem~\ref{thm realizing group}
gives a homomorphism $\theta_0\colon \gal(K_0)\to S_\Sigma$ for
which $\theta_0|_{\gal(K)} = i  \theta$ (note that
$\theta_0|_{\gal(K_0)} = \phi\theta_0$ in the diagram above). In
particular, the image $H$ of $\theta_0$ contains $i(G)$, and hence
$H$ is transitive. Thus $(H:H') = |\Sigma|= [L:K]$, where $H'$ is
the stabilizer in $H$ of the coset $G'\in \Sigma$. Also, as the
subgroup $G'\leq G$ is the stabilizer in $G$ of the coset $G'\in
\Sigma$, it follows that $H'\cap i(G)=i(G')$. Hence $i^{-1}(H')=G'$

Let $M_0$ be the fixed field of $\ker (\theta_0)$ (i.e.\ $\gal(M_0)=
\ker(\theta_0)$) and let $L_0\subseteq M_0$ be the corresponding
fixed field of $H'$ (i.e.\ $\gal(L_0) = \theta_0^{-1}(H')$). Since
$i$ is an embedding and $\theta_0|_{\gal(K)} = i  \theta$, we have
\begin{eqnarray*}
\gal(L)
    &=
        & \theta^{-1}(G')= \theta^{-1}i^{-1}(H') = \phi^{-1}
        \theta_0^{-1} (H') = \phi^{-1} (\gal(L_0)) = \gal(L_0K).
\end{eqnarray*}
Hence $L = L_0K$. In addition, $L_0$ is linearly disjoint from $K$,
since $[L_0:K_0]=(H:H') = [L:K]$, as needed.

The case where $L/K$ is an infinite extension follows from Zorn's
Lemma. The main point is that for a tower of algebraic extensions
$L_1\subseteq L_2\subseteq L_3$, $L_3/L_1$ is separable if and only
if both $L_2/L_1$ and $L_3/L_2$ are. The details can be found in
\cite{Razon2000}.
\end{proof}

\begin{remark}
In the last proof $H'$ was the stabilizer of a point of a subgroup
of $S_n$. This stabilizer is, in general, not normal even if $L/K$
is Galois. That is to say, $L_0/K_0$ need not be Galois, even if
$L/K$ is.
\end{remark}

\begin{example}
In this example we build a field tower $K_0\subseteq K_1 \subseteq
K$ such that $K/K_1$ is PAC, $K/K_0$ is not PAC, and $K_1/K_0$ is
finite.

Let $K_0 = \bbQ_{\rm sol}$ the maximal pro-solvable extension of
$\bbQ$. Then $K_0$ has no extensions of degree $2$. This imply that
if $K$ is an extension of $K$ and there exists $L/K$ of degree $2$,
then $K/K_0$ is not PAC. Indeed, had $K/K_0$ PAC, we would have got
$L_0/K_0$ of degree $2$ such that $L = L_0 K$
(Corollary~\ref{cor:Razon}), a contradiction.

Let $K_1/K_0$ be a finite proper extension and $e\geq 1$ an integer.
Weissauer's theorem \cite[Theorem~13.9.1]{FriedJarden2005} asserts
that $K_1$ is Hilbertian. Hence for almost all $\bfsigma\in
\gal(K_1)^e$ the extension $K/K_1$ is PAC and the absolute Galois
group of $K$ is free on $e$ generators, where $K =
\bbQgal(\bfsigma)$. In particular, $K$ has an extension of degree
$2$, so by the previous paragraph, $K/K_0$ is not PAC.
\end{example}

\section{Restrictions on the Galois Structure of $K/K_0$}
\label{sec:restriction}%%
Not every extension can be a PAC extension. This section reveals
some of the restrictions.

Warning: The reader is advised to recall the definitions and basic
properties of wreath products appearing in
Section~\ref{sec:wreathproduct} before continuing any further.

\begin{lemma}\label{lem:noGaloisExt}
Let $K/K_0$ be a proper Galois PAC extension. Then $K$ is separably
closed.
\end{lemma}

\begin{proof}
Let $A$ be a Galois group over $K$ and let $\theta\colon \gal(K) \to
A$ be a corresponding epimorphism. Let $G$ be a nontrivial finite
quotient of $\gal(K/K_0)$ and $\phi \colon \gal(K_0) \to G$ the
restriction map. We may assume that $A\leq S_n$ for some large $n$.
Define an embedding $i\colon A \to S_n\wre G$ by setting $i(a) =
(f_a,1)$, where $f_a(1)=a$ and $f_a(\sigma)=1$ for any nontrivial
$\sigma \in G$. As $S_n$ is regularly realizable over any field,
$(\phi,S_n\wre G \to G)$ is rational over $K_0$
(Corollary~\ref{cor:wreath_regularlysolvable}).

By the lifting property we can extend the solution $\theta$ of the
lower embedding problem to a weak solution $(\theta,\theta_0)$ of
the DEP
\[
\xymatrix{%
    & \gal(K_0)\ar[dr]\ar[dl]_{\theta_0}\\
S_n\wre G\ar[rr]
    &&G
\\
    & \gal(K)\ar[dr]\ar[dl]_{\theta}\ar'[u][uu]\\
A\ar[rr]\ar[uu]
    &&1.\ar[uu]
}%
\]
As $\gal(K)$ is normal in $\gal(K_0)$, we have that $i(A)$ is
invariant under the image of $\gal(K_0)$. Let $1\neq \sigma\in G$,
extend it to $\hat \sigma\in \gal(K_0)$, and let $f\sigma =
\theta_0(\hat\sigma)$. Then $i(A)\cap i(A)^{f\sigma} = i(A)\cap
i(A)^\sigma = 1$. But $i(A)$ is invariant under $f\sigma$; we thus
get that $i(A)=1$. Therefore $K$ is separably closed, as desired.
\end{proof}

Some Galois extensions of $\bbQ$ are known to be PAC as fields. So
we get below PAC fields which are not PAC over any proper subfield.

\begin{corollary} \label{cor:PACofnoSubField}
Let $K$ be a Galois extension of $\bbQ$ which is not algebraically
closed. Then $K$ is a PAC extension of no proper subfield. This
result holds even if $K$ is PAC. In particular it holds in the
following cases.
\begin{enumerate}
\item
the Galois hull $K=\bbQgal[\bfsigma]$ of $\bbQ$ in
$\bbQgal(\bfsigma)$, for almost all $\bfsigma\in \gal(\bbQ)^e$
\cite[Theorem 18.10.2]{FriedJarden2005}.
\item
$K=\bbQ_{\textnormal{tot},\bbR}(i)$, where
$\bbQ_{\textnormal{tot},\bbR}$ is the maximal real Galois extension
of $\bbQ$ and $i^2 = -1$ \cite{Pop1996}.
\item
The compositum $K=\bbQ_{\textnormal{sym}}$ of all Galois extensions
of $\bbQ$ with a symmetric Galois group
\cite[Theorem~18.10.3]{FriedJarden2005}.
\end{enumerate}
\end{corollary}

\begin{proof}
If $K/\bbQ$ is Galois, then, since any subfield $K_0$ of $K$
contains $\bbQ$, $K/K_0$ is Galois and hence not PAC (as an
extension).
\end{proof}

A somewhat stronger result is the following.

\begin{theorem}\label{thm:GaloisClosurePAC}
The Galois closure of an algebraic separable proper PAC extension is
the separable closure.
\end{theorem}

\begin{proof}
Let $K/K_0$ be a proper PAC extension and let $N$ be its Galois
closure. By Corollary~\ref{cor:Razon}, there exists a separable
extension $N_0/K_0$, linearly disjoint from $K$ over $K_0$, such
that $N = N_0 K$. In particular, $N/N_0$ is a proper Galois PAC
extension, which implies, by Lemma~\ref{lem:noGaloisExt}, that $N$
is separably closed.
\end{proof}

Theorem~\ref{thm:GaloisClosurePAC} is a generalization of
Chatzidakis' result from 1986. Chatzidakis prove that if $K$ is a
countable Hilbertian and $e\geq 1$ an integer, then for almost all
$\bfsigma \in \gal(K)^e$ the field $K_s(\bfsigma)$ is a Galois
extension of no proper subextension $K\subseteq K' \subsetneq
K_s(\bfsigma)$. (To see this, recall that $K_s(\bfsigma)/K$ is PAC
for almost all $\bfsigma$.)

Let us discuss finite PAC extensions. It is clear that every PAC
field is a PAC extension of itself. Another trivial example for a
finite PAC extension is $C/R$, where $C$ is the algebraic closure of
a real closed field $R$. In this case $[C:R]=2$ \cite[VI\S
Corollary~9.3]{Lang2002}.

A family of finite PAC extensions are the purely inseparable ones:
Let $K$ be a purely inseparable extension of a PAC field $K_0$. Then
$K/K_0$ is PAC \cite[Corollary~2.3]{JardenRazon1994}. Hence if
$K/K_0$ is finite purely inseparable extension and $K_0$ is PAC,
then $K/K_0$ is also a finite PAC extension.

The following result asserts that, in fact, these are all the finite
PAC extensions.

\begin{corollary}
\label{cor: finite PAC extension}%
Let $K/K_0$ be a finite extension. Then $K/K_0$ is PAC if and only
if either
\begin{enumerate}
\item $K_0$ is real closed and $K$ is the algebraic closure of
$K_0$, or
\item $K_0$ is PAC and $K/K_0$ is purely inseparable.
\end{enumerate}
\end{corollary}

\begin{proof}
Let $K_1$ be the maximal separable extension of $K_0$ contained
in $K$. Then $K/K_1$ is purely inseparable \cite[V\S6
Proposition 6.6]{Lang2002}. By \cite[Corollary
2.3]{JardenRazon1994} $K_1/K_0$ is PAC, and in particular $K_1$
is a PAC field. If $K_1 = K_0$, we are done.

Assume $K_1\neq K_0$. By Theorem~\ref{thm:GaloisClosurePAC}, the
Galois closure $N$ of $K_1/K_0$ is the separable closure. Hence, by
Artin-Schreier Theorem \cite[Corollary 9.3]{Lang2002} $N$ is, in
fact, algebraically closed and $K_0$ is real closed (recall that
$1<[N:K_0]<\infty$). In particular, the characteristic of $K$ is
$0$, and hence $K_1=K$.
\end{proof}

\chapter{Projective Pairs}\label{chapter:projectivepairs}
In this chapter we define the group theoretic analog of PAC
extensions -- \textit{projective pairs}. This analogy motivates the
study of these pairs.

This study influences the understanding of PAC extension in two
ways. First we have a characterization of PAC extensions of PAC
fields in term of projective pairs.

\begin{proposition}\label{prop:projpair-PACext}
\begin{enumerate}
\item
Let $M$ be a PAC extension of a PAC field $K$. Then the pair
$(\gal(M), \gal(K))$ is projective.
\item
Let $M$ be an algebraic extension of a PAC field $K$. Then $M/K$ is
PAC if and only if the restriction map $(\gal(M),\gal(K))$ is
projective.
\item
Let $(\Gamma, \Lambda)$ be a projective pair. Then there exists a
separable algebraic PAC extension $M$ of a PAC field $K$ such that
$\Gamma\cong \gal(M)$, $\Lambda\cong \gal(K)$.
\end{enumerate}
\end{proposition}

(See the definition of a projective pair and the proof of this
result in the body of the chapter.)

Although we cannot directly translate the results on projective
pairs to results on PAC extensions (of a non-PAC field), the ideas
from the study of the former do apply to the latter. This explains
the extensive use of group theoretic methods which appears in the
previous chapters. In particular, the key property, the lifting
property, comes from group theoretic considerations.

Using Proposition~\ref{prop:projpair-PACext} we can transfer result
about PAC extensions to result about projective pairs. Nevertheless
in this chapter we give direct proofs without going via
Proposition~\ref{prop:projpair-PACext} for two reasons.

First the group theoretic group are usually simpler and sometimes
give a stronger results. The second reason is that we work in a
slightly more general setting. Namely we work in the category of
pro-$\C$ groups, where $\C$ is a Melnikov formation of finite groups
(see below).

\section{The Basic Properties of Projective Pairs}
\subsection{Definitions}\label{sec:ppDefinition}
Throughout this chapter we fix a Melnikov formation $\C$ of finite
groups. That means that $\C$ is closed under fiber products and
given a short exact sequence
\[
\xymatrix@1{%
1\ar[r]& A \ar[r]& B \ar[r]& C\ar[r]&1
}%
\]
we have that $A,C\in \C$ if and only if $B\in \C$. In particular,
$\C$ is closed under direct products. The exact sequence
\[
\xymatrix@1{%
1\ar[r]& A^{(G:G_0)} \ar[r]& A\wr_{G_0} G \ar[r]& G\ar[r]&1
}%
\]
implies that $A,G\in \C$ if and only if $A\wr_{G_0} G\in \C$.

The following three families are examples of Melnikov formations.
The family of all finite groups; the family of all $p$-groups, for a
prime $p$; the family of all solvable groups.

Let $\Gamma\leq \Lambda$ be pro-$C$ groups. A $\C$-DEP for the pair
$(\Gamma, \Lambda)$ is a commutative diagram
\begin{eqnarray}
\label{double embedding problem}%
\xymatrix@C=40pt{%
    &\Lambda\ar[dr]^{\nu}\ar@{.>}[dl]_\theta\\
H\ar[rr]^(.4)\beta
        &&B\\
    &\Gamma\ar[dr]^{\mu}\ar'[u]^{\phi}[uu]\ar[dl]_{\eta}\\
,G\ar[rr]^{\alpha}\ar[uu]^{j}
        &&A\ar[uu]^{i}
}%
\end{eqnarray}
where $G,H,A,B\in \C$, $A\leq B$, $G\leq H$, $i,j,\phi$ are the
inclusion maps, and $\alpha,\mu,\beta,\nu$ are surjective. Therefore
a $\C$-DEP consists of two compatible $\C$ embedding problems: the
\textbf{lower} embedding problem for $\Gamma$ and the
\textbf{higher} embedding problem for $\Lambda$.

In case $\C$ is the family of all finite groups, we omit the $\C$
notation and say that \eqref{double embedding problem} is a double
embedding problem (as in \eqref{eq:DoubleEP}). Sometimes we
abbreviate \eqref{double embedding problem} and write
$((\mu,\alpha),(\nu,\beta))$.

A $\C$-DEP is said to be \textbf{split} if the higher embedding
problem splits, i.e., in \eqref{double embedding problem} there
exists $\beta'\colon B\to H$ for which $\beta\beta'=\id_B$. If,
in addition, the lower embedding problem splits, then we call
the DEP \textbf{doubly split}.

If we allow the groups $G,H,A,B$ to be pro-$\C$, then \eqref{double
embedding problem} is a \textbf{pro-$\C$-DEP}.

Giving a weak solution $\eta\colon \Gamma\to G$ of the lower
embedding problem and a weak solution $\theta\colon \Lambda\to H$ of
the higher embedding problem, we say that $(\eta,\theta)$ is a
\textbf{weak solution} of \eqref{double embedding problem} if $\eta$
and $\theta$ are compatible, i.e.\ $j\eta = \theta\phi$.

Note that a weak solution of \eqref{double embedding problem} is
completely determined by $\theta$. Indeed, $\theta$ induces a
solution $\eta$ of the $\C$-DEP if and only if $\theta(\Gamma)\leq
G$, and then $\eta=\theta|_{\Gamma}$.

\begin{definition}\label{def:pp}
A pair $(\Gamma, \Lambda)$ of pro-$\C$ groups is called
\textbf{$\C$-projective} if any $\C$-DEP is weakly solvable.
\end{definition}

Let us begin with some basic properties of projective pairs. The
first is trivial.

\begin{proposition}\label{prop:Cproj-simplecharac}
A pro-$\C$ group $\Lambda$ is $\C$-projective if and only if the
pair $(1, \Lambda)$ is $\C$-projective.
\end{proposition}

\begin{lemma}\label{lem:DEPdominated}
Consider a $\C$-DEP \eqref{double embedding problem} for a pair
$(\Gamma, \Lambda)$ of pro-$\C$ groups. Assume that both the higher
and lower embedding problems are weakly solvable. Then \eqref{double
embedding problem} is dominated by a doubly split $\C$-DEP.
\end{lemma}

\begin{proof}
Let $\theta\colon \Lambda \to H$ be a weak solution of the higher
embedding problem and $\eta\colon \Gamma\to G$ a weak solution of
the lower embedding problem. Choose an open normal subgroup $N\leq
\Lambda$ such that $N\leq \ker(\theta)$ and $\Gamma\cap N\leq
\ker(\eta)$.

Let $\Bhat=\Lambda/N$, $\Ahat=\Gamma/\Gamma\cap N$ and let $\Hhat =
H\times_B \Bhat$, $\Ghat = G\times_A \Ahat$. Then the following
commutative diagram defines a dominating $\C$-DEP.
\[
\xymatrix{%
    &\Gamma\ar[d]\\
\Ghat\ar[r]\ar[d]
    &\Ahat\ar[d]\\
G\ar[r]^\alpha
    &A
}%
\qquad
\xymatrix{%
    &\Lambda\ar[d]\\
\Hhat\ar[r]\ar[d]
    &\Bhat\ar[d]\\
H\ar[r]^\beta
    &B
}%
\]
(Here all the maps are canonically defined.)
Lemma~\ref{lem:dominatin_fiber_product} implies that this $\C$-DEP
doubly splits.
\end{proof}

\begin{corollary}\label{cor:ppreductiontodoublesplit}
Let $(\Gamma,\Lambda)$ be a pair of pro-$\C$ groups and suppose that
$\Lambda$ is $\C$-projective. Then $(\Gamma,\Lambda)$ is
$\C$-projective if and only if every doubly split $\C$-DEP is weakly
solvable.
\end{corollary}

\begin{proof}
Since $\Lambda$ is $\C$-projective, $\Gamma$ is also
$\C$-projective. In other words, every finite embedding problem for
$\Lambda$ (resp.\ $\Gamma$) is weakly solvable.
Lemma~\ref{lem:DEPdominated} implies that every $\C$-DEP for
$(\Gamma,\Lambda)$ is dominated by a doubly split $\C$-DEP.

The converse is trivial.
\end{proof}

\begin{proposition}
\label{prop infinite EPs}%
Let $(\Gamma,\Lambda)$ be a $\C$-projective pair. Then any
pro-$\C$-DEP for $\phi$ is weakly solvable.
\end{proposition}

\begin{proof}
We follow the proof of \cite[Lemma 22.3.2]{FriedJarden2005}. In
order to solve pro-$\C$-DEPs for $(\Gamma,\Lambda)$ we need to solve
more general pro-$\C$-DEPs, in which the maps of the higher
embedding problem are not necessarily surjective.

In case of $\C$-DEPs, we can solve such $\C$-DEPs. Indeed, assume
that in \eqref{double embedding problem} $\nu,\mu$ are not
surjective. First $\ker(\alpha),\ker(\beta)\in \C$ since $\C$ is a
Melnikov formation. Next $\nu(\Gamma),\mu(\Lambda)\in \C$ since
$\Gamma,\Lambda$ are pro-$\C$. Finally
$\alpha^{-1}(\nu(\Gamma)),\beta^{-1}(\mu(\Lambda))\in\C$, again
since $\C$ is a Melnikov formation.
\[
\xymatrix{%
1\ar[r]&\ker(\alpha)\ar[r]&\alpha^{-1}(\nu(\Gamma))\ar[r]^(0.6)\alpha&\nu(\Gamma)\ar[r]&1\\
1\ar[r]&\ker(\beta)\ar[r]&\beta^{-1}(\mu(\Lambda))\ar[r]^(0.6)\beta&\mu(\Lambda)\ar[r]&1
}%
\]
Replace $A,B$ with $\nu(\Gamma),\mu(\Lambda)$ and $G,H$ with
$\alpha^{-1}(\nu(\Gamma)),\beta^{-1}(\mu(\Lambda))$. In this new
$\C$-DEP all the maps are surjective. So by assumption there is a
weak solution.

Let us move to the more general case of pro-$\C$-DEP: Consider a
pro-$\C$-DEP \eqref{double embedding problem} and write $K =
\ker(\beta)$. We prove the assertion in two steps.

\vspace{10pt}%
\noindent\textsc{Step A:} \textit{Finite Kernel.} Assume $K$ is
finite. Then $G$ is open in $KG$ since $(KG:G)\leq |K|$. Choose an
open normal subgroup $U \leq H$ for which $U\cap KG \leq G$ and
$K\cap U = 1$ (note that $K$ is finite and $H$ is Hausdorff). Then
$U\cap KG = U\cap G$. By the second isomorphism theorem (in the
group $UG$) we have that $(KG\cap UG:G) = (U\cap (KG\cap UG): U\cap
G)=(U\cap KG:U\cap G)=1$, i.e.\
\begin{equation}
\label{eq KG intersection UG = G}
   (K G) \cap (U G) = G.
\end{equation}
Write $\Hgag = H/U$, let $\pi \colon H\to \Hgag$ be the quotient
map, $\Ggag = \pi(G)$, $\Bgag = B/\beta(U)$, $\Agag=A/A\cap\beta(U)$
and $\betagag\colon \Hgag\to \Bgag$, $\alphagag\colon \Ggag\to
\Agag$ the epimorphisms induced from $\beta$, $\alpha$,
respectively.

Since $\Hgag\in \C$, there is a homomorphism $\thetagag\colon
\Lambda \to \Hgag$ with $\thetagag(\Gamma)\leq \Ggag$ (let
$\etagag=\thetagag|_{\Gamma}$) for which
\[
\xymatrix{%
    &\Gamma\ar[d]^{\nu}\ar@{.>}[ddl]^(0.7){\etagag}\\
G\ar[r]^{\alpha}\ar[d]
    & A\ar[d]\\
\Ggag\ar[r]^-{\alphagag}
    &\Agag
}%
\qquad
\xymatrix{%
            &&&\Lambda\ar[d]^{\mu}\ar@{.>}[ddl]^(0.7){\thetagag}\\
1\ar[r]
    & K \ar[r]\ar@{=}[d]
        & H \ar[r]^{\beta}\ar[d]_{\pi}
            &B \ar[d] \ar[r]
                &1\\
1\ar[r]
    & K \ar[r]
        & \Hgag \ar[r]^{\betagag}
            &\Bgag \ar[r]
                &1
}%
\]
are commutative diagrams. The right square in the right embedding
problem is a cartesian square, since $K\cap U = 1$ (\cite[Example
22.2.7(c)]{FriedJarden2005}). Hence we can lift $\thetagag$ to
$\theta \colon \Lambda \to H$ such that $\beta \theta=\mu$ (see
\cite[Lemma 22.2.1]{FriedJarden2005}). We claim that
$\theta(\Gamma)\leq G$. Indeed,
\[
A \geq \mu(\Gamma) = \beta(\theta(\Gamma)),
\]
hence $\theta(\Gamma) \leq K \beta^{-1}(A) = K \alpha^{-1}(A) =
KG$. Also,
\[
\Ggag \geq \thetagag(\Gamma) = \pi(\theta(\Gamma)),
\]
hence $\theta(\Gamma)\leq UG$. Then, from \eqref{eq KG intersection
UG = G} we have $\theta(\Gamma) \leq (KG) \cap (UG) = G$, as
claimed.

\vspace{10pt}%
\noindent\textsc{Step B:} \textit{The General Case.} We use Zorn's
Lemma. Consider the family of pairs $(L,\theta)$ where $L\subseteq
K$ is normal in $H$, $\theta$ is a weak solution of the following
embedding problem, and $\theta(\Gamma) \subseteq GL/L$.
\[
\xymatrix{ &&&\Lambda\ar[d]\ar@{.>}[ld]_{\theta}
\\
1\ar[r]&K/L\ar[r]& H/L\ar[r]^{\bar\beta}&B\ar[r]&1. }
\]
We say that $(L,\theta) \leq (L',\theta')$ if $L\subseteq L'$ and
\[
\xymatrix{
&\Lambda\ar[d]\ar@{.>}[ld]_{\theta}\ar@{.>}[ldd]^{\theta'}
\\
H/L\ar[r]\ar[d]& B\ar@{=}[d]\\
H/L'\ar[r]& B}
\]
is commutative. For a chain $\{(L_i,\theta_i)\}$ we define a lower
bound $(L,\theta)$ by $L = \bigcap_{i}L_i$ and $\theta = \invlim
\theta_i$ (note that $\theta(\Gamma)\subset GL/L$ by
\cite[Lemma~1.2.2(b)]{FriedJarden2005}). By Zorn's Lemma there
exists a minimal element $(L,\theta)$ in the family. We claim that
$L = 1$. Otherwise, there is an open normal subgroup $U$ of $H$ with
$L \not\leq U$. Part A gives (since $L/U\cap L$ is finite) a weak
solution $\theta'$ of the following embedding problem such that
$\theta'(\Gamma) \subseteq G(U\cap L)/(U\cap L)$.
\[
\xymatrix{ &&&\Lambda\ar[d]_{\theta}\ar@{.>}[dl]_{\theta'}
\\
1\ar[r]&L/U\cap L\ar[r] &H/U\cap L \ar[r]& H/L. }
\]
Hence $(L,\theta)$ is not minimal. This contradiction implies that
$L=1$, as claimed.
\end{proof}

Recall that a pro-$\C$ group $\Lambda$ is projective if and only if
any short exact sequence of pro-$\C$ groups
\[
\xymatrix{1\ar[r]&K\ar[r]& \Delta\ar[r]& \Lambda\ar[r]&1}
\]
splits. Similar characterization is given in the next result for a
pair $(\Gamma,\Lambda)$ of pro-$\C$ groups.

\begin{corollary}\label{cor:splitting}
Let $(\Gamma,\Lambda)$ be a pair of pro-$\C$ groups. Then
$(\Gamma,\Lambda)$ is $\C$-projective if and only if the rows of
any exact commutative diagram of pro-$\calC$ groups
\[
\xymatrix{%
\Delta \ar[r]^\beta
    & \Lambda \ar[r]\ar@<1ex>@{.>}[l]^{\beta'}
        &1 \\
E \ar[r]^\alpha \ar[u]^{\psi}
    & \Gamma\ar[r]\ar[u]_{\phi}\ar@<1ex>@{.>}[l]^{\alpha'}
        &1
        \\
1\ar[u]&1\ar[u]
}%
\]
\emph{compatibly} split. That is to say, there exists a splitting
$\beta'\colon \Lambda\to \Delta$ of $\beta$ (i.e.\
$\beta\beta'=\id$) such that $\beta' \phi (\Gamma)\leq \psi(E)$ and
$\alpha'$ defined by $\psi\alpha'=\beta'\phi$ is a splitting of
$\alpha$.
\end{corollary}

\begin{proof}
Since the splittings of an epimorphism $\gamma\colon M\to N$
correspond bijectively to solutions of the embedding problem
$(\id\colon N\to N, \gamma\colon M\to N)$, the assertion follows
immediately from Proposition~\ref{prop infinite EPs}.
\end{proof}

\subsection{Basic Properties and Characterizations}

\begin{proposition}[The lifting property]
Let $(\Gamma,\Lambda)$ be $\C$-projective and consider a
pro-$\C$-DEP for $(\Gamma,\Lambda)$. Then any weak solution $\eta$
of the lower embedding problem can be lifted to a weak solution
$(\eta,\theta)$ of the DEP.
\end{proposition}

\begin{proof}
Let \eqref{double embedding problem} be a pro-$\C$ DEP and let
$\eta\colon \Gamma\to G$ be a weak solution of the lower
embedding problem. Define $\Hhat = H \times_B \Lambda$ and let
$\Ghat =\{(\eta(\gamma),\gamma )\mid \gamma\in \Gamma\}\leq
G\times_A \Gamma$. Since $\Ghat\cong \Gamma$, there is a unique
weak solution of the lower embedding problem of
\[
\xymatrix{
    & \Gamma\ar[d]^{\id}\\
\Ghat\ar[r]^{\alphahat}
    & \Gamma
}\qquad \xymatrix{
    & \Lambda\ar[d]^{\id}\\
\Hhat\ar[r]^{\betahat}
    & \Lambda.}
\]
Now by Proposition~\ref{prop infinite EPs} there exists a weak
solution $(\etahat,\thetahat)$. Hence $\etahat(\gamma) =
\alphahat^{-1}(\gamma) = (\eta(\gamma),\gamma )$. This implies that
$(\eta,\theta)$ is a solution of the DEP, where
$\theta=\pi\thetahat$ and $\pi\colon \Hhat\to H$ is the quotient
map.
\end{proof}

The next result follows from the lifting property using the same
argument that implied Corollary~\ref{cor:splitting} from
Proposition~\ref{prop infinite EPs}.

\begin{corollary}\label{cor:lifting_splitting}
Let $(\Gamma,\Lambda)$ be a $\C$-projective pair and consider a
diagram as in Corollary~\ref{cor:splitting}. Then any splitting
$\alpha'$ of $\alpha$ can be lifted to a splitting $\beta'$ of
$\beta$.
\end{corollary}

\begin{proposition}[Transitivity]
Let $\Lambda_3\leq \Lambda_2\leq \Lambda_1$ be pro-$\C$ groups.
Then
\begin{enumerate}
\item
    If $(\Lambda_3,\Lambda_1)$ is $\C$-projective, then $(\Lambda_3,\Lambda_2)$ is
    $\C$-projective.
\item
    If $(\Lambda_3,\Lambda_2)$ and $(\Lambda_2,\Lambda_1)$ are $\C$-projective, then so is
    $(\Lambda_3,\Lambda_1)$.
\end{enumerate}
\end{proposition}

\begin{proof}
To show (a) it suffices to solve a doubly split $\C$-DEP
\begin{equation}\label{eq:DEP13to12}
\xymatrix{%
    &\Lambda_3\ar[d]^{\phi_3}
\\
G_3\ar[r]^{\alpha_3}
    &A_3
}%
\qquad%
\xymatrix{%
        &&&\Lambda_2\ar[d]^{\phi_2}
\\
1\ar[r]
    &K_2\ar[r]
        &G_2\ar[r]^{\alpha_2}
            &A_2\ar[r]
                &1
}%
\end{equation}
for $(\Lambda_3,\Lambda_2)$. Fix a splitting of $\alpha_2$ to
identify $G_2$ and $K_2\rtimes A_2$. We can assume that
$\ker(\phi_2) = \Lambda_2\cap L$, where $L$ is open normal subgroup
of $\Lambda_1$. (Otherwise, we can choose an open normal subgroup
$L$ of $\Lambda_1$ such that $\ker(\phi_2) \geq \Lambda_2\cap L$ and
replace $A_i$ with $\Lambda_i/(\Lambda_i\cap L)$ and $G_i$ with the
corresponding fiber product, $i=2,3$.)

Let $A_1 = \Lambda_1/L$, let $\phi_1\colon \Lambda_1\to A_1$ be the
quotient map. Embed $G_3$ inside $K_2\wr_{A_2} A_1$ as in
Proposition~\ref{prop:embeddingtwotowreath}. In particular, under
the Shapiro map
\[
\pi\colon \Ind_{A_2}^{A_1}(K_2) \rtimes A_2 \to K_2\rtimes A_2
\]
we have $\pi(G_3) = G_3$. The $\C$-DEP
\[
\xymatrix{%
    &\Lambda_3\ar[d]^{\phi_3}
\\
G_3\ar[r]^{\alpha_3}
    &A_3
}%
\qquad%
\xymatrix{%
        &\Lambda_1\ar[d]^{\phi_1}
\\
K_2\wr_{A_2}A_1\ar[r]^(.6){\alpha_2}
        &A_1
}%
\]
has a weak solution $(\eta'_3,\eta'_1)$ since
$(\Lambda_3,\Lambda_1)$ is $C$-projective. Let $\eta_i = \pi
\eta'_1|_{\Lambda_i}$, $i=2,3$. Then $(\eta_3,\eta_2)$ is a weak
solution of \eqref{eq:DEP13to12} and the proof of (a) is concluded.

The proof of (b) is analogous to the proof of
Lemma~\ref{lem:transitive}, one merely need to disregard the word
`geometric' everywhere it appears.

For the sake of completeness we give a proof based on exact
sequences. Consider the following commutative diagram with injective
rows.
\[
\xymatrix{ \Delta_1\ar[r]^{\alpha_1}
     &\Lambda_1\ar[r]
          &1\\
\Delta_2\ar[r]^{\alpha_2}\ar[u]
     &\Lambda_2\ar[r]\ar[u]
          &1\\
\Delta_3\ar[r]^{\alpha_3}\ar[u]
     &\Lambda_3\ar[r]\ar[u]
          &1\\
}
\]
By Corollary~\ref{cor:splitting} it suffices to show that the higher
and lower rows compatibly split. As $(\Lambda_3,\Lambda_2)$ is
$\C$-projective, the two bottom rows compatibly split, let
$\alpha_3'$, $\alpha_2'$ be the corresponding sections. Now by
Corollary~\ref{cor:lifting_splitting}, $\alpha_2'$ extends to a
section $\alpha_1'$ of $\alpha_1$. Thus $\alpha_3'$, $\alpha_1'$ are
compatible sections, as needed.
\end{proof}

\begin{remark}
In the above proposition it might happen that
$(\Lambda_3,\Lambda_1)$ is $\C$-projective but
$(\Lambda_2,\Lambda_1)$ is not. For example, take $\Lambda_1$ to be
$\C$-projective, take $1\neq \Lambda_2\normal \Lambda_1$ and
$\Lambda_3=1$. Then $(1, \Lambda_1)$ is $\C$-projective
(Proposition~\ref{prop:Cproj-simplecharac}) while if $\Lambda_2\neq
\Lambda_1$, then $(\Lambda_2\leq \Lambda_1)$ is not
(Proposition~\ref{prop:normalprojectivepairs}).
\end{remark}

Projective pairs behave well under taking subgroups.
\begin{proposition}\label{prop:ppundersubgroups}
Let $( \Gamma, \Lambda)$ be a $\C$-projective pair, let
$\Lambda_0\leq \Lambda$ be a subgroup, and write $\Gamma_0 =
\Lambda_0\cap \Gamma$. Then $( \Gamma_0, \Lambda_0)$ is
$\C$-projective.
\end{proposition}

\begin{proof}
By Lemma~\ref{lem:DEPdominated} it suffices to show that any doubly
split $\C$-DEP for $( \Gamma_0, \Lambda_0)$ is weakly solvable. Let
\[
((\mu_1\colon \Gamma_0 \to A_1, \alpha_1\colon G_1\to A_1),
(\nu_1\colon \Lambda_0 \to B_1, \beta_1\colon H_1\to B_1))
\]
be a doubly split $\C$-DEP for $( \Gamma_0, \Lambda_0)$.

\vskip5pt %
\noindent\textsc{The case where $\Lambda_0$ is open in $\Lambda$.}
Choose an open normal subgroup $N$ of $\Lambda$ such that $N\leq
\ker \nu_1$. Then $N\cap\Gamma_0\leq \ker(\mu_1)$. Let $\nu\colon
\Lambda \to B=\Lambda/N$ be the quotient map. Let $B_0$, $A$, and
$A_0$ be the respective images of $\Lambda_0$, $\Gamma$, and
$\Gamma_0$ under $\nu$ and let $\nu_0$, $\mu$, and $\mu_0$ be the
restrictions of $\nu$ to the respective subgroups. Let
$G_0=G_1\times_{A_1} A_0$, $H_0 = H_1\times_{B_1} B_0$ and let
$\alpha_0\colon G_0\to A_0$ and $\beta_0\colon H_0\to B_0$ be the
projection maps. Then the doubly split $\C$-DEP
$((\mu_0,\alpha_0),(\nu_0,\beta_0))$ dominates
$((\mu_1,\alpha_1),(\nu_1,\beta_1))$
(Lemma~\ref{lem:dominatin_fiber_product}). Hence it suffices to
weakly solve $((\mu_0,\alpha_0),(\nu_0,\beta_0))$.

Clearly $A_0\leq B_0 \cap A$. Let $\xgag\in B_0 \cap A$. Then $\xgag
= \nu(x)$, where $x = y n$, $x\in \Lambda_0$, $y\in \Gamma$, and
$n\in N$. Since $N\leq \Lambda_0$ we get that $y = xn^{-1} \in
\Lambda_0$, hence $y\in \Gamma_0$. Since $\nu(y) = \nu(x)=\xgag$ we
have $\nu(\Gamma_0)\geq \nu (\Lambda_0) \cap \nu (\Gamma)$.
Consequently $A_0 = B_0 \cap A$.

If $\alpha_0'\colon A_0\to G_0$ is a splitting of $\alpha_0$, then
it suffices to find a weak solution $(\eta_0,\theta_0)$ such that
$\eta_0(\Gamma_0) = \alpha_0'(A_0)$. Therefore, we can replace $G_0$
by $\alpha_0'(A_0)$ to assume that $G_0\cong A_0$.

Let $\beta_0'\colon B_0\to H_0$ be a splitting of $\beta_0$.
Identify $H_0$ with $K\rtimes B_0$ via $\beta_0'$, where $K = \ker
\beta_0$. This identification induces an embedding $i\colon A_0 \to
K_0\rtimes B_0$ satisfying $\beta_0 i = \alpha_0$.

\[
\xymatrix{
K_0\rtimes B_0 \ar@{>>}[r]^{\beta_0}&B_0\ar@{^(->}[r]&B\\
A_0\ar[r]^{\alpha_0}\ar@{^(->}[u]^{i}&A_0\ar@{^(->}[r]\ar@{^(->}[u]&A\ar@{^(->}[u]
    }
\]

Let $\beta\colon K\wr_{B_0} B\to B$ be the quotient map and
$\pi\colon \Ind_{B_0}^B(K)\rtimes B_0\to K\rtimes B_0$ the
corresponding Shapiro map.
Proposition~\ref{prop:embeddingtwotowreath} gives an embedding
$j\colon A \to K \wr_{B_0} B$ such that $\pi j|_{A_0} = i$, and in
particular $\beta j (a) = a$, for all $a\in A$. Let
$\alpha=\beta|_{j(A)} \colon j(A)\to A$.

\[
\xymatrix@!C@C20pt{
    &\Lambda\ar[dl]_\theta\ar[dr]^{\nu}\\
K\wr_{B_0} B \ar[rr]^(.4){\beta}
        &&B\\
    &\Gamma\ar'[u][uu]\ar[dl]_{\eta}\ar[dr]^{\mu}\\
j(A)\ar[rr]^{\alpha}\ar[uu]
        &&A\ar[uu]
    }
\]

Since $(\Gamma,\Lambda)$ is $\C$-projective, there exists a weak
solution $(\eta,\theta)$ of $((\mu,\alpha),(\nu,\beta))$. Set
$\eta_0 = \pi \eta|_{\Gamma_0}$ and $\theta_0 = \pi
\theta|_{\Lambda_0}$. Then $(\eta_0,\theta_0)$ is a weak solution of
$((\mu_0,\alpha_0),(\nu_0,\beta_0))$.

\[
\xymatrix{
    &\Gamma_0\ar@/^7pt/[ddl]^{\eta_0}\ar[dl]_{\eta}\ar[d]^{\mu_0}\\
j(A_0) \ar[r]^(.55){\alpha}\ar[d]^\pi
        &A_0\\
A_0 \ar[r]^{\alpha_0}
        &A_0\ar@{=}[u]
        }
\qquad \xymatrix{
    &\Lambda_0\ar@/^10pt/[ddl]^{\theta_0}\ar[dl]_{\theta}\ar[d]^{\nu_0}\\
\Ind_{B_0}^B(K)\rtimes B_0 \ar[r]^(.65){\beta}\ar[d]^\pi
        &B_0\\
K\rtimes B_0 \ar[r]^{\beta_0}
        &B_0.
\ar@{=}[u]
       }
\]

\vskip5pt %
\noindent\textsc{The general case.} Fix an open normal subgroup $N$
of $\Lambda$ such that $N\cap \Lambda_0 \leq \ker\nu_1$. Then $N\cap
\Gamma_0 \leq \ker (\mu_1)$. We can extend $\nu_1$ to $\Lambda_0 N$
by setting $\nu_1(\lambda n) = \nu_1(\lambda)$ for each $\lambda\in
\Lambda_0$ and $n\in N$. Similarly, $\Gamma_0 (N\cap \Gamma)$ is an
open subgroup of $\Gamma$ and $\mu_1$ extends to it. Note that the
restriction of $\nu_1$ to $\Gamma_0 (N\cap \Gamma)$ equals $\mu_1$.

By \cite[Lemma 1.2.2(b)]{FriedJarden2005} we have an open subgroup
$\Lambda'$ of $\Lambda$ that contains $\Lambda_0 N$ and such that
$\Lambda'\cap \Gamma\leq \Gamma_0 (N\cap \Gamma)$. Let $\Gamma' =
\Lambda' \cap \Gamma$. Now we have the $\C$-DEP
\[
((\mu_1\colon \Gamma' \to A_1, \alpha_1\colon G_1\to A_1),
(\nu_1\colon \Lambda' \to B_1, \beta_1\colon H_1\to B_1))
\]
for $(\Gamma',\Lambda')$. By the first part $(\Gamma',\Lambda')$ is
$\C$-projective, hence there exists a weak solution $(\eta,\theta)$
of this DEP. Now if we set $\eta_0 =\eta|_{\Gamma_0}$ and $\theta_0
= \theta|_{\Lambda_0}$, we get the weak solution $(\eta_0,\theta_0)$
of the original DEP.
\end{proof}

\begin{remark}
The analogous result for PAC extensions -- if $M/K$ is a PAC
extension and $L/K$ algebraic, then $ML/L$ is PAC -- is proved with
a descent argument (see \cite[Lemma 2.1]{JardenRazon1994}). This has
suggested us that wreath product should be the tool for this group
theoretic version.
\end{remark}

\begin{corollary}\label{cor:splittingpp}
Let $(\Gamma, \Lambda)$ be a $\C$-projective pair and let $N\leq
\Gamma$. Then there exists $M\leq \Lambda$ such that $N = \Gamma
\cap M$ and $\Gamma\cdot M = \Lambda$. Moreover, if $N\normal
\Gamma$, then $M\normal \Lambda$.
\end{corollary}

\begin{proof}
Let $\Nhat = \bigcap_{\sigma} N^\sigma$ and let $\eta\colon \Gamma
\to \Gamma/\Nhat$ be the natural quotient map. Lift $\eta$ to a
solution $(\eta,\theta)$ of the DEP
\[
((\Gamma\to 1, \Gamma/\Nhat\to 1), (\Lambda\to 1, \Gamma/\Nhat\to
1)).
\]
Let $\Mhat = \ker(\theta)$ and $M = \theta^{-1}(N/\Nhat)$. Then
(since $\eta = \theta|_{\Gamma}$)
\[
N = \eta^{-1}(N/\Nhat) = \Gamma \cap \theta^{-1}(N/\Nhat) =
\Gamma\cap M.
\]
Since $\theta(\Gamma) = \Gamma/\Nhat$, it follows  that $\Gamma
\Mhat = \Lambda$, and in particular, $\Gamma \cdot M = \Lambda$.

To conclude the proof, note that if $N\normal\Gamma$, then
$N=\Nhat$, and hence $M=\Mhat$. So $M\normal \Lambda$, as needed.
\end{proof}

Taking $N=1$ in the above result we get the following splitting
corollary.

\begin{corollary}\label{cor:semidirectproductpp}
If $(\Gamma, \Lambda)$ is $\C$-projective, then $\Lambda \cong M
\rtimes \Gamma$ for some $M\normal \Lambda$.
\end{corollary}

\section{Families of Projective Pairs}
In this section, for simplicity, we assume that $\C$ is the
family of all finite groups.
\subsection{Free Products}
Let $\Gamma_1,\Gamma_2$ be profinite groups. The (profinite)
free product $\Gamma = \Gamma_1\ast \Gamma_2$ is defined by the
following universal property. Let $(\mu\colon \Gamma \to A,
\alpha\colon H\to A)$ be an arbitrary embedding problem for
$\Lambda$. Let $A_1 = \mu(\Gamma_1)$, $A_2 = \mu(\Gamma_2)$,
$H_1=\alpha^{-1}(A_1)$, and $H_2=\alpha^{-1}(A_2)$. Also let
$\mu_1,\mu_2$ the respective restriction of $\mu$ to
$\Gamma_1,\Gamma_2$ and $\alpha_1,\alpha_2$ the respective
restriction of $\alpha$ to $H_1,H_2$. Then for every weak
solutions $\theta_1,\theta_2$ of $(\mu_1,\alpha_1)$ and
$(\mu_2,\alpha_2)$ there exists a unique lifting $\theta$ to a
solution of $(\mu,\alpha)$.

\begin{proposition}\label{prop_freefactor}
Let $\Gamma$ be a free factor of a projective group $\Lambda$. Then
$(\Gamma, \Lambda)$ is projective.
\end{proposition}

\begin{proof}
By assumption $\Lambda = \Gamma\ast N$. The subgroups $\Gamma$,
$N$ are projective as closed subgroups of a projective group.
Consider a finte DEP \eqref{double embedding problem}. Let $A_1
= \nu(N)$ and $G_1 = \beta^{-1}(A_1)$. Let $\eta_1,\eta$ be
respective weak solutions of $(\mu_1,\alpha_1)$ and
$(\mu,\alpha)$, where $\mu_1=\nu|_{N}$ and $\alpha_1 =
\beta|_{G_1}$.

By the definition of free products, we can lift $\eta_1,\eta$ to a
weak solution $\theta$ of $(\nu,\beta)$. Hence $(\eta,\theta)$ is a
weak solution of \eqref{double embedding problem}.
\end{proof}

The following result appears in \cite{HaranLubotzky1985}.

\begin{lemma}[Haran and Lubotzky]\label{lem:ppHL}
Let $\kappa$ be a cardinal and let $P$ be a projective profinite
group of rank $\leq \kappa$. Then $\Fhat_\kappa \cong P\ast
\Fhat_{\kappa}$.
\end{lemma}

As a consequence we get a family of examples of projective pairs
inside a free group.

\begin{corollary}\label{cor:ppHL}
Let $\Lambda$ be a free profinite group of infinite rank $\kappa$.
Then for any projective group $\Gamma$ of rank $\leq \kappa$ there
exists an embedding of $\Gamma$ in $\Lambda$ such that
$(\Gamma,\Lambda)$ is projective.
\end{corollary}

\subsection{Random Subgroups}
Let us start by fixing some notation. We write $e$-tuples in
bold face letters, e.g., $\bfb=(b_1\nek b_e)$. For a
homomorphism of profinite groups $\beta\colon H\to B$, we write
that $\beta(\bfh)=\bfb$ if $\beta(h_i) = b_i$ for all $i=1\nek
e$. Let $C$ be a coset of a subgroup $N^e$ in $H^e$, where
$N\leq \ker(\beta)$. By abuse of notation we write $\beta(C)$
for the unique element $\bfb\in B^e$ such that
$\beta(\bfc)=\bfb$ for every $\bfc \in C$.

\begin{proposition}
Let $\Lambda = \Fhat_\omega$. Then for almost all $\bfsigma \in
\Lambda^e$, $(\left<\bfsigma\right>, \Lambda)$ is projective.
\end{proposition}

\begin{proof}
As a profinite group, $\Lambda$ is equipped with a normalized Haar
measure $m$. Let
\begin{eqnarray}
\label{eq_ep}%
\fEP{\Lambda}{\mu}{H}{\beta}{B}
\end{eqnarray}
be a finite embedding problem for $\Lambda$, let $\bfb \in B^e$, let
$A = \<\bfb\>$ be the subgroup of $B$ generated by $\bfb$, and let
$\bfh \in H^e$ be such that $\beta(\bfh) = \bfb$. Define $\Sigma =
\Sigma(\bfb,\bfh,\mu,\beta) \subseteq \Lambda^e$ to be the following
set.
\[
\Sigma=\{\bfsigma \in \Lambda^e\mid (\mu(\bfsigma) = \bfb)\
\Rightarrow\ (\exists \theta\colon \Lambda \to H, \
(\beta\theta=\mu)\wedge (\theta(\bfsigma)=\bfh))\},
\]
that is to say, all $\bfsigma \in \Lambda^e$ such that there exists
a weak solution $\theta$ of \eqref{eq_ep} with $\theta(\bfsigma) =
\bfh$, provided $\mu(\bfsigma) = \bfb$. Note that $\Sigma =
(\Sigma\cap C)\cup (\Lambda^e \smallsetminus C)$, where $C$ is the
coset of $\ker(\mu)^e$ in $\Lambda^e$ for which $\mu(C)=\bfb$.

We break the proof into three parts. In the first two we show that
$m(\Sigma\cap C) = m(C)$, and hence $m(\Sigma) = 1$.

\scsltitle{Part A:}{Construction of solutions.} Let
\[
\Delta = \{ (b_i)\in H^{\bbN} \mid \beta(b_i) = \beta(b_j)\ \forall
i,j\in \bbN\}.
\]
It is equipped with canonical projections $\pi_i\colon \Delta\to
H_i$, $i\in \bbN$. Set $\betahat\colon \Delta \to B$ by $\betahat(x)
= \beta\pi_i(x)$, $x\in \Delta$. Note that this definition is
independent of $i$ and $\betahat$ is an epimorphism.

Let $\theta\colon \Lambda \to \Delta$ be a solution of $(\mu\colon
\Lambda \to B, \betahat\colon \Delta\to B)$ (for the existence of
$\theta$ see \cite[Theorem~3.5.9]{RibesZalesskii}). Then for every
$i\in \bbN$ the map $\theta_i = \pi_i\theta$ is a solution of
\eqref{eq_ep}. Moreover, by Lemma~\ref{lem:independent_fiber} the
set $\{\ker(\theta_i)\}$ is an independent set of subgroups of
$\ker(\mu)$.

\vspace{10pt} \scsltitle{Part B:}{Calculating $m(\Sigma)$.} For each
$i\in \bbN$ take the coset $X_i$ of $\ker(\theta_i)^e$ with
$\theta_i(X_i) = \bfh$. Then, since
\[
\mu(X_i) = \beta \theta_i(X_i) = \beta(\bfh) = \bfb,
\]
it follows that $X_i \subseteq C$. Moreover, from it follows Part A
that $\{X_i\mid i\in \bbN\}$ is an independent set in $C$.

By Borel-Cantelli Lemma \cite[Lemma 18.3.5]{FriedJarden2005},
$\sum_i m_{C}(X_i) = \sum_i \frac{|B|^e}{|H|^e} = \infty$ implies
that $m_C(X) = 1$, where $m_C$ is the normalized Haar measure on $C$
and $X = \cap_{j = 1}^\infty \cup_{i=j}^\infty X_i$. So it suffices
to show that $X\subseteq \Sigma$.

Indeed, let $\bfsigma \in X$. Then $\bfsigma \in X_i$ for some $i$.
It implies that $\theta_i$ is a solution of \eqref{eq_ep} and that
$\theta_i(\bfsigma) =\bfh$. Hence $\bfsigma\in \Sigma$ and
$X\subseteq \Sigma$, as desired.
\vspace{10pt}\scsltitle{Part C:}{Conclusion.} Let $\Upsilon$ be the
intersection of all $\Sigma(\bfb,\bfh,\mu,\beta)$. Since there are
only countably many of them and each is of measure $1$, we have
$m(\Upsilon)=1$. Let $\bfsigma \in \Upsilon$ and let $\Gamma =
\<\bfsigma\>$.

Then $\Gamma\leq\Lambda$ is projective. Indeed, consider a double
embedding problem as in \eqref{double embedding problem} and choose
$\bfh\in G$ such that $\beta(\bfh)=\mu(\bfsigma)$. Then, since
$\bfsigma \in \Sigma(\mu(\bfsigma), \bfh ,\mu,\beta)$, there exists
an homomorphism $\theta\colon \Lambda \to H$ such that
$\theta(\Gamma)=\<\theta(\bfsigma)\>=\<\bfh\>\leq G$.
\end{proof}

\begin{remark}
In the above theorem we actually prove  that for almost all
$\bfsigma\in \Lambda^e$ the pair $(\<\bfsigma\> ,\Lambda)$ has the
following strong lifting property. For any embedding problem
\eqref{double embedding problem} and for any $\bfh\in G^e$ that
satisfies $\alpha(\bfh)=\mu(\bfsigma)$ there exists a weak solution
$\theta\colon \Lambda\to B$ with $\theta(\bfsigma) = \bfh$.
\end{remark}

\section{Restrictions on Projective Pairs}
We prove the analogs for projective pairs of PAC extensions
properties.

\begin{lemma}\label{lem:normalprojectivepairs}
Let $(\Gamma, \Lambda)$ be a $\C$-projective pair and assume that
$\Gamma\normal \Lambda$. Then either $\Gamma=1$ or $\Gamma=\Lambda$.
\end{lemma}

\begin{proof}
Let $N\normal \Gamma$ be an open normal subgroup of $\Gamma$ and let
$D\normal \Lambda$ be an open normal subgroup of $\Lambda$ such that
$\Gamma\leq D$. Write $A=\Gamma/N$ and $G = \Lambda/D$ and let
$\eta\colon \Gamma\to A$ and $\mu \colon \Lambda \to G$ be the
natural quotient maps. Identify $A$ with the subgroup $\{f_a\mid
a\in A\}$ of $A\wr G$. (Recall that $f_a(1)=a$ and $f_a(\sigma)=1$
for any $\sigma\neq 1$.)

Extend $\eta$ to a weak solution $\theta$ of
\[
\xymatrix{%
    &\Gamma\ar[d] \ar[ld]_{\eta}\\
A\ar[r]
    &1
}%
\qquad
\xymatrix{%
    &\Lambda\ar[d]^{\mu} \ar[ld]_{\theta}\\
A\wr G\ar[r]^-{\pi}
    &G.
}%
\]
Lemma~\ref{lem:criterionforpropernesswreath} implies that $\theta$
is surjective (since $A=\theta(\Gamma)\leq \theta(\Lambda)$). Since
$\Gamma\normal \Lambda$ we have $A\normal A\wr G$. This evidently
implies that either $A=1$ or $G=1$, and hence the assertion.
\end{proof}

\begin{proposition}\label{prop:normalprojectivepairs}
Let $(\Gamma,\Lambda)$ be a $\C$-projective pair such that
$\Gamma\neq \Lambda$. Then $\bigcap_{x\in \Lambda}\Gamma^x = 1$.
\end{proposition}

\begin{proof}
Let $\Gamma_0 = \bigcap_{x\in \Lambda}\Gamma^x$. By
Corollary~\ref{cor:splittingpp} there exists $\Lambda_0$ such that
$\Gamma_0 = \Lambda_0\cap \Gamma$ and $\Gamma\Lambda_0=\Lambda$. In
particular $\Gamma_0 \neq \Lambda_0$. Moreover, by
Proposition~\ref{prop:ppundersubgroups}, $(\Gamma_0, \Lambda_0)$ is
a $\C$-projective pair. But since $\Gamma_0\normal \Lambda_0$ and
$\Gamma_0\neq \Lambda_0$, Lemma~\ref{lem:normalprojectivepairs}
implies $\Gamma_0=1$.
\end{proof}

\begin{corollary}\label{cor:openpp}
Let $(\Gamma,\Lambda)$ be a $\C$-projective pair. Assume that
$\Gamma$ is open in $\Lambda$. Then $\Gamma=\Lambda$.
\end{corollary}

\begin{proof}
Assume that $\Gamma\neq \Lambda$ (and in particular, $\Lambda\neq
1$). Since $\Gamma$ is open, the normal core $\bigcap_{x\in \Lambda}
\Gamma^x$ is also open. By
Proposition~\ref{prop:normalprojectivepairs}, $\bigcap_{x\in
\Lambda} \Gamma^x=1$. Consequently $\Lambda/\bigcap_{x\in \Lambda}
\Gamma^x = \Lambda$ is finite. This contradicts the fact that
$\Lambda$ is $\C$-projective.
\end{proof}

\begin{proposition}
Let $\Lambda$ be a $\C$-profinite group and $\Gamma$ a $p$-Sylow
subgroup. Assume that $\Lambda$ has a non-abelian simple quotient
that is divisible by $p$. Then the pair $(\Gamma,\Lambda)$ is not
$\C$-projective.
\end{proposition}

\begin{proof}
Assume the contrary, i.e.\ $(\Gamma,\Lambda)$ is $\C$-projective.
Hence, by Corollary~\ref{cor:semidirectproductpp}, we have that
$\Lambda = M\rtimes \Gamma$. Note that $p\nmid (\Lambda:\Gamma) =
|M|$ since $\Gamma$ is $p$-Sylow.

Let $\psi\colon \Lambda \to S$ be an epimorphism onto a non-abelian
simple group of order divisible by $p$. Then $\psi(M)\neq S$. On the
one hand, $\psi(M)=1$, since $\psi(M)\normal S$. On the other hand,
$\psi(\Gamma)$ is a proper subgroup of $S$. (Otherwise $S$ would be
a $p$-group, which is solvable.) The assertion now follows from the
contradiction $S = \psi(\Lambda) = \psi(M)\psi(\Gamma) =
\psi(\Gamma)< S$.
\end{proof}

\section{Applications to PAC Extensions}

\subsection{Proof of Proposition~\ref{prop:projpair-PACext}}
\begin{trivlist}
\item
Let $M$ be a PAC extension of a PAC field $K$. Since $\gal(M)$ and
$\gal(M\cap K_s)$ are isomorphic and $M\cap K_s/K$ is PAC
(Theorem~\ref{thm:nonalgebraicPAC}), we can assume that $M/K$ is
separable and algebraic. Let $\Gamma = \gal(M)$ and
$\Lambda=\gal(K)$.

Since $K$ is PAC, $\Lambda$ is projective \cite[Theorem
11.6.2]{FriedJarden2005}. By
Corollary~\ref{cor:ppreductiontodoublesplit}, to show that $\phi$ is
projective it suffices to solve a doubly split double embedding
problem \eqref{double embedding problem}. Over PAC fields any finite
split embedding problem is rational \cite{Pop1996,HaranJarden1998},
and hence \eqref{double embedding problem} is rational by
definition. Hence there exists a weak solution
(Proposition~\ref{prop:ExistGeoSol}).

Next assume that $M/K$ is algebraic and separable and that
$(\gal(M),\gal(K))$ is projective. By
Proposition~\ref{prop:ExistGeoSol}, to show that $M/K$ is PAC it
suffices to geometrically solve (in the weak sense) all finite
rational double embedding problems. Since $(\gal(M),\gal(K))$ is
projective, the double embedding problem is weakly solvable.
Moreover, all weak solutions are geometric
(Corollary~\ref{cor:PACfield_solutionisgeometric}), as needed.

Let $(\Gamma,\Lambda)$ be a projective pair. We would like to find a
PAC extension $M/K$ such that $\Gamma = \gal(M)$, $\Lambda=\gal(K)$.

By \cite[Corollary 23.1.2]{FriedJarden2005}, there exists a PAC
field $K$ such that $\gal(K)\cong \Lambda$ (since $\Lambda$ is
projective). Let $M$ be the fixed field of $\Gamma$, i.e., $\gal(M)
= \Gamma$. Since $(\Gamma,\Lambda)$ is projective, by the third
paragraph, $M/K$ is PAC. \qed
\end{trivlist}

\subsection{New Examples of PAC Extensions}

\begin{proposition}
Let $K_0$ be a field which has a PAC extension $K/K_0$. Assume that
$\gal(K)$ is free of infinite rank $\kappa$. Then for any projective
group $P$ of rank $\leq \kappa$ there exists a PAC extension $M/K_0$
such that $P\cong \gal(M)$.
\end{proposition}

\begin{proof}
By Theorem~\ref{thm:nonalgebraicPAC}, we can assume that $K/K_0$ is
a separable algebraic extension. By Corollary~\ref{cor:ppHL} and
Proposition~\ref{prop:projpair-PACext} it follows that there exists
a (separable algebraic) PAC extension $M/K$ with $\gal(M) \cong P$.
Now transitivity of PAC extensions
(Proposition~\ref{prop:transitivityofPACextension}) implies that
$M/K_0$ is PAC.
\end{proof}

Recall that a Galois extension $N/K$ is unbounded if the set
$\{\ord(\sigma)\mid \sigma\in \gal(N/K)\}$ is unbounded.

\begin{corollary}\label{cor:ppAbelianExtension}
Let $P$ be a projective group of at most countable rank, let $E$ be
a countable Hilbertian field, and let $K_0/E$ be an unbounded
abelian extension. Then $K_0$ has a PAC extension $M$ such that
$P\cong \gal(M)$.
\end{corollary}

\begin{proof}
In the proof of \cite[Proposition 3.8]{Razon1997} it is shown that
there exists a PAC extension $K/K_0$ such that $\gal(K)\cong
\Fhat_\omega$. Hence the assertion follows from the previous
proposition.
\end{proof}

\begin{corollary}
Let $P$ be a projective profinite group. Then there exists a
Hilbertian field $K$ and a PAC extension $M/K$ such that
$\gal(M)\cong P$.
\end{corollary}

\chapter{Weak Hilbert's Irreducibility Theorem}
Recall that a field $K$ is \textbf{Hilbertian} if it satisfies the
following condition.
\begin{enumerate}
\renewcommand{\labelenumi}{\theenumi}
\renewcommand{\theenumi}{$\blacktriangle$}
\item\label{eq:irrSpeciallll}
Every irreducible polynomial $f(T,X)\in K(T)[X]$ that is separable
in $X$ admits a specialization $T\mapsto a\in K$ under which
$f(a,X)$ remains irreducible.
\end{enumerate}
Although it seems that Hilbertianity and PACness are contradictory,
there is a surprising link between them.
\begin{theorem}\label{thm:Roquette}
A PAC field $K$ is Hilbertian if and only if it is $\omega$-free.
\end{theorem}
Here a field $K$ is \textbf{$\omega$-free} if
\begin{enumerate}
\renewcommand{\labelenumi}{\theenumi}
\renewcommand{\theenumi}{$\vartriangle$}
\item\label{eq:omegafree}
Every finite embedding problem for $K$ is solvable.
\end{enumerate}

Roquette proved the right-to-left direction of this theorem (see
\cite[Corollary 27.3.3]{FriedJarden2005}). The opposite direction
was later proved by Fried-V\"olklein (in characteristic zero)
\cite{FriedVolklein1992} and Pop (in arbitrary characteristic)
\cite{Pop1996}. It is important to mention that Haran and Jarden
reprove this result using the simpler method of algebraic patching
\cite{HaranJarden1998}.

In this chapter we extend Theorem~\ref{thm:Roquette} and
establish a quantitative form -- the ``weak Hilbert's
irreducibility theorem." Note that Theorem~\ref{thm:Roquette}
assumes that \textbf{all} curves have rational points (i.e.\ $K$
is PAC), and under this assumption it connects the property
\ref{eq:omegafree} of \textbf{all} finite embedding problems for
$K$ with the property \ref{eq:irrSpeciallll} of \textbf{all}
irreducible polynomials. In the proof of
Theorem~\ref{thm:Roquette}, in order to show that a
specialization keeps the polynomial irreducible, a stronger
object is preserved -- the Galois group of the polynomial.

Our quantitative form, on the other hand, deals with one (arbitrary)
polynomial. We find a specific finite embedding problem with the
property that if it is `transitively solvable,' then there is an
irreducible specialization, provided that a certain curve has a
point. The main ingredient is Proposition~\ref{prop:
characterization of solutions}.
In the case that $K$ is PAC, this condition, i.e.\ the existence of
a transitive solution, becomes necessary and sufficient for the
existence of an irreducible specialization. It is somewhat
surprising that, in contrast to the above, there might exist an
irreducible specialization even if the Galois group of the
polynomial does not occur as Galois group over $K$.

Next we apply weak Hilbert's irreducibility theorem to a field $K$
which has a PAC extension. We get some sufficient conditions to have
irreducible specializations of a polynomial in terms of solutions of
some double embedding problem. This result has some interesting
applications which will be discuss in the following chapters.

\section{Embedding Problems and Polynomials}\label{sec:EPandPol}
Let $E$ be a finitely generated regular extension of a field $K$.
For a separable polynomial $f(X)\in E[X]$, let $F/E$ be its
splitting field, and let $L = F\cap K_s$. Recall that $\gal(F/E)$
acts faithfully on the roots of $f(X)$. Set $A=\gal(L/K)$ and
$G=\gal(F/E)$. Let $G_0\leq G$ be the stabilizer of some fixed root
$x\in F$ of $f$.

Now $f$ admits the \textbf{induced embedding problem}
\begin{equation}\label{eq:inducedEP}
\calE_f=(\nu\colon \gal(K)\to A, \alpha\colon G\to A)
\end{equation}
with a \textbf{distinguished} subgroup $G_0$. Here $\nu$ and
$\alpha$ are the restriction maps.

Let $\theta\colon \gal(K)\to G$ be a weak solution of the induced
embedding problem, and denote its image by $D=\theta(\gal(K))$. If
$D$ acts transitively on the roots of $f$, we say that $\theta$ is
\textbf{transitive}. Since $G_0$ is the stabilizer of the action, a
solution $\theta$ is transitive if and only if $(G:G_0) = (D:D\cap
G_0)$, that is, transitivity is determined by the distinguished
subgroup. The following lemma trivially holds.

\begin{lemma}
If $\theta$ is a solution (i.e.\ surjective), then $\theta$ is
transitive.
\end{lemma}

Next we give a criterion for a polynomial $f(X)\in E[X]$ to remain
irreducible after applying a place on $E$.

\begin{lemma}\label{lem:characterizationofirrspecialization}
Let $E$ and $f$ be as above and let $\phi$ be a $K$-rational place
of $E/K$. 
Assume that $\phi(f)$ is well defined, separable and of the same
degree of $f$. Then $\phi(f)$ is irreducible over $K$ if and only if
the geometric solution $\phi^*$ of the induced embedding problem is
transitive.
\end{lemma}

\begin{proof}
Let $\calX\subseteq F$ be the set of all roots of $f$ and
$\calXgag\subseteq K_s$ the set of all roots of $\phi(f)$. Then
$\phi(\calX) = \calXgag$. By assumptions, $|\calX|=\deg f = \deg
\phi(f) = |\calXgag|$, and thus $\phi$ bijectively maps $\calX$ onto
$\calXgag$.

Now $\phi(f)$ is irreducible if and only if $\gal(K)$ acts
transitively on $\calXgag$. Recall that $\phi(\phi^*(\sigma)(x)) =
\sigma (\phi(x))$ for all $\sigma\in \gal(K)$ and $x\in \calX$ (see
Lemma~\ref{lem:act_geo_sol}). Let $x_1,x_2\in \calX$ and
$\phi(x_1),\phi(x_2)\in \calXgag$. Then
\begin{eqnarray*}
\exists \sigma\in \gal(K), \sigma(\phi(x_1)) = \phi(x_2) &\Iff& \exists \sigma\in \gal(K), \phi(\phi^*(\sigma) (x_1)) = \phi(x_2)\\
&\Iff& \exists \sigma\in \gal(K), \phi^*(\sigma)(x_1) = x_2.
\end{eqnarray*}
In other words, $\gal(K)$ acts transitively on $\calXgag$ if and
only if $\phi^*$ acts transitively on $\calX$, that is $\phi^*$ is
transitive.
\end{proof}

The previous result in particular asserts that for an irreducible
polynomial to have a place under which it remains irreducible it is
necessary that the induced embedding problem have a transitive
solution. This condition also suffices, provided that certain
regular extension has a $K$-rational place:

\begin{lemma}\label{lem:irrplace}
Let $E/K$ be a finitely generated regular extension. Let $f\in E[X]$
be irreducible and separable. Assume that the induced embedding
problem is transitively solvable. Then there exists a finite
separable extension $\Ehat/E$ that is regular over $K$ satisfying
the following property. For any $K$-rational place $\phi$ of
$\Ehat$, $\phi(f)$ is irreducible, provided that $\phi(f)$ is well
defined, separable, and of the same degree as $f$.
\end{lemma}

\begin{proof}
Let $\theta$ be a transitive solution of the induced embedding
problem. Let $\Ehat$ be as in Proposition~\ref{prop:
characterization of solutions}. Then the residue field of $\Ehat$ is
$K$ implies that $\phi^* = \theta$. Thus
Lemma~\ref{lem:characterizationofirrspecialization} gives the
irreducibility of $\phi(f)$.
\end{proof}

As every finitely generated regular extension of a PAC field has a
rational place we get the following result.

\begin{proposition}
Let $K$ be a PAC field, let $E/K$ be a finitely generated regular
extension, and let $f(X)\in E[X]$ be an irreducible polynomial. Then
there exists a place $\phi$ of $E$ under which $\phi(f)$ is well
defined, irreducible, and of the same degree as $f$ if and only if
the induced embedding problem $\calE_f$ is transitively solvable.
\end{proposition}

\begin{remark}
To the best of our knowledge in previous works, to find an
irreducible specialization for a polynomial one always preserves the
Galois group of the polynomial.

The innovation of this work is that we do not insist that the Galois
group be preserved in order that $f$ remains irreducible. This
approach proves to be useful in applications. For example in
Chapter~\ref{chapter:Dirichlet} an irreducible specialization of a
polynomial whose Galois group is the symmetric group $S_n$ is
needed. The results of those chapters are valid e.g.\ for any
pro-solvable extension $K$ of $\bbQ$. Over such $K$ it may happen
that $S_n$ does not occur. Hence our observation is indeed important
to those applications.

Indeed, assume that $\bbQ_{\rm sol}$ and assume that $S_n$ occurs as
a Galois group over it. If $n\geq 2$, then as $\bbZ/2\bbZ$ is a
quotient of $S_n$, we get that $\bbZ/2\bbZ$ occurs as a Galois group
over $\bbQ_{\rm sol}$. This contradiction implies that $n=1$.
\end{remark}

\section{Weak Hilbert's Irreducibility Theorem for PAC Extensions}
In this section we consider a more classical setting. Instead of an
arbitrary regular extension $E/K$, we take $E=K(\bft)$, where $\bft$
is an $e$-tuple of algebraically independent transcendental
elements. In other words, we consider a polynomial $f(\bfT,X)$
defined over $K$ that is irreducible in the ring
$K[\bfT,X]$\footnote{If $f(\bfT,X)$ is irreducible in $K[\bfT,X]$,
then $f(\bft,X)$ is irreducible over $K(\bft)$ (by Gauss' lemma). On
the other hand, assume that $f(\bft,X)\in K(\bft)[X]$ is irreducible
over $K(\bft)$. Then there exists $c(\bft)\in K(\bft)$ such that
$f^*(\bfT,X)=c(\bfT)f(\bfT,X)\in K[\bfT,X]$ and is primitive. Hence
$f^*$ is irreducible in the ring $K[\bfT,X]$.} and asks whether
there is an irreducible specialization for $f$. That is to say, a
map $\bfT\mapsto\bfa\in K^e$ for which $f(\bfa,X)\in K[X]$ is
irreducible, separable, and $\deg(f(\bfa,X)) = \deg_X(f(\bfT,X))$.

Let $K$ be a field, $f\in K[\bfT,X]$ an irreducible polynomial that
is separable in $X$, and $\bft$ a tuple of algebraically independent
transcendental elements. Let $F$ be the splitting field of
$f(\bft,X)$ over $E=K(\bft)$, $L=F\cap K_s$, and %
\addtocounter{equation}{-1}
\begin{equation}\renewcommand{\theequation}{6.1}
\calE_f=(\mu\colon \gal(K)\to A, \alpha\colon G\to A)
\end{equation}
the induced embedding problem (see \eqref{eq:inducedEP}). Here $G =
\gal(F/K(\bft))$ and $A = \gal(L/K)$.

For a field extension $M/K$ that is algebraically independent of
$K(\bft)$ we define a corresponding embedding problem
\begin{equation}\label{eq:inducedcoresspondingEP}
\calE_f(M) = (\mu\colon \gal(M) \to B, \beta \colon H\to B).
\end{equation}
Here $N = LM$, $\Fhat=FM$, $B = \gal(N/M)\cong \gal(L/M\cap L)$, and
$H=\gal(\Fhat/M(\bft))\cong \gal (F/M\cap L(\bft))$. Then $B\leq A$
and $H\leq G$, so $\beta = \alpha|_H$.

\[
\xymatrix@C=10pt@R=10pt{
               &&&&\Fhat\\
            &&&F\ar@{-}[ur]\\
        &&M(\bft)\ar@{-}'[r][rr]
              &&N(\bft)\ar@{-}[uu] \\
K(\bft)\ar@{-}[r]
    &(M\cap L)(\bft)\ar@{-}[rr]\ar@{-}[ur]
             &&L(\bft)\ar@{-}[uu]\ar@{-}[ur]\\
        &&M\ar@{-}'[r][rr]\ar@{-}'[u][uu]
             &&N\ar@{-}[uu]\\
K\ar@{-}[r]\ar@{-}[uu]
    &M\cap L \ar@{-}[rr]\ar@{-}[uu]\ar@{-}[ur]
            &&L\ar@{-}[uu]\ar@{-}[ur]
}
\]

Note that the pair $(\calE_f(M),\calE_f)$ is a double embedding
problem for $M/K$.

\begin{proposition}\label{prop:PACextirr}
Let $K$ be a field, $f\in K[\bfT,X]$ an irreducible polynomial that
is separable in $X$, and $\calE_f=(\mu\colon \gal(K)\to A,
\alpha\colon G\to A)$ the induced embedding problem as in
\eqref{eq:inducedEP} with distinguished subgroup $G_0$. Further
assume that there exists a PAC extension $M/K$ (which is
algebraically independent of $K(\bft)$) and let $\calE_f(M)$ be the
corresponding embedding problem as in
\eqref{eq:inducedcoresspondingEP}.

Each of the following conditions suffices for the existence of an
irreducible specialization $\bfT\mapsto \bfa\in K^e$ for $f$.
\begin{enumerate}
\item \label{case:PACextirrproper}
$\calE_f(M)$ is solvable.
\item \label{case:PACextirrtransitive}
$(G:G_0) = (H:H_0)$, where $H_0 = H\cap G_0$ and $\calE_f(M)$ is
transitively solvable w.r.t.\ the distinguished subgroup $H_0$.
\end{enumerate}
\end{proposition}

\begin{proof}
Let $r(\bft)$ be the product of the leading coefficient of $f$ and
its discriminant. Assume there exists a weak solution $\eta\colon
\gal(M)\to H$ of $\calE_f(M)$.

By the lifting property
(Proposition~\ref{prop:ExtensionSoltoGeoSol_DEP}) we can lift $\eta$
to a geometric weak solution $\phi^*$ of $\calE_f$ satisfying
$\bfa=\phi(\bft)$ is finite (and hence in $K^e$) and $r(\bfa)\neq
0$. In particular, $\phi$ is unramified in $F$.

For \eqref{case:PACextirrproper} assume that $\eta$ is surjective.
Then its image $\eta(\gal(M))$ contains $\ker(\beta)$. But
$\ker(\beta) = \ker(\alpha)$, hence $\ker(\alpha) \leq
\eta(\gal(M))\leq \phi^*(\gal(K))$. By
Lemma~\ref{lem:criterionforproperness}, $\phi^*$ is surjective.
Consequently Lemma~\ref{lem:characterizationofirrspecialization}
implies the assertion.

For \eqref{case:PACextirrtransitive} assume that $(G:G_0)=(H:H_0)$
and that $\eta$ is transitive. Let $C = \eta(\gal(M))$ be the image
of $\eta$. Since $\eta$ is transitive we have $(C:C\cap H_0) =
(H:H_0)$ (see discussion after the definition of transitive
solutions). Let $D = \phi^*(\gal(K))$. Then $C \leq H\cap D$, hence
$C\cap H_0 = C\cap H\cap G_0 = C\cap G_0$. It implies that
\[
(G:G_0) \geq (D:D\cap G_0) \geq 
 = (C:C\cap H_0) = (H:H_0)=(G:G_0).
\]
Hence $(G:G_0) = (D:D\cap G_0)$, i.e.\ $\phi^*$ is transitive. Then
Lemma~\ref{lem:characterizationofirrspecialization} implies the
assertion.
\end{proof}

\begin{remark}\label{rem:overM}
In the proof we actually showed that $f(\bfa, X)$ is irreducible
over $M$.
\end{remark}

\begin{remark}[On the condition \eqref{case:PACextirrtransitive}]
In the notation of the above proposition, recall that $G_0$ is the
stabilizer of a root $x\in F$ of $f(\bft,X)$, so
$(G:G_0)=\deg_X(f)$. Now let $g(\bft,X)$ be the irreducible
polynomial of $x$ over $M(\bft)$. Then $g$ divides $f$ in the ring
$M(\bft)[X]$ and $(H:H_0)=\deg_X g$. Thus $(H:H_0)=(G:G_0)$ if and
only if $f$ is irreducible over $M(\bft)$.

In particular, if $f$ is absolutely irreducible, the condition
$(G:G_0)=(H:H_0)$ always holds.

Another example in which the condition $(G:G_0) = (H:H_0)$ holds is
when $A\cong B$. Indeed, since $\ker(\alpha) = \ker(\beta)$ we get
\[
|G| = |A| | \ker(\alpha)| = |B| |\ker(\beta)| = |H|,
\]
and thus $G\cong H$. Consequently $(G:G_0) = (H:H_0)$.
\end{remark}

In what follows we apply Proposition~\ref{prop:PACextirr} to
polynomials which are in some sense the most irreducible. These cases
are important to applications, since by transcendental
constructions, it is usually easier to get extreme irreducibilities.

Call a polynomial $f(\bfT,X)\in K[\bfT,X]$ of degree $n$ in $X$ a
\textbf{stable symmetric polynomial} if the Galois group of
$f(\bft,X)$ over $\Kgal(\bft)$ is the symmetric group $S_n$. Regard
stable symmetric polynomial as the most irreducible polynomials. A
stable symmetric polynomial is absolutely irreducible and its
splitting field over $K(\bft)$ is regular over $K$. Thus the induced
embedding problem is actually the following realization problem.
\[
\calE_f = (\gal(K) \to 1, S_n \to 1)
\]
Here the distinguished subgroup is $S_{n-1}$ (all permutations that
fix $1$).

\begin{remark}
Let $E/K$ be a regular extension. One can asks whether there exists
a separating transcendence basis $\bft$ of $E/K$ such that a
polynomial $f(\bft,X)\in K(\bft)[X]$ whose root generates
$E/K(\bft)$ is a stable symmetric polynomial of degree $n$.

The stability theorem (see \cite[Theorem 18.9.3]{FriedJarden2005})
asserts that indeed for infinitely many positive integers $n$ there
exists such a separating transcendence basis.
\end{remark}

We have the following nice sufficient condition to have irreducible
specialization of stable symmetric polynomials.

\begin{corollary}\label{cor:wHIT_stablesymmetric}
Let $K$ be a field, let $r(\bfT)\in K[\bfT]$ non-constant, and let
$f(\bfT,X)\in K[\bfT,X]$ be a stable symmetric polynomial of degree
$n$. Assume that there exists a PAC extension $M/K$ and a separable
extension $N/M$ of degree $n$. Then there exists $\bfa\in K^e$
such that $r(\bfa)\neq 0$ and $f(\bfa,X)$ is separable irreducible
polynomial of degree $n$.
\end{corollary}

\begin{proof}
Since $F$ (the splitting field of $f(\bft,X)$ over $K(\bft)$) is
regular over $K$, we have $\gal(FM/M(\bft))\cong
\gal(F/K(\bft))\cong S_n$. Hence
\[
\calE_f(M) = (\gal(M) \to 1 , S_n \to 1)
\]
is the corresponding embedding problem.

Let $\Nhat$ be the Galois closure of $N/M$ and let $\eta\colon
\gal(M)\to \gal(\Nhat/M)$ be the restriction map. The action of
$\gal(\Nhat/M)$ on the cosets of $\gal(\Nhat/N)$ induces an
embedding $\gal(\Nhat/M)\to S_n$ with a transitive image. Identify
$\gal(\Nhat/M)$ with its image in $S_n$. Hence $\eta$ is a
transitive solution of the corresponding embedding problem. By
Proposition~\ref{prop:PACextirr} the assertion follows.
\end{proof}

\section{Some Corollaries}
Theorem~\ref{thm:intAlmostallPAC} from the introduction asserts
that a countable Hilbertian field has a PAC extension with free
absolute Galois group of rank $e$, for an arbitrary positive
integer $e\geq 1$. The following result implies that having such
PAC extension is also sufficient for Hilbertianity.

\begin{corollary}
Let $K$ be a field. Assume that there exists a PAC extension $M/K$
whose absolute Galois group $\gal(M)$ is free of rank $e$, for
infinitely many positive integers $e$. Then $K$ is Hilbertian.
\end{corollary}

\begin{proof}
Let $f\in K[\bfT,X]$ be an irreducible polynomial that is separable
in $X$ and let $\calE_f$ be its induced embedding problem as in
\eqref{eq:inducedEP}. Take a PAC extension $M/K$ whose absolute
Galois group $\gal(M)$ is free of rank $\geq |G|$.

Let $\calE_f(M)$ be the corresponding embedding problem as in
\eqref{eq:inducedcoresspondingEP}. We have
\[
\rank \gal(M) \geq |G| \geq |H| \geq \rank H.
\]
Hence, by the embedding property \cite[Proposition
17.7.3]{FriedJarden2005}, $\calE_f(M)$ is solvable. It follows that
$f$ has an irreducible specialization
(Proposition~\ref{prop:PACextirr}).
\end{proof}

If $K$ has a PAC extension $M$ whose absolute Galois group is
$\omega$-free (i.e.\ any finite embedding problem for $M$ is
solvable) then we get the following result of Razon \cite[Lemma
2.2]{Razon1997}.

\begin{corollary}[Razon]
Let $K$ be a field and $M/K$ a PAC extension whose absolute Galois
group is $\omega$-free. Then $M$ is Hilbertian over $K$. In
particular, $K$ is Hilbertian.
\end{corollary}

\begin{proof}
Let $f(\bfT,X) \in K[\bfT,X]$ be an irreducible polynomial that is
separable in $X$. Since over $M$ every finite embedding problem is
solvable, by Proposition~\ref{prop:PACextirr} there exists an
irreducible specialization $\bfT\mapsto \bfa \in K^e$ for $f$.
Moreover, $f(\bfa,X)$ is irreducible over $M$
(Remark~\ref{rem:overM}).
\end{proof}

We pose a natural question.
\begin{question}
Let $K$ be a finitely generated infinite field. Does there exist a
PAC extension $M/K$ whose absolute Galois group is $\omega$-free?
\end{question}

We even do not know the answer to the following simpler problem in
the case when $K=\bbQ$.
\begin{question}
Let $K$ be a finitely generated infinite field and let $M/K$ be a
PAC extension. Is $\gal(M)$ finitely generated?
\end{question}

\chapter{Dirichlet's Theorem for Polynomial Rings} \label{chapter:Dirichlet}
\section{Introduction and the Main Result}
Dirichlet's classical theorem about primes in arithmetic
progressions states that if $a,b$ are relatively prime positive
integers, then there are infinitely many $c\in\bbN$ such that
$a+bc$ is a prime number. Following a suggestion of E.~Landau,
Kornblum \cite{Kornblum1919} proved an analog of Dirichlet's
theorem for the ring of polynomials $F[X]$ over a finite field
$F$. Later, Artin refined Kornblum's result and proved that if
$a(X),b(X)\in F[X]$ are relatively prime, then for every
sufficiently large integer $n$ there exists $c(X)\in F[X]$ such
that $a(X)+b(X)c(X)$ is irreducible of degree $n$ in $F[X]$
\cite[Theorem 4.8]{Rosen2002}.

To avoid repetition, we shall say that Dirichlet's theorem holds for
a polynomial ring $F[X]$ and for a set of positive integers $\calN$,
if for any relatively prime polynomials $a,b\in F[X]$ there exist
infinitely many $c\in F[X]$ such that $a+bc$ is irreducible of
degree $n$, provided that $n\in \calN$ is sufficiently large.

Jarden raised the question of whether the Artin-Kornblum result
generalizes to other fields. Of course, if $F$ is algebraically
closed, then the polynomial $a(X)+b(X)c(X)$ is reducible unless it
is of degree $1$. On the other hand, if $F$ is Hilbertian, then
Dirichlet's theorem trivially holds. Indeed, $a(X)+b(X)Y$ is
irreducible, and hence there are infinitely many $c\in F$ such that
$a(X)+b(X)c$ is irreducible in $F[X]$. To get irreducible
polynomials of higher degree in this case, just choose $c_0(X)\in
F[X]$ relatively prime to $a(X)$ and of high degree, and then repeat
the above argument for $a(X)+b(X)c_0(X)Y$.

By the weak Hilbert's irreducibility theorem, this simple argument
can be extended to fields having PAC extensions. We start by
introducing a piece of notation: For a field $F$ we let $\calN(F)$
be the set of all positive integers for which there exists a PAC
extension $M/F$ and a separable extension $N/M$ such that $n=[N:M]$.

\begin{theorem}\label{thm A}
Let $F$ be a field and $\calN = \calN(F)$. Then Dirichlet's theorem
holds for $F$ and $\calN$.
\end{theorem}

In Section~\ref{sec:3} we calculate $\calN(F)$ for certain
families of fields $F$. For example, if $F$ is a pro-solvable
extension of a countable Hilbertian fields, then $\calN(F)$
contains every $n\geq 5$.

Let $n\in \calN$ be sufficiently large. By
Corollary~\ref{cor:wHIT_stablesymmetric} and the above argument, it
suffices to find $c\in F[X]$ of degree $n$ such that $f(Y,X)=a(X) +
b(X) c(X) Y$ is of degree $n$ and its Galois group over $\Fgal(Y)$
is $S_n$. So the proof of Theorem~\ref{thm A} reduces to the
following

\begin{proposition}\label{thmB}
Let $F$ be an infinite field with an algebraic closure $\Fgal$ and
let $a,b\in F[X]$ be relatively prime polynomials. Then for any
$n\geq 2\max(\deg a,\deg b) + \log n (1+o(1))$ there exists $c(X)\in
F[X]$ such that $f(X,Y)=a(X)+b(X)c(X)Y$ is irreducible over $F(Y)$
of degree $n$ in $X$ and $\gal(f(X,Y),\Fgal(Y))\isom S_n$.
\end{proposition}

Note that each infinite algebraic extension $F$ of a finite field
$K$ is PAC \cite[Corollary 11.2.4]{FriedJarden2005}. By
Theorem~\ref{thm A}, $F[X]$ satisfies Dirichlet's theorem for
$\calN=\{n\mid \exists N/M,\ n=[N:M]\}$. This result already follows
from a
quantitative form of the result of Artin-Kornblum. %
Nevertheless, our proof has the advantage that the constructions are
essentially explicit: The polynomial $c(X)$ in Theorem~\ref{thm A}
equals to the polynomial $c(X)$ appearing in Proposition~\ref{thmB}
times some factor, say $\alpha$, coming from the PAC property. The
construction in Proposition~\ref{thmB} is explicit as it uses
nothing but the Euclidean algorithm.

\section{Polynomials over Infinite Fields}
\subsection{Calculations with Polynomials}
The following result is a special case of Gauss' Lemma.
\begin{lemma}
\label{lem1}%
A polynomial $f(X,Y) =a(X)+b(X)Y\in F[X,Y]$ is irreducible if and
only if $a(X)$ and $b(X)$ are relatively prime.
\end{lemma}

\begin{lemma}
\label{lem_separable}%
Let $a,b,c\in F[X]$ such that $\gcd(a,b) = 1$ and $c\neq 0$. Then
there exists a finite subset $S\subseteq F$ such that for each
$\alpha\in F\smallsetminus S$ the polynomials $a + \alpha b$ and $c$
are relatively prime. Moreover, if $b'\neq 0$, we may choose $S$
such that $a +\alpha b$ is separable.
\end{lemma}

\begin{proof}
Let $S = \{-\frac{a(\gamma)}{b(\gamma)}\mid \gamma\in \Fgal,
b(\gamma)\neq 0 \mbox{ and } c(\gamma) = 0\}\cap F$. Then $a +
\alpha b$ has no common zero in $\Fgal$ with $c$ for each
$\alpha\not\in S$, hence these polynomials are relatively prime.

Let $d(Y)\in F[Y]$ be the discriminant of $a(X) + Y b(X)$ over
$F(Y)$, then $b'(X)\neq 0$ implies that $\frac{\partial }{\partial
X}(a(X)+b(X)Y)\neq 0$, and hence $d(Y)\neq 0$. In this case add all
the roots of $d(Y)$ to $S$.
\end{proof}

\begin{lemma}
\label{lem_final_poly}%
Let $a,b,p_1,p_2\in F[X]$ be pairwise relatively prime polynomials
and let $\alpha_1, \alpha_2 \in F$ be distinct nonzero elements.
Then for each $n > \deg p_1 + \deg p_2$ there exists $c\in F[X]$ of
degree $n$ and separable $h_1,h_2\in F[X]$ such that $a = p_i h_i +
b c \alpha_i$ and $\gcd(h_i, a p_i) = 1$ for $i=1,2$.
\end{lemma}

\begin{proof}
Write $b_i = b \alpha_i$. Since $\gcd(b_ip_{3-i},p_i) = 1$ for
$i=1,2$, there are $c_1,c_2,h_{1,0},h_{2,0}\in F[X]$ such that
\begin{equation} \label{eq a 0}
a = p_i h_{i,0} + b_i c_i p_{3-i} \qquad i=1,2,
\end{equation}
with $\deg c_i <\deg p_i$. For $\bar c = c_1 p_2 + c_2 p_1$ and
$h_{i,1} = h_{i,0} - c_{3-i} b_i$, we have
\begin{equation}\label{eq a 1}
a = p_i h_{i,1} + b_i \bar c \qquad i=1,2.
\end{equation}
Here $h_{i,1}$ is relatively prime to $b_i$, since $a$ is. Taking
\eqref{eq a 0} with $i=2$ and \eqref{eq a 1} with $i=1$, we get
\[
p_1 h_{1,1} \equiv a - b_1 \bar c \equiv a - b_1 c_{2} p_1  \equiv
b_2 c_2 p_1 - b_1 c_2 p_1 \equiv p_1 b c_2 (\alpha_2 - \alpha_1)
\pmod{p_2}.
\]
Therefore $h_{1,1}$ is relatively prime to $p_2$, since $bc_2p_1$ is
(by \eqref{eq a 0} with $i=2$). Similarly, $h_{2,1}$ is relatively
prime to $p_1$.

Take $c = \bar c + p_1 p_2 s$ for some $s\in F[X]$. Then, for $h_i =
h_{i,1} - b_i p_{3-i} s$, we have
\[
a = p_i h_i + b_i c \qquad i=1,2.
\]
To conclude the proof it suffices to find $s\in F[X]$ such that
$h_1$ and $h_2$ are separable, $\gcd(h_i, ap_i) = 1$, and $\deg c =
n$. Choose $s\in F[X]$ for which $\deg s = n - (\deg p_1 + \deg
p_2)\geq 1$, $(b p_{i} s)'\neq 0$, and $\gcd(s,h_{i,1}) = 1$ for
$i=1,2$ (e.g., $s(X) = (X-\beta)^{n-1}(X-\gamma)$, where
$\beta,\gamma\in F$ are not roots of $h_{1,1} h_{2,1} b p_1 p_2$.)

By Lemma~\ref{lem_separable} with $h_{i,1}$, $b p_{3-i}s$, $a p_i$
(for $i=1,2$) we get a finite set $S\subseteq F$ such that for each
$\alpha \in F \smallsetminus S$ the polynomial $h_{i,1} - \alpha b
p_{3-i} s$ is separable and relatively prime to $a p_i$. Replace $s$
with $\alpha s$, for some $\alpha\neq 0$ for which
$\alpha_i\alpha\in F\smallsetminus S$, if necessary, to assume that
$\alpha_1,\alpha_2\in F\smallsetminus S$. This $s$ has all the
required properties.
\end{proof}

\subsection{Other Lemmas}
The next lemma gives a sufficient criterion for a transitive group
to be primitive and to be the  symmetric group (cf. \cite[Lemma
4.4.3]{Serre1992}).

\begin{lemma}\label{lem_primitive}
Let $A\leq S_n$ be a transitive group and let $e$ be a positive
integer in the segment $\frac n2<e<n$ such that $\gcd(e,n)=1$. Then,
if $A$ contains an $e$-cycle, it is primitive. Moreover, if $A$
contains also a transposition, then $A = S_n$.
\end{lemma}

\begin{proof}
Let $\Delta \neq \{1,\ldots,n\}$ be a block of $A$. We have
$|\Delta| \leq \frac n2$, since $|\Delta|\mid n$. For the first
assertion, it suffices to show that $|\Delta| = 1$, and since
$\gcd(e,n)=1$, it even suffices to prove that $|\Delta|\mid e$.
Without loss of generality assume that $\sigma = (1\ 2\ \cdots\  e)
\in A$ and $1\in \Delta$. Then $\{1,\ldots,e\} \not\subseteq
\Delta$, since $e>\frac n2\geq |\Delta|$. Hence $\Delta \neq \sigma
\Delta$, and hence $\Delta\cap \sigma \Delta = \emptyset$. As
$\sigma(x) = x$ for any $n\geq x>e$, we have $\Delta \subseteq
\{1,\ldots,e\}$. Consequently, $\Delta$ is a block of
$\left<\sigma\right>$, so $|\Delta|\mid e$.

The second assertion follows since a primitive group containing a
transposition is the symmetric group \cite[Theorem
3.3A]{DixonMortimer}.
\end{proof}

The following number-theoretic lemma will be needed later.

\begin{lemma}\label{lem:numbertheoretic}
For any positive integers $n,m$ and prime $p$ satisfying $n\geq 2m +
\log n (1+o(1))$, there exists an integer $e$ in the segment $n-m> e
> n/2$ such that $\gcd(e,np)=1$.
\end{lemma}

\begin{proof}
Let $e$ be
\[
\begin{array}{ll}
\frac{n}{2}+2,
    & \mbox{if $n$ is even but not divisible by $4$,}\\
\frac{n}{2}+1,
    & \mbox{if $n$ is divisible by $4$, or}\\
\frac{n+1}{2},
    & \mbox{if $n$ is odd.}
\end{array}
\]
Then $e$ is the first integer greater than $\frac{n}{2}$ for which
$\gcd(e,n)=1$. If $p\nmid e$, we are done (and we only need $n>2m +
4$). Next assume that $p\mid e$ (and hence $p\nmid n$). Firstly, if
$n$ is even but not divisible by $4$, then the next candidate $e' =
e+2$ works, since $\gcd(e',n)=1$ and $p\nmid e'$ (otherwise, $p$
divides $e'-e=2$, hence $e$ is even, a contradiction). Next, if $n$
is divisible by $4$, then the first relatively prime to $n$ integer
greater than $e$ is $e' = \frac{n}{2}+q$, where $q$ is the smallest
prime not dividing $n$.

If $p\mid e'$, then $p\mid (e'-e) = q-1$. In particular, $p<q$, and
hence $p\mid n$ by minimality of $q$, a contradiction. Therefore
$p\nmid e'$. Finally, if $n$ is odd, the same argument will show
that $e' = \frac{n+q}{2}$ is relatively prime to $np$, where now $q$
is the smallest odd prime not dividing $n$.

It remains to evaluate $q$ -- a standard exercise in number theory:
Let $\omega(n)$ be the number of distinct prime divisors of $n$.
Then $q$ is no more than the $(\omega(n)+2)$-th prime number. Since
the $k$-th prime equals to $k\log k(1+o(1))$ and
$$\omega(n)\leq \frac{\log n}{\log\log n}(1 + o(1))$$
\cite[Theorem 2.10]{MontgomeryVaughan2007}, we have
\[
q\leq \omega(n)\log(\omega(n))(1+o(1)) = \log n (1+o(1)).
\]
Note that for $n = 4 \prod_{2<l<q} l$ (i.e., $4$ times the product
of all the odd prime numbers less than $q$) the inequality is in
fact equality. Thus the estimation $n>2m+\log(n)(1+o(1))$ is the
best possible.
\end{proof}

The following result is very well known, however, for the sake of
completeness, we give a proof.

\begin{proposition} \label{prop:cyclicelements}
Let $F$ be an algebraically closed field of characteristic $l\geq
0$. Let $E/K$ be a separable extension of degree $n$ of algebraic
function fields of one variable over $F$. Assume that a prime
divisor $\pfrak$ of $K$ factors as
\[
\pfrak = \Pfrak_1^{e_1} \cdots \Pfrak_r^{e_r}
\]
in $E$. Assume either
\begin{enumerate}
\item $l=0$,  or
\item $l>0$ and $\gcd(e_i,l)=1$, for $i=1,\ldots, r$, or
\item $l=e_r = 2$ and $e_1,\ldots, e_{r-1}$ are odd.
\end{enumerate}
Then the Galois group of the Galois closure of $E/K$ (as a subgroup
of $S_n$) contains an element of cyclic type $(e_1,\ldots,e_r)$.
\end{proposition}

\begin{proof}
The completion $\Khat$ of $K$ at $\pfrak$ is a field of Laurent
series over $F$ \cite[Theorem II.2]{Serre1979}, say $\Khat =
F((Y))$. Let $x$ be a primitive element of $E/K$, integral at
$\pfrak$ and let $f$ be its irreducible polynomial over $K$.

Then $E=K[X]/(f)$. Let $f = f_1 \cdots f_r$ be the factorization of
$f$ over $\Khat$ into a product of distinct separable irreducible
monic polynomials. Then by \cite[Theorem II.1]{Serre1979}, the
completion of $E$ at $\Pfrak_i$ is $\Khat[X]/(f_i)$, and hence $\deg
f_i = e_i$, for $i=1,\ldots,r$.

If either $l=0$, or $l>0$ and $\gcd(e_i,l)=1$, then $F((Y))$ has a
unique extension of degree $e_i$, namely $F((Y^{1/e_i}))$
\cite[III\S6]{Chevalley} which is cyclic. We thus get that the
splitting field of $f$ over $F((Y))$ is $F((Y^{1/e}))$, where
$e={\rm lcm}(e_1,\ldots,e_r)$, unless $l=e_r=2$ and then the
splitting field of $f$ is the compositum of $F((Y^{1/e'}))$ with an
extension of degree $2$, where $e' = {\rm lcm}(e_1,\ldots,e_{r-1})$
is odd. In both cases the Galois group of $f$ over $F((Y))$ is
cyclic of order $e$. Its generator $\sigma$ acts cyclicly on the
roots of each of the $f_i$'s. Consequently, the cyclic type of
$\sigma$ is $(e_1,\ldots,e_r)$, as required.
\end{proof}

\subsection{Proof of Proposition~\ref{thmB}}
\begin{trivlist}
\item
Let $f(X,Y) = a(X) + b(X)Y\in F[X,Y]$ be an irreducible polynomial.
For a large integer $n$ we need to find $c(X)\in F[X]$ such that
$f(X,c(X)Y)=a(X) + b(X) c(X)Y$ is irreducible of degree $n$ and the
Galois group of $f(X,c(X)Y)$ over $\Fgal(Y)$ is $S_n$.

Lemma~\ref{lem:numbertheoretic} with $m=\max\{\deg a(X), 2 +
\deg b(X)\}$ and $p = \textnormal{char}(F)$ gives (for $n\geq
2m+\log n(1+o(1))$) a positive integer $e>\frac{n}{2}$ such that
\begin{eqnarray}
&& \label{eq_n_large} n > \max\{\deg a(X), e+2+\deg b(X)\}, \\
\label{eq_p_notdivide} &&\gcd(e,np)=1\ (\mbox{or $\gcd(e,n)=1$, if
$p=0$}).
\end{eqnarray}
Let $\alpha_1\neq \alpha_2$ and $\gamma_1\neq \gamma_2$ be elements
of $F$ such that $\alpha_i$ is nonzero and $\gamma_i$ is not a root
of $a(X)b(X)$, $i=1,2$. In Lemma~\ref{lem_final_poly} we have
constructed (for $a$, $b$, $p_1 = (X-\gamma_1)^e$, $p_2 =
(X-\gamma_2)^2$, $\alpha_1 $, and $\alpha_2$) a polynomial $c(X)\in
F[X]$ of degree $\deg c = n - \deg b(X)$ which is relatively prime
to $a(X)$ such that
\begin{eqnarray}
\label{eq:ram 1}%
f(X,c(X)\alpha_1 ) &=& a(X) + \alpha_1b(X) c(X) =  (X - \gamma_1)^e h_1(X),\\
\label{eq:ram alpha}%
f(X,c(X)\alpha_2 ) &=& a(X) + \alpha_2 b(X) c(X) = (X-\gamma_2)^2
h_2(X).
\end{eqnarray}
Here $h_1(X),h_2(X)\in F[X]$ are separable polynomials which are
relatively prime to $(X-\gamma_1)a(X)$, $(X-\gamma_2)a(X)$
respectively. In particular $\gcd(a,c)=1$, and hence $f(X,c(X)Y)$ is
irreducible (Lemma~\ref{lem1}). By \eqref{eq_n_large}, $\deg_X
f(X,c(X) Y)  = \deg b(X) +\deg c(X) = n$. Taking
\eqref{eq_p_notdivide}, \eqref{eq:ram 1}, and \eqref{eq:ram alpha}
in mind, Proposition~\ref{prop:cyclicelements} with $\pfrak =
(Y-\alpha_1)$ gives us an $e$-cycle in $\gal(f(X,c(X)Y),\Fgal(Y))$
and with $\pfrak=(Y-\alpha_2)$ gives a transposition. Thus
$\gal(f(X,c(X)Y),\Fgal(Y))=S_n$ (Lemma~\ref{lem_primitive}) as
needed. \qed
\end{trivlist}

\begin{question}
Let $f(X,Y)\in F[X,Y]$ be an absolutely irreducible polynomial. For
large $n$, is there a polynomial $c(X)\in F[X]$ for which (a)
$f(X,c(X)Y)$ is an $X$-stable polynomial of degree $n$? (b)
$\gal(f(X,c(X)Y), \Fgal(Y)) \cong S_n$?
\end{question}

\chapter{Families of PAC Extensions}
The aim of this chapter is to produce several new families of PAC
extensions. First we find that some fields (with high probability in
some sense) are PAC extensions over any subfield which is not
algebraic over a finite field. As an application we solve Problem
18.7.8 in \cite{FriedJarden2005} for finitely generated fields.

In the second section we fix a field and find PAC extensions of it.
Specifically we get that an extension of a countable Hilbertian
field having certain group theoretic features (e.g.\ pro-solvable
extensions) has many PAC extensions. This implies the results of
Chapter \ref{chapter:Dirichlet} and Theorem~\ref{thm:intTrace} for
those fields.

\section{PAC Extensions over Subfields}
Recall that Jarden and Razon prove that if $K$ is a countable
Hilbertian field, then $K_s(\bfsigma)/K$ is PAC for almost all
$\bfsigma\in \gal(K)^e$. In this section we generalize this result.

Let us start by introducing some notation. Let
\[
f_1(T_1,\ldots, T_e ,X_1, \ldots, X_n),\ldots, f_m(\bfT,\bfX)\in
K[\bfT,\bfX] \] be irreducible polynomials and $g(\bfT)\in K[\bfT]$
nonzero. The corresponding \textbf{Hilbert set} is the set of of all
irreducible specializations $\bfT\mapsto \bfa$ for $f_1,\ldots, f_m$
under which $g$ does not vanish, i.e.
\[
H_K(f_1,\ldots, f_m;g) = \{ \bfa\in K^r \mid \forall i\:
f_i(\bfa,\bfX) \mbox{ is irreducible in } K[\bfX]\ \mbox{and} \
g(\bfa)\neq 0\}.
\]
Now $K$ is \textbf{Hilbertian} if any Hilbert set is nonempty
provided that $f_i=f_i(\bfT,X)$ is separable in $X$ for each $i$.
(Some authors use the terminology `$K$ is separable Hilbertian'.) A
stronger property is that any Hilbert set for $K$ is nonempty. We
call such a field \textbf{s-Hilbertian}.

In case the characteristic of $K$ is zero, these two properties
coincide. If the characteristic of $K$ is positive, there is a
simple criterion for a Hilbertian field to be s-Hilbertian.

But first let us recall the following notion. Let $K$ be a field of
positive characteristic $p>0$. Then the subset of all $p$th-powers
$K^p$ is a subfield of $K$. If $[K:K^p]=1$, we call $K$
\textbf{perfect}, otherwise we say it is \textbf{imperfect}.

\begin{theorem}[Uchida]
Let $K$ be a Hilbertian field of positive characteristic. Then $K$
is s-Hilbertian if and only if $K$ is imperfect.
\end{theorem}

\begin{definition}
A field $E$ is said to be \textbf{Hilbertian over} a subset $K$ if
$H_E(f_1,\ldots,f_m;g)\cap K^r\neq \emptyset$ for any irreducible
$f_1,\ldots,f_m\in E[\bfT, X]$ that are separable in $X$ and any
nonzero $g(\bfT)\in E[\bfT]$. If furthermore
$H_E(f_1,\ldots,f_m;g)\cap K^r\neq \emptyset$ for any irreducible
$f_1,\ldots,f_m\in E[\bfT,\bfX]$, then we say that $E$ is
\textbf{s-Hilbertian over} $K$.
\end{definition}

Note that a field $K$ is Hilbertian (resp.\ s-Hilbertian) if and
only if it is Hilbertian (resp.\ s-Hilbertian) over itself.

We begin with the observation that the proof of \cite[Proposition
3.1]{JardenRazon1994} gives the following stronger statement:

\begin{theorem}[Jarden-Razon]
Let $E$ be a countable field that is Hilbertian over a subset $K$.
Then for almost all $\bfsigma\in \gal(E)^e$ the fields
$E_{s}(\bfsigma)$ and $\Egal(\bfsigma)$ are PAC over $K$.
\end{theorem}

(Recall that a field $E$ is a PAC extension of a subset $K$ if for
any $f(T,X)\in K[T,X]$ that is absolutely irreducible and separable
in $X$ there exists $(a,b)\in K\times E$ for which $f(a,b)=0$.) Next
we wish to find new PAC extensions, and we start by finding
Hilbertian fields over other fields.

\begin{lemma}\label{lem:sHilb_p_t}
Let $K$ be an s-Hilbertian field over a subset $S$ and $E/K$ a
purely transcendental extension. Then $E$ is s-Hilbertian over $S$.
\end{lemma}

\begin{proof}
Let $f_1(\bfT,\bfX),\ldots, f_r(\bfT,\bfX)\in E[\bfT,\bfX]$ be
irreducible polynomials and $0\neq g(\bfT)\in E[\bfT]$. Since $E =
K(u_\alpha \mid \alpha\in A)$, where $\{u_\alpha\mid \alpha\in A\}$
is a set of variables, we can assume that $f_i(\bfT,\bfX) =
g_i(\bfu,\bfT,\bfX)$, where
\[
g_1(\bfu,\bfT,\bfX),\ldots,g_r(\bfu,\bfT,\bfX)\in K[\bfu,\bfT,\bfX]
\]
for some finite tuple of variables $\bfu$.

Since $K$ is s-Hilbertian over $S$, there exists a tuple $\bfa$ of
elements in $S$ such that all $f_i(\bfa,\bfX) = g_i(\bfu,\bfa,\bfX)$
are irreducible in $K[\bfu,\bfX]$ and $g(\bfa)\neq 0$. But the
elements in $\{u_\alpha\mid \alpha\in A\}$ are algebraically
independent, so all $f_i(\bfa,\bfX) = g_i(\bfu,\bfa,\bfX)$ are
irreducible in the larger ring $E[\bfX]$.
\end{proof}

\begin{proposition}
Let $K$ be an s-Hilbertian field over a subset $S$ and let $E/K$ be
a finitely generated extension. Then $E$ is Hilbertian over $S$.
Moreover, if $E/K$ is also separable, then $E$ is even s-Hilbertian
over $S$.
\end{proposition}

\begin{proof}
Choose a transcendence basis $\bft$ for $E/K$, i.e., $K(\bft)/K$ is
purely transcendental and $E/K(\bft)$ is finite. Let $H\subseteq
E^r$ be a separable Hilbert set for $E$. By \cite[Proposition
12.3.3]{FriedJarden2005}, there exists a separable Hilbert set
$H_1\subseteq K(\bft)^r$ such that $H_1\subseteq H$. By
Lemma~\ref{lem:sHilb_p_t}, we get that $H_1\cap S^r\neq \emptyset$,
and hence the assertion.

If $E/K$ is also separable, then we can choose $\bft$ to be a
separating transcendence basis, that is, we can assume that
$E/K(\bft)$ is separable. Now the same argument as above work for
any Hilbert set $H\subseteq E^r$ (using \cite[Corollary
12.2.3]{FriedJarden2005} instead of \cite[Proposition
12.3.3]{FriedJarden2005}).
\end{proof}

Combining the results that we attained so far we enlarge the family
of PAC extensions:

\begin{theorem}\label{thm:almostallPACextension}
Let $e\geq 1$ be an integer, let $K$ be a countable field which is
s-Hilbertian over some subset $S$, and let $E/K$ be a finitely
generated extension. Then for almost all $\bfsigma\in \gal(E)^e$ the
fields $E_s(\bfsigma)$ and $\Egal(\bfsigma)$ are PAC over $S$.

In particular, the result is valid when $K$ is a countable
s-Hilbertian field (and $S=K$).
\end{theorem}

\begin{corollary} \label{cor:PACoverf.g.}
Let $e\geq 1$ be an integer and let $K$ be a finitely generated
infinite field (over its prime field). Then for almost all
$\bfsigma\in \gal(K)^e$ the field $K_s(\bfsigma)$ is a PAC extension
of any subfield which is not algebraic over a finite field.
Moreover, if $K$ is of characteristic~$0$, then $K_s(\bfsigma)$ is
also PAC over any subring.
\end{corollary}

\begin{proof}
First assume that $K$ is of characteristic $0$. Then any ring
contains $\bbZ$, so it suffices to show that $K_s(\bfsigma)/\bbZ$ is
PAC for almost all $\bfsigma\in \gal(K)^e$. And indeed, since $\bbQ$
is Hilbertian over $\bbZ$, Theorem~\ref{thm:almostallPACextension}
implies that $K_s(\bfsigma)/\bbZ$ is PAC for almost all $\bfsigma$.

Next assume that the characteristic of $K$ is $p>0$. Since any field
$F$ which is not algebraic over $\bbF_p$ contains a rational
function field $\bbF_p(t)$, it suffices to show that
$K_s(\bfsigma)/\bbF_p(t)$ is PAC for almost all $\bfsigma\in
\gal(K)^e$ and any $t\in K_s(\bfsigma) \smallsetminus \tilde
\bbF_p$.

Set $G=\gal(K)^e$ and let $\mu$ be its normalized Haar measure.
For any $t\in K_s \smallsetminus \tilde\bbF_p$ we define a subset
$\Sigma_{t}\subseteq G$ as follows.
\begin{equation}\label{eq:Sigma_F}
\Sigma_{t} = \{ \bfsigma\in G \mid \mbox{if } t\in K_s(\bfsigma),
\mbox{ then } K_s(\bfsigma)/\bbF_p(t) \mbox{ is PAC}\}.
\end{equation}

We claim that $\mu(\Sigma_{t})=1$. Indeed, let $E=K(t)$. Then $E/K$
is a finite separable extension. Let $H=\gal(E)^e$ be the
corresponding open subgroup of $G$.

Note that $t\in K_s(\bfsigma)$ if and only if $\bfsigma\in H$. Then
the definition of $\Sigma_t$ implies that
\[
\Sigma_t =  (H\cap \Sigma_t)\cup (G\smallsetminus H).
\]
Hence it suffices to show that $\mu(H\cap \Sigma_t)=\mu(H)$, or
equivalently, $\nu(H\cap \Sigma_t)=1$, where $\nu$ denotes the
normalized Haar measure on $H$.

Since $\bbF_p(t)$ is Hilbertian (\cite[Theorem
13.3.5]{FriedJarden2005}) and imperfect (\cite[Lemma
2.7.2]{FriedJarden2005}), Uchida's theorem implies that $\bbF_p(t)$
is s-Hilbertian. Also $E/\bbF_p(t)$ is finitely generated because
$K$ is.

Finally, since $H\cap \Sigma_t$ is the set of all $\bfsigma\in
\gal(E)^e$ for which $E_s(\bfsigma)/\bbF_p(t)$ is PAC, and since
$E_s=K_s$, Theorem~\ref{thm:almostallPACextension} implies that
$\nu(H\cap \Sigma_t)=1$, as desired.
\end{proof}

\section{An Application}
In this section we address Problem 18.7.8 of \cite{FriedJarden2005},
the so called `bottom theorem'. Let $K$ be a Hilbertian field and
$e\geq 1$ an integer. The problem asks whether for almost all
$\bfsigma$ the field $M=K_s(\bfsigma)$ has no cofinite subfield
(that is, $N\subsetneq M$ implies $[M:N]=\infty$).

Note that the Hilbertian field $K=\bbF_p(t)$ has imperfect degree
$p$, i.e., $[K:K^p]=p$. Moreover, the imperfect degree is preserved
under separable extensions (see \cite[Lemma
2.7.3]{FriedJarden2005}), and hence every separable extension $M/K$
satisfies $[M:M^p]=p$. In other words, the problem requires a small
modification.

\begin{conjecture}
Let $K$ be a Hilbertian field and $e\geq 1$ an integer. Then for
almost all $\bfsigma\in\gal(K)^e$ and for every proper subfield
$N\subsetneqq K_s(\bfsigma)$, if the extension $K_s(\bfsigma)/N$ is
separable, then it is infinite.
\end{conjecture}

In \cite{Haran1985} Haran proves an earlier version of this
conjecture, namely with the additional assumption that $K\subseteq
N$. The following result settles the conjecture for finitely
generated fields.

\begin{theorem}\label{thm:bottom}
Let $K$ be a finitely generated infinite field and $e\geq 1$ an
integer. Then for almost all $\bfsigma\in\gal(K)^e$ and for every
proper subfield $N\subsetneqq K_s(\bfsigma)$, if the extension
$K_s(\bfsigma)/N$ is separable, then it is infinite.
\end{theorem}

\begin{proof}
Clearly $N$ contains a finitely generated infinite field $F$. By
Corollary~\ref{cor:PACoverf.g.} $K_s(\bfsigma)/F$ is PAC. Hence
$K_s(\bfsigma)/N$ is PAC. The assertion now follows from
Corollary~\ref{cor: finite PAC extension}.
\end{proof}

\section{Algebraic Extensions of Hilbertian Fields}\label{sec:3}
Throughout this section let $K$ be a countable Hilbertian field and
$F/K$ an extension. Since $K$ has an abundance of algebraic PAC
extensions (Theorem~\ref{thm:almostallPACextension}) we can look on
the following family of PAC extensions of $F$.
\[
\{MF \mid M/K \mbox{ is an algebraic PAC extension}\}.
\]
In this section we discuss when this family contains non
separably-closed fields. For applications it is more important to
understand the following set of positive integers.
\[
\calN(F) = \{n\geq 1\mid \exists M/F \mbox{ PAC} \mbox{ and a
separable extension } N/M \mbox{ such that } n=[N:M]\}.
\]
Note that always $1\in \calN(F)$. Also if there exists a non
separably-closed PAC extension $M/F$, then $\calN(F)$ is infinite.
Indeed, $M$ is not a formally real field since $X^2+Y^2+1=0$ has a
solution. Therefore, by Artin-Schreier Theorem \cite[Corollary
9.3]{Lang2002}, $M$ has infinitely many separable extensions.

We start by giving a general result about the infinitude of
$\calN(F)$ in terms of the absence of free subgroups in $\gal(F/K)$.

\begin{theorem}
Assume $F/K$ is Galois and for some $e\geq 1$ no closed subgroup of
$\gal(F/K)$ is isomorphic to $\Fhat_e$. Then $\calN(F)$ is infinite.
\end{theorem}

\begin{proof}
Let $\bfsigma\in \gal(K)^e$ be an $e$-tuple such that
$M=K_s(\bfsigma)$ is a PAC extension of $K$ and
$\left<\bfsigma\right>\cong \Fhat_e$
(Theorem~\ref{thm:intAlmostallPAC}).

It suffices to show that $MF\neq K_s$. If $MF=K_s$, then
\[
\gal(F/F\cap M) \cong \gal(K_s/M)=\left<\bfsigma\right> \cong
\Fhat_e,
\]
a contradiction.
\end{proof}

In what follows we shall describe $\calN(F)$ more explicitly for
more specific extensions. We shall use the following simple group
theoretic lemma.

\begin{lemma}\label{lem:grp}
Let $N\leq N_0\leq G$ be profinite groups such that $N$ is normal in
$G$. Let $H$ be a quotient of $G$ such that $H$ and $G/N$ have no
common nontrivial quotients. Then, $H$ is a quotient of $N_0$. In
particular, if $H$ has an open subgroup of index $n$, so does $N_0$.
\end{lemma}

\begin{proof}
Let $U\lhd G$ such that $G/U= H$. Since $G/NU$ is a common quotient
of $G/U$ and $G/N$, we get that $G/NU=1$, so $G=NU$. Therefore
 also $N_0U=G$, and hence $N_0/N_0\cap U \cong N_0U/U = G/U = H$.
\end{proof}

\subsection{Pro-solvable Extensions}
Recall that a finite separable field extension is called
\textbf{solvable} if the Galois group of its Galois closure is
solvable. Then an algebraic separable field extension is called
\textbf{pro-solvable} if it is the compositum of solvable
extensions.

\begin{theorem}\label{prop:sol}
Let $F$ be a pro-solvable extension of a countable Hilbertian field
$K$ and $e\geq 2$. Then for almost all $\bfsigma\in \gal(K)^e$ the
field $M=FK_{s}(\bfsigma)$ is a PAC extension of $F$ and it has a
separable extension of every degree $>4$. In particular
\[
\{5,6,7,\ldots \} \subseteq \calN(F).
\]
\end{theorem}

\begin{proof}
Let $\hat F$ be the Galois closure of $F/K$, so $\gal(\hat F/K)$ is
pro-solvable. For almost all $\bfsigma\in \gal(K)^e$ the field
$K_{s}(\bfsigma)$ is a PAC extension of $K$ and its absolute Galois
group $G=\left< \bfsigma \right>$ is a free profinite group of rank
$e$ (Theorem~\ref{thm:intAlmostallPAC}). Fix such $\bfsigma$ and
write $M = FK_{s}(\bfsigma)$. Then $M/F$ is PAC. Let $N_0 = \gal(M)$
be the absolute Galois group of $M$ and let $N = \gal(\hat
FK_{s}(\bfsigma))$. Then $N\leq N_0\leq G$, $N$ is normal in $G$,
and $G/N = \gal(\hat F K_{s}(\bfsigma)/K_s(\bfsigma))$. The
restriction map $\gal(\Fhat K_s(\bfsigma)/K_s(\bfsigma)) \to
\gal(\hat F/K)$ is an embedding, so $G/N$ is pro-solvable.

Let $n>4$, we show that $M$ has a separable extension of degree $n$.
By the Galois correspondence, it suffices to show that $N_0$ has an
open subgroup of index $n$. As $G$ is free of rank $\geq 2$ it has
$A_n$ (the alternating group) as a quotient. Now $A_n$ and $G/N$
have no nontrivial common quotients (as $G/N$ is pro-solvable and
$A_n$ is simple). Lemma \ref{lem:grp} with $H=A_n$ implies that
$N_0$ has an open subgroup of index $n$ (since $(A_n:A_{n-1})=n$).
\end{proof}

\subsection{Prime-to-$p$ Extensions}
Let $p, m,k$ be positive integers such that $p$ is prime, $p\nmid
m$, and $p\mid \phi(m)$. Here $\phi$ is Euler's totient function.
Fix $\alpha \in (\bbZ/m\bbZ)^*$ of (multiplicative) order $p$. Then
$\bbZ/p^k \bbZ$ acts on $\bbZ/m\bbZ$ (from the left) by $ x b :=
\alpha^x b$, $x\in \bbZ/p^k \bbZ$, $b\in \bbZ/m\bbZ$.

Consider the semidirect product $H = \bbZ/m\bbZ \rtimes \bbZ/p^k
\bbZ$ of all pairs $(a,x)$, $a\in \bbZ/m\bbZ$ and $x\in \bbZ/p^k
\bbZ$ with multiplication given by
\[
(a,x)(b,y) = (a + \alpha^x b, x+y).
\]
In particular
\[
(a,1)^n =(a(1+\alpha+\alpha^2 + \cdots +\alpha^{n-1}),n) =
\left(\frac{a(1-\alpha^{n})}{1-\alpha},n\right).
\]
We embed $\bbZ/m \bbZ$ and $\bbZ/p^k \bbZ$ in $H$ in the natural
way.

\begin{lemma}
\renewcommand{\labelenumi}{\alph{enumi}.}
Let $p,m,k$, and $H = \bbZ/m\bbZ \rtimes \bbZ/p^k \bbZ$ be as above.
Then
\begin{enumerate}\label{lem:semi}
\item
$H$ is quasi-$p$ (i.e.\ it is generated by its $p$-Sylow subgroups).
\item
The only prime-to-$p$ quotient of $H$ is the trivial quotient.
\item
For $n\mid |H|=p^k m$, there exists a subgroup of $H_0$ of order
$n$.
\end{enumerate}
\end{lemma}

\begin{proof}
a.\! The elements $(0,1)$ and $(1,1)=(1,0)(0,1)$ generate $H$, so it
suffices to show that their order divides $p^k$ (and in fact equals
to $p^k$). We have $(0,1)^{p^k} = (0,p^k)=(0,0)$ and
\[
(1,1)^{p^k} = \left(\frac{1-\alpha^{p^k}}{1-\alpha},0\right)=(0,0).
\]

b.\! follows from a.: Indeed, let $\bar H = H/N$ be a quotient of
$H$ with order prime-to-$p$. Thus $p^k$ divides the order of $N$,
and hence $N$ contains a $p$-Sylow subgroup of $H$. As $N\lhd H$, it
contains all the $p$-Sylow subgroups. By a., $N= H$ and $\bar H =
1$, as desired.

To show c., assume $n\mid p^k m$. Then $n = n_0 q$, where $n_0\mid
m$ and $q\mid p^k$. Let $A$, $B$ be the unique subgroups of $\bbZ/m
\bbZ$, $\bbZ/p^k\bbZ$ of order $n_0$, $q$, respectively. Because of
the uniqueness, $A$ is $B$-invariant. Hence $A\rtimes B\leq H$.
Obviously $H_0 = A\rtimes B$ is of order $n$.
\end{proof}

Recall that a finite extension $F/K$ is \textbf{prime-to-$p$} if $p$
does not divide $[F:K]$. An infinite algebraic extension is
\textbf{prime-to-$p$} if any finite subextension is.

\begin{theorem}
Let $p$ be a prime, let $F$ be a separable extension of a countable
Hilbertian field $K$ whose Galois closure is prime-to-$p$ over $K$.
Let $e\geq 2$. Then for almost all $\bfsigma\in \gal(K)^e$ the field
$M=FK_{s}(\bfsigma)$ is a PAC extension of $F$ and it has a
separable extension of every degree. In particular, $\calN(F) =
\bbZ_+$.
\end{theorem}

\begin{proof}
Let $\hat F$ be the Galois closure of $F/K$, so every finite
quotient of $\gal(\hat F/K)$ has order prime to $p$. By
Theorem~\ref{thm:intAlmostallPAC}, for almost all $\bfsigma\in
\gal(K)^e$ the field $M_0=K_{s}(\bfsigma)$ has a free absolute
Galois group of rank $e$, namely $G=\left< \bfsigma \right>$, and
$M=M_0F$ is a PAC extension of $F$. Let $N_0 = \gal(M)$ be the
absolute Galois group of $M$ and let $N = \gal(\hat
FK_{s}(\bfsigma))$. Then $N\leq N_0\leq G$, $N$ is normal in $G$,
and $G/N = \gal(\hat F K_{s}(\bfsigma)/K_s(\bfsigma))$. The
restriction map $G/N \to \gal(\hat F/K)$ is an embedding, so every
finite quotient of $G/N$ has order prime to $p$ (because it is a
subgroup of a finite quotient of $\gal(\hat F/K)$).

By the Galois correspondence, it suffices to show that $N_0$ has
open subgroups of any index. Let $n$ be a positive integer prime to
$p$ and $k\geq 1$. Let $l$ be a prime number such that $l\nmid n$
and $p\mid l-1$ (such a prime $l$ exists since there are infinitely
many primes in the arithmetic progression $l\equiv 1\pmod p$) and
let $m=nl$. Then $p\nmid m$ and $p\mid \phi(m)$. Let $H =
\bbZ/m\bbZ\rtimes \bbZ/p^k\bbZ$ as in Lemma~\ref{lem:semi}. Then $H$
and $G/N$ have no nontrivial common quotients (by
Lemma~\ref{lem:semi}b.). By Lemma \ref{lem:grp} and
Lemma~\ref{lem:semi}c., $N_0$ has open subgroups of index $n$ and
$np^k$, i.e., of any index.
\end{proof}

% ----------------------------------------------------------------
\bibliographystyle{amsplain}
\def\cprime{$'$}
\providecommand{\bysame}{\leavevmode\hbox
to3em{\hrulefill}\thinspace}
\providecommand{\MR}{\relax\ifhmode\unskip\space\fi MR }

\providecommand{\MRhref}[2]{%
  \href{http://www.ams.org/mathscinet-getitem?mr=#1}{#2}
} \providecommand{\href}[2]{#2}

\end{document}